\documentclass[11pt]{amsart}

\allowdisplaybreaks
\usepackage[left=1in,top=1in,right=1in,bottom=1in]{geometry}
\usepackage{amsmath,amssymb,amsthm,amsfonts}
\usepackage[colorlinks,citecolor=blue, linkcolor=blue, urlcolor=magenta]{hyperref}
\usepackage{dsfont}
\usepackage{multicol}
\usepackage{paralist}
\usepackage{times}
\usepackage{enumitem}
\usepackage{epsfig}
\usepackage{graphicx}
\usepackage{pgf}
\usepackage{chemarr}
\usepackage{float}
\usepackage{wrapfig}
\usepackage{array, multirow}
\usepackage{tikz}
\usetikzlibrary{positioning, arrows.meta, shapes.geometric, calc}
\restylefloat{table}

\usepackage[linesnumbered,ruled,commentsnumbered,noline]{algorithm2e}

\usepackage{cases}
\usepackage{caption}
\usepackage{subcaption}

\usepackage{url}
\usepackage{bm}
\usepackage[square, numbers]{natbib}
\usepackage{tcolorbox}

\usepackage[mathscr]{eucal}
\usepackage{color}
\usepackage{macro}
\numberwithin{equation}{section} % Number equations by section

\title[Glycolytic Pathway]{Stochastic Averaging and Statistical Inference of Glycolytic Pathway}

\author{Arnab Ganguly and Hye-Won Kang}
\address{Department of Mathematics, Louisiana State University, USA.}
\address{Department of Mathematics, University of Maryland-Baltimore County, USA.}
\email{aganguly@lsu.edu}
\email{hwkang@umbc.edu}
\thanks{A. Ganguly is the corresponding author. Research of A. Ganguly is supported in part byNSF DMS - 2246815 and Simons Foundation (via Travel Support for Mathematicians).
The authors gratefully acknowledge the organizers of the \textit{Workshop on Chemical Reaction Network Theory} held at POSTECH (Pohang University of Science and Technology) in the Republic of Korea in July 2024, where this project was initiated.
}



\begin{document}
	
	\begin{abstract}

% {\grn(Abstract needs to be one paragraph, within 250 words; current one is 227 words.)}\\
Many biological processes exhibit oscillatory behavior. Among these, glycolytic oscillations have been extensively studied due to their well-characterized biochemical reaction networks. However, the complexity of these networks necessitates low-dimensional ordinary differential equation (ODE) models to identify core mechanisms and perform stability analysis. While previous studies proposed reduced ODE models, these were typically introduced from deterministic descriptions rather than the underlying stochastic dynamics, which more accurately represent discrete reaction events occurring at random times.
In this paper, we develop a rigorous probabilistic framework for deriving a reduced Othmer-Aldridge model of the glycolytic pathway from its stochastic formulation. The full system is modeled as a multiscale continuous-time Markov chain  with different time and abundance scales. Under an appropriate scaling regime and specific structural conditions, we prove that the dynamics of the slow components are approximated by a two-dimensional ODE. The proof is technically involved due to the network’s complexity and strong coupling between its components.
We further consider the problem of parameter estimation when observations are limited to the slow species: fructose-6-phosphate and ADP. The reduced system yields a tractable loss function depending solely on these variables. We prove that the resulting estimators are statistically consistent when the data originate from the full stochastic reaction network. Together, these results provide a mathematically rigorous framework linking stochastic biochemical reaction networks, reduced deterministic dynamics, and statistically reliable parameter estimation.

% In this work, we introduce a stochastic formulation of the glycolytic pathway proposed by Othmer and Aldridge using a continuous-time Markov chain framework. We establish a law of large numbers result and rigorously derive a corresponding reduced ODE model through a probabilistic multiscale analysis. The proof relies on stochastic averaging under carefully identified scaling regimes and structural conditions, whose verification is nontrivial due to the network's complexity.
% In addition, we address parameter inference using the reduced ODE model under partial observability, assuming that only the slow variables appearing in the reduced system are measurable. We prove statistical consistency of parameter estimates obtained by minimizing a data-driven loss function based on observations of these slow variables. Together, these results provide a mathematically rigorous framework linking stochastic biochemical reaction networks, reduced deterministic dynamics, and statistically reliable parameter estimation.

	\end{abstract}
%\date{{\cb \today, \DTMcurrenttime}}
	
	\maketitle

	\setcounter{equation}{0}
	\renewcommand {\theequation}{\arabic{section}.\arabic{equation}}
	\section{Introduction.}\label{sec:intro}

Understanding oscillatory systems is fundamental in biology, as many biological processes exhibit rhythmic behaviors across multiple scales. At the molecular level, certain gene expression dynamics induce circadian oscillations in mRNA and protein concentrations. At the subcellular level, metabolites display oscillatory behavior through alternating activation and inhibition of enzymes. On a larger scale, some microbial populations, such as yeast, exhibit collective population rhythms arising from the synchronization of individual cellular oscillations. Previous studies have revealed that nonlinear interactions and feedback mechanisms are key components required to generate and sustain such oscillations~\cite{Novak:2008:DPB,Goldbeter:2018:DSB}. Nevertheless, further research is needed to elucidate these oscillatory behaviors across complex biological systems spanning genetics, physiology, and ecology.

Among these oscillatory systems, glycolytic oscillations have been extensively studied using dynamical models because the underlying biochemical reaction networks are relatively well characterized. The classical Higgins and Selkov models~\cite{Higgins:1967:TOR,Selkov:1968:SOG}, formulated as systems of ordinary differential equations (ODEs), describe glycolytic oscillations using two variables—substrate and product—interacting through nonlinear feedback. Building on Higgins’ general framework, Othmer and Aldridge proposed ODE-based models to study oscillations and synchronization at the population level~\cite{Othmer:1978:ECD}. However, the derivation of these two-variable models was largely heuristic and lacked rigorous mathematical justification.

Deriving low-dimensional ODE models remains a powerful approach for understanding oscillatory dynamics in biological systems. Such models allow researchers to identify essential mechanisms responsible for oscillatory behavior and to perform analytical investigations using tools such as stability and bifurcation analysis. These analyses help determine the existence of limit cycles and the conditions under which oscillations arise. Moreover, simple ODE models enable systematic exploration of oscillation properties such as period and amplitude. Together, these theoretical and numerical analyses contribute to a deeper understanding of biological oscillations. However, to ensure reliability, a rigorous derivation of such reduced models is essential.

% One common method for deriving reduced ODE models from complex systems is singular perturbation analysis. By identifying fast and slow variables, one can separate time scales within the system and derive an approximate ODE model defined on a lower-dimensional manifold. This approach is known as the quasi-steady-state approximation (QSSA), in which fast variables are assumed to rapidly reach quasi-equilibrium and become functions of slow variables~\cite{Segel:1989:QSS}. However, applying QSSA to oscillatory systems can be problematic since a full model with no oscillatory behavior can produce a reduced model with oscillatory behavior and the opposite cases are also possible~\cite{Flach:2006:UAQ,Kim:2020:MMM}. This discrepancy makes model reduction using QSSA particularly challenging in oscillatory biological systems.

One common method for deriving reduced ODE models from complex biochemical systems is singular perturbation analysis. By identifying fast and slow variables, one separates time scales and derives an approximate lower-dimensional system defined on a slow manifold. This approach, often referred to as the quasi-steady-state approximation (QSSA), assumes that fast variables rapidly converge to quasi-equilibrium and can therefore be expressed as functions of the slow variables~\cite{Segel:1989:QSS}. While this methodology has been widely successful, its application to oscillatory systems is delicate. In particular, it is known that a full deterministic model without oscillations may yield a reduced system exhibiting oscillations, and conversely, oscillatory behavior in the full model may disappear after reduction~\cite{Flach:2006:UAQ,Kim:2020:MMM}. Since oscillations depend sensitively on nonlinear feedback structure and global phase-space geometry, time-scale reduction at the deterministic level can alter essential dynamical features.

Moreover, deterministic ODE descriptions themselves arise as approximations of underlying stochastic reaction networks. Biochemical systems evolve through discrete reaction events occurring randomly in time, and their natural mathematical representation is a continuous-time Markov chain (CTMC) model ~\cite{Gillespie:1977:ESS, Anderson:2011:CTM}. Deterministic mass-action equations emerge from such models under suitable scaling limits via law-of-large-numbers (LLN) arguments \cite{Kurtz72}. This observation suggests that model reduction can be performed at a more fundamental level by starting from the stochastic formulation rather than from  a
deterministic system that is already an approximation.

% {\grn
% Moreover, discrepancies between reduced ODE models and the true dynamics of oscillatory biological systems may be exacerbated when stochastic effects are neglected. Reduced models are often derived from deterministic ODE descriptions of full systems under the assumption that biochemical reaction networks evolve continuously. In contrast, biochemical systems are inherently discrete, as molecular populations change in integer-valued increments, and stochastic, as reaction events occur randomly in time~\cite{Gillespie:1977:ESS}. These stochastic fluctuations can play a significant role in oscillatory dynamics, particularly in systems with nonlinear interactions, where noise may alter phase, amplitude, or even the existence of oscillations. This observation motivates the derivation of reduced models directly from stochastic formulations rather than from deterministic approximations.

% }

Motivated by this perspective, this paper takes a probabilistic approach. The primary contribution is twofold. First, we establish a LLN \Cref{thm:FLLN-Glyco} showing that, under an appropriate multiscale scaling regime, the CTMC formulation of the glycolytic reaction network converges to a two-dimensional ODE model describing the evolution of the slow species. In contrast to singular perturbation or QSSA-based reductions, where one begins with a deterministic ODE system and derives a lower-dimensional approximation through formal time-scale separation, our derivation proceeds at the level of the stochastic reaction network itself. The reduced ODE does not arise from a perturbative expansion of a pre-existing deterministic model; rather, it emerges as a scaling limit in probability of the underlying jump process, in the same spirit as other LLNs in which deterministic mean-field equations arise as limits of stochastic particle systems in statistical physics. This approach provides a rigorous probabilistic foundation for the reduced dynamics and clarifies precisely under what scaling assumptions the low-dimensional ODE accurately captures the macroscopic behavior of the biochemical system.
The rigorous analysis yields structural insight into the multiscale structure of the glycolytic network, and specifies which combinations of microscopic parameters govern the effective macroscopic dynamics. 

Second, we address statistical inference for the glycolytic pathway when observations are available only for the slow species fructose-6-phosphate (F6P) and ADP (denoted by $A_1$ and $A_2$). In the full stochastic reaction network, the high dimensionality, strong nonlinear coupling, and presence of unobserved fast species render direct likelihood-based or trajectory-based inference computationally prohibitive and statistically ill-posed. The reduced-order model, by contrast, yields a closed and low-dimensional dynamical system expressed solely in terms of the observed slow variables, thereby providing a tractable and well-defined loss function for parameter estimation.
Our second main result, \Cref{th:consistency}, establishes statistical consistency of the estimators obtained by minimizing this loss function. Importantly, the observational data are assumed to be generated by the original multiscale CTMC dynamics rather than by the limiting ODE, reflecting the fact that the latter is an approximation  rather than  the true data-generating mechanism. The LLN result, which rigorously connects the microscopic CTMC to the reduced ODE under a specified multiscale regime, plays a crucial role in the proof. While this scaling limit provides the essential link between the microscopic and reduced dynamics, establishing consistency of the estimators for the effective parameters of the reduced model requires a substantially deeper argument. This result provides a mathematically justified framework for data-driven inference using the reduced model, which would otherwise be infeasible.

% Second, we address the problem of statistical inference for the glycolytic pathway using data available only on the slow species fructose-6-phosphate (F6P) and ADP (denoted by $A_1$ and $A_2$ on the paper). Clearly, direct inference from the full stochastic model (or even the full ODE system) is computationally and statistically infeasible, whereas the reduced-order model naturally yields a tractable loss function based solely on trajectories of the slow variables. Our second main result, \Cref{th:consistency}, establishes that the resulting estimators are statistically consistent (i.e. asymptotically unbiased). This mathematically justifies the utility and accuracy of the reduced system for data-driven inference. 

Our LLN-result falls within the realm of stochastic averaging.
At a conceptual level, the overall strategy for proving stochastic averaging is well-known and can be summarized in a couple of steps. One first typically proves tightness of the slow component in a suitable function space (e.g., $C([0,T], \R^d)$ for continuous processes or $D([0,T],\R^d)$ for c\`adl\`ag processes) and tightness of the occupation measures associated with the fast component, viewed as measure-valued random variables. Tightness ensures the existence of limit points for the joint process. The final step is to identify these limit points and prove uniqueness for the limiting slow dynamics, thereby characterizing the reduced model.

This program has been implemented for generic continuous It\^o diffusion processes under ideal conditions, such as Lipschitz drift and diffusion coefficients and often weak coupling between slow and fast components \cite{Khas68, PaVe01, PaVe02, FrWe08, FrWe12}. It has been formalized in \cite{Kang:2013:STM} for equations arising from reaction systems (see also \cite{Crudu2012AAP} for simpler classes of reaction systems). These works primarily adopt a generator-based approach building on \cite{Kurtz:1992:AMP}, which provides a general framework for stochastic averaging of certain martingale problems. However, such results serve as a structural blueprint, rather than a turnkey solution: the principal mathematical challenge of course lies in executing the averaging strategy  for specific models. 
% In some settings, convergence of generators and semigroups can be handled via functional analytic techniques~\cite{Ball:2006:AAM, Crudu2012AAP}, but such approaches are often difficult to implement for complex reaction networks.

In the present case, the CTMC model of the glycolytic pathway is highly intricate (see \Cref{fig:full_gly}), involving ten species and sixteen reactions that form a strongly coupled network of jump processes. The reaction rates span four distinct scales, $O(n^{-1})$, $O(1)$, $O(n^{1/2})$, and $O(n)$, while species abundances occur on three scales, $O(1)$, $O(n)$, and $O(n^2)$. Here $n$ denotes a scaling parameter that captures differences in species abundances and reaction speeds. Moreover, both slow and fast reactions influence the dynamics of the slow and fast components (see \eqref{eq:scaling-regime}, \eqref{eq:full_gly_scaled_eq}). These features places our system outside the scope of standard stochastic averaging frameworks, and even if one were content with an abstract characterization of the limit --- without seeking an explicit closed-form expression --- the hypotheses of generic stochastic averaging theorems are not satisfied, and hence such results cannot be invoked to justify convergence.

Thus, implementing the broad steps outlined above for this system is a delicate task and requires a careful examination of the specific network structure underlying the glycolytic pathway. In particular, establishing tightness of the occupation measures for the fast subsystem, together with controlling certain integrals involving these measures, is challenging. Both are required at various points in the proofs and are complicated by the rapid fluctuations of AMP (denoted by $A_3$). Addressing these obstacles requires non-standard techniques, including delicate martingale estimates that exploit the specific structure of the network (see \Cref{prop:rel_compactness}).
 In contrast to functional analytic techniques based on convergence of relevant generators and semigroups \cite{Kurtz:1992:AMP, Ball:2006:AAM, Crudu2012AAP}, which are often difficult to employ for complex models like ours, our approach is probabilistic in nature relying on representation of the species-processes as solutions of stochastic differential equations (SDEs) driven by Poisson random measures (PRMs). In our opinion, this in fact provides more transparent insights into species dynamics and the interplay between fast and slow reactions. Such SDE-based representations were previously used to mathematically justify popular total Quasi-Steady State Approximations (tQSSA) for Michaelis-Menten kinetics in a recent work of the first author, \cite{GaKh26}.

A further difficulty arises at the identification stage. In generic stochastic averaging theorems, the limiting slow equation is characterized in abstract form: the drift of the slow variable is expressed as an average of certain coefficients with respect to an invariant distribution of the fast process, which is assumed to be unique. Such an abstract characterization of the reduced model is clearly inadequate for practical purposes, as it does not provide  a tractable model for numerical simulation and parameter estimation. To be usable, the limiting dynamics must be available in explicit closed form.

But, to this end, we note that the abstract averaging framework is not even  applicable in our setting. The reason is that, in our case, the equilibrium distribution of the fast subsystem, which satisfies a complex PDE,  is neither available in closed form nor guaranteed to be unique. As a result, one cannot simply apply a general stochastic averaging theorem to justify the reduced equation. Instead, a more detailed analysis of the network structure is required to establish the relevant ergodic properties and to demonstrate that, despite these obstacles, the limiting equation for the slow species is both unique and explicitly computable, thereby  yielding a concrete and computationally tractable reduced-order model suitable for both simulation and parameter estimation.
Unlike earlier heuristic reductions of the complex glycolytic pathway relying on quasi-steady-state or partial equilibrium assumptions, our derivation provides a mathematically rigorous foundation of the Othmer–Aldridge model that not only retains the key oscillatory kinetics but also mathematically justifies its efficacy in statistical inference from data on the slow species.

The rest of the paper is organized as follows. In \Cref{sec:math-frame}, we introduce the stochastic model of the glycolytic pathway and formulate its pathwise representation as a system of SDEs driven by Poisson random measures. The main stochastic averaging result is established in \Cref{sec:stochavg}. We also discuss the resulting reduced-order system and present numerical simulations comparing it with the full CTMC model. Section~\Cref{sec:para} is devoted to parameter estimation based on the reduced model. Our primary result in this section is the consistency of the proposed estimators. Numerical experiments validating accuracy of the estimators are also presented.

% Beyond capturing oscillatory behavior, we show that our reduced ODE model facilitates efficient parameter estimation and helps address identifiability issues. When time-series data for some or all species are available, ODE models can support parameter inference, as their reduced dimensionality typically involves fewer unknown parameters~\cite{Pantea:2013:QCK,Snowden:2017:MMR}. Assuming that only two slowly varying species are observable in the glycolytic pathway, we generate synthetic time-series data and perform parameter estimation. Our model provides accurate estimates for all reduced ODE model parameters, demonstrating its practical utility in data-driven inference. Since experimental observations often capture only slow species, a reduced ODE model with fewer parameters is particularly advantageous. Therefore, our approach not only clarifies the mathematical foundation of the Othmer–Aldridge model but also offers a general framework for analyzing and parameterizing other biochemical systems exhibiting oscillatory dynamics.

	   \subsection{Notational conventions}

    \textbullet\ The space of continuous functions from $E$ to $F$ will be denoted by $C(E, F)$ with the subset $C_b(E, F), C_c(E, F)$ denoting  the bounded, continuous and compactly supported continuous functions, respectively. These spaces are equipped with their usual topologies. For $T>0$, the space $C([0,T], F)$ is a subset of $D([0,T], F)$, the space of  c\`adl\`ag functions from $[0,T]$ to $F$. $D([0,T], F)$ will be equipped with the Skorokhod topology. \textbullet\ The $\sigma$-field generated by the Borel subsets of $E$ will be denoted by $\cal{B}(E)$. \textbullet\  $\cal{M}(E)$ will denote the space of finite (non-negative) measures on $E$ equipped with the topology of weak convergence. For $r>0$, $\cal{M}_r(E) \subset \cal{M}(E)$ will denote the space of (non-negative) measures $\nu$ such that $\nu(E) =r$. %\textbullet\ The set of non-negative integers, real numbers, non-negative real numbers will be denoted by $\setOfNaturals, \setOfNonnegativeIntegers, \setOfReals$, and $\setOfPositiveReals$ respectively. 
    \textbullet\ The indicator function of a set $A$ will be denoted by $\indic_{A}(\cdot)$, i.e., $\indic_{A}(x) = 1$ if $x\in A$, and zero otherwise. \textbullet\ We use $\leb$ to denote the Lebesgue measure on $\R$. \textbullet\ For a c\`adl\`ag function $f$, we denote the left-hand limit of the function $f$ at $t$ by $f(t-)$.   %For a differentiable function $f:\R^d \rt \R$, the function $\partial_i f$ will denote the first-order derivative of $f$ with respect to the $i$-th coordinate. If the coordinate is clear from the context, we will simply use $\partial f$. The $k$-th order derivative with respect to the $i$-th coordinate (when it exists) will be denoted by $\partial^k_i f$. \textbullet\ The notation $a \wedge b$ denotes the minimum of $a$ and $b$. Similarly, $a \vee  b$ denotes the maximum of $a$ and $b$. 
   \textbullet\ For $u \in \R^d$ or $\Z^d$, $\|u\|_1 = \sum_{i=1}^d |u_i|$.    
    \textbullet\ Other notations will be introduced when needed.
	
	\section{Mathematical Framework}\label{sec:math-frame}
A wide variety of simplified glycolytic pathways have been studied to understand the oscillatory behavior observed in metabolic intermediates of glycolysis~\cite{Higgins:1967:TOR,Selkov:1968:SOG,Othmer:1978:ECD}. Among these, we focus on the enzyme-catalyzed reaction mechanism introduced by Othmer and Aldridge~\cite{Othmer:1978:ECD}. 

This mechanism consists of ten chemical species and sixteen reactions. We denote by $A_1$ and $A_2$ fructose-6-phosphate (F6P) and ADP, respectively; $A_3$ and $A_4$ represent AMP and ATP. The variables $E_1$ and $E_1^*$ correspond to the low-activity and activated forms of phosphofructokinase-1 (PFK). The complexes formed between PFK abd F6P are denoted by $E_1A_1$ and $E_1^*A_1$. In addition, $E_2$ denotes an enzyme responsible for ADP degradation, and $E_2A_2$ denotes its complex with ADP. 

The full chemical reaction network is is presented in \Cref{fig:full_gly}.

    % Reaction system {\crd AG: We should introduce this in the introduction}

	% \begin{equation}
	% 	\begin{split}
	% 		& \emptyset \xrightarrow{\nkappa_0} A_1, \\
	% 		& E_1+A_1 \xrightleftharpoons[\nkappa_{-1}]{\nkappa_1} E_1A_1  \xrightarrow{\nkappa_2}  E_1+A_2, \\ 
	% 		& E_1^*+A_1 \xrightleftharpoons[\nkappa_{-3}]{\nkappa_3} E_1^*A_1  \xrightarrow{\nkappa_4}  E_1^*+A_2, \\ 
	% 		& E_2+A_2 \xrightleftharpoons[\nkappa_{-5}]{\nkappa_5} E_2A_2  \xrightarrow{\nkappa_6}  E_2+\mbox{Product}, \\ 
	% 		& E_1 +A_3 \xrightleftharpoons[\nkappa_{-7}]{\nkappa_7} E_1^*, \qquad\quad
	% 		E_1A_1 +A_3 \xrightleftharpoons[\nkappa_{-7}]{\nkappa_7} E_1^*A_1, \\
	% 		& 2A_2 \xrightleftharpoons[\nkappa_{-8}]{\nkappa_8} A_3+A_4.
	% 	\end{split}
	% 	\label{eq:full_gly}
	% \end{equation}
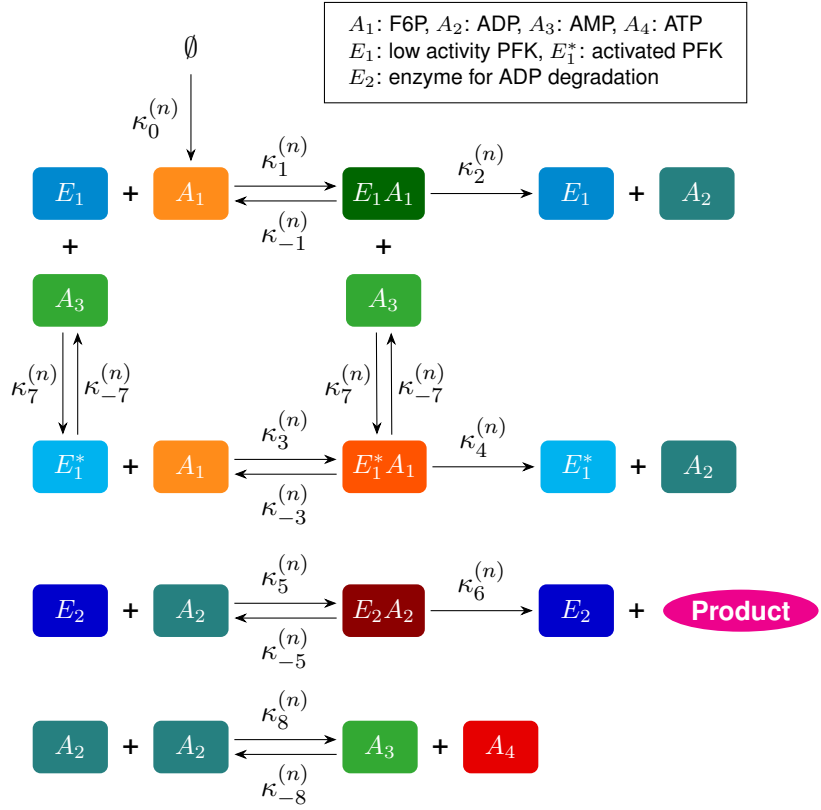
\begin{figure}

\begin{tikzpicture}[scale=1.0, transform shape,
    % Define the styles for the species boxes
    species/.style={
        draw=none, 
        rectangle, 
        rounded corners=3pt, 
        minimum width=1.0cm, 
        minimum height=0.7cm, 
        font=\sffamily\small\bfseries, 
        text=white
    },
    % Color definitions
    E1_style/.style={species, fill=cyan!70!blue},
    E1star_style/.style={species, fill=cyan!90},
    E2_style/.style={species, fill=blue!80!black}, 
    A1_style/.style={species, fill=orange!90},
    A2_style/.style={species, fill=teal!70!gray},
    % A2_style/.style={species, fill=green!60!gray!80},
    A3_style/.style={species, fill=green!58!black!80},
    A4_style/.style={species, fill=red!90!black},
    E1A_style/.style={species, fill=green!40!black},
    E1starA_style/.style={species, fill=orange!65!red},
    E2A_style/.style={species, fill=red!55!black},
    product_style/.style={
    ellipse, 
    fill=magenta, 
    font=\sffamily\small\bfseries, 
    text=white, 
    inner sep=2pt, 
    minimum width=1.2cm
    },
    % Arrow styles with gaps
    rev_fwd/.style={transform canvas={yshift=2.8pt}, shorten >=2pt, shorten <=2pt, -{Stealth[scale=1.0]}},
    rev_back/.style={transform canvas={yshift=-2.8pt}, shorten >=2pt, shorten <=2pt, {Stealth[scale=1.0]}-},
    forward/.style={shorten >=2pt, shorten <=2pt, -{Stealth[scale=1.0]}},
    up_down_left/.style={transform canvas={xshift=-2.8pt}, shorten >=2pt, shorten <=2pt, {Stealth[scale=1.0]}-},
    up_down_right/.style={transform canvas={xshift=2.8pt}, shorten >=2pt, shorten <=2pt, -{Stealth[scale=1.0]}}
]

% --- Legend Box ---
\node[draw, rectangle, thin, inner sep=4pt, font=\sffamily\scriptsize, anchor=south east] at (9, 3.2) {
    \begin{tabular}{l}
        $A_1$: F6P, $A_2$: ADP, $A_3$: AMP, $A_4$: ATP \\
        $E_1$: low activity PFK, $E_1^*$: activated PFK \\
        $E_2$: enzyme for ADP degradation
    \end{tabular}
};

% --- ROW 1 ---
\node[E1_style] (E1_r1) at (0, 2) {$E_1$};
\node[right=0.05cm of E1_r1, font=\bfseries] (p1) {+};
\node[A1_style, right=0.05cm of p1] (A1_r1) {$A_1$};

% Move empty set slightly higher to make the arrow longer
\node[above=1.3cm of A1_r1, font=\normalsize] (empty) {$\emptyset$};

% Draw the arrow with label on the left
\draw[forward] 
    (empty) -- (A1_r1) node[midway, left, xshift=-0.1pt] {$\nkappa_{0}$};
% \draw[forward] (empty) -- (A1_r1);
% \node[left=0.3cm] at ($(empty)!0.5!(A1_r1)$) {$\nkappa_0$}; % <-- left of arrow

\node[E1A_style, right=1.5cm of A1_r1] (E1A) {$E_1A_1$};
\node[E1_style, right=1.5cm of E1A] (E1_r1_end) {$E_1$};
\node[right=0.05cm of E1_r1_end, font=\bfseries] (p2) {+};
\node[A2_style, right=0.05cm of p2] (A2_r1) {$A_2$};

\draw[rev_fwd] (A1_r1) -- (E1A) node[midway, above] {$\nkappa_1$};
\draw[rev_back] (A1_r1) -- (E1A) node[midway, below] {$\nkappa_{-1}$};
\draw[forward] (E1A) -- (E1_r1_end) node[midway, above] {$\nkappa_2$};

% --- ROW 2 ---
\node[below=0.25cm of E1_r1, font=\bfseries, inner sep=0] (p3) {+};
\node[A3_style, below=0.25cm of p3] (A3_left) {$A_3$};

\node[below=0.25cm of E1A, font=\bfseries, inner sep=0] (p4) {+};
\node[A3_style, below=0.25cm of p4] (A3_mid) {$A_3$};

% --- ROW 3 ---
\node[E1star_style, below=1.5cm of A3_left] (E1s_r3) {$E_1^*$};
\node[right=0.05cm of E1s_r3, font=\bfseries] (p5) {+};
\node[A1_style, right=0.05cm of p5] (A1_r3) {$A_1$};

\node[E1starA_style, right=1.5cm of A1_r3] (E1sA) {$E_1^*A_1$};
\node[E1star_style, right=1.5cm of E1sA] (E1s_r3_end) {$E_1^*$};
\node[right=0.05cm of E1s_r3_end, font=\bfseries] (p6) {+};
\node[A2_style, right=0.05cm of p6] (A2_r3) {$A_2$};

\draw[transform canvas={xshift=-2.8pt}, shorten >=2pt, shorten <=2pt, -{Stealth[scale=1.0]}] 
   (A3_left) -- (E1s_r3) node[midway, right, xshift=4pt] {$\nkappa_{-7}$};

\draw[transform canvas={xshift=2.8pt}, shorten >=2pt, shorten <=2pt, -{Stealth[scale=1.0]}] 
  (E1s_r3) -- (A3_left) node[midway, left, xshift=-3pt] {$\nkappa_{7}$};

\draw[transform canvas={xshift=2.8pt}, shorten >=2pt, shorten <=2pt, -{Stealth[scale=1.0]}] 
  (E1sA) -- (A3_mid) node[midway, left, xshift=-3pt] {$\nkappa_{7}$}; 

\draw[transform canvas={xshift=-2.8pt}, shorten >=2pt, shorten <=2pt, -{Stealth[scale=1.0]}] 
   (A3_mid) -- (E1sA) node[midway, right, xshift=4pt] {$\nkappa_{-7}$};
% \draw[up_down_right] (A3_mid) -- (E1sA) node[midway, right] {$\nkappa_7$};

\draw[rev_fwd] (A1_r3) -- (E1sA) node[midway, above] {$\nkappa_3$};
\draw[rev_back] (A1_r3) -- (E1sA) node[midway, below] {$\nkappa_{-3}$};
\draw[forward] (E1sA) -- (E1s_r3_end) node[midway, above] {$\nkappa_4$};

% --- ROW 4 ---
\node[E2_style, below=1.2cm of E1s_r3] (E2_r4) {$E_2$};
\node[right=0.05cm of E2_r4, font=\bfseries] (p7) {+};
\node[A2_style, right=0.05cm of p7] (A2_r4) {$A_2$};

\node[E2A_style, right=1.5cm of A2_r4] (E2A) {$E_2A_2$};
\node[E2_style, right=1.5cm of E2A] (E2_r4_end) {$E_2$};
\node[right=0.05cm of E2_r4_end, font=\bfseries] (p8) {+};
\node[product_style, right=0.1cm of p8] (prod) {Product};
% \node[right=0.05cm of p8, font=\sffamily\small, text=black] (prod) {Product};

\draw[rev_fwd] (A2_r4) -- (E2A) node[midway, above] {$\nkappa_5$};
\draw[rev_back] (A2_r4) -- (E2A) node[midway, below] {$\nkappa_{-5}$};
\draw[forward] (E2A) -- (E2_r4_end) node[midway, above] {$\nkappa_6$};

% --- ROW 5 ---
\node[A2_style, below=1.1cm of E2_r4] (A2_r5_1) {$A_2$};
\node[right=0.05cm of A2_r5_1, font=\bfseries] (p9) {+};
\node[A2_style, right=0.05cm of p9] (A2_r5_2) {$A_2$};

\node[A3_style, right=1.5cm of A2_r5_2] (A3_r5) {$A_3$};
\node[right=0.05cm of A3_r5, font=\bfseries] (p10) {+};
\node[A4_style, right=0.05cm of p10] (A4_r5) {$A_4$};

\draw[rev_fwd] (A2_r5_2) -- (A3_r5) node[midway, above] {$\nkappa_8$};
\draw[rev_back] (A2_r5_2) -- (A3_r5) node[midway, below] {$\nkappa_{-8}$};

\end{tikzpicture}
\caption{Glycolytic Pathway}
\label{fig:full_gly}
\end{figure}

This mechanism provides a simplified description of the glycolytic pathway, focusing on the rate-limiting phosphorylation of F6P ($A_1$) by ATP, which produces fructose-1,6-bisphosphate and ADP ($A_2$). PFK catalyzes this step and exhibits either low ($E_1$) or high ($E_1^*$) activity depending on the cellular AMP concentration. ADP ($A_2$) participates in downstream reactions of glycolysis, modeled here as an ADP-degradation process catalyzed by $E_2$. AMP ($A_3$) binds to both free and substrate-bound low-activity PFKs ($E_1$ and $E_1A_1$), converting them into their activated forms ($E_1^*$ and $E_1^*A_1$). Moreover, AMP and ATP interconvert with ADP through a reversible reaction.

In their study, Othmer and Aldridge~\cite{Othmer:1978:ECD} derived a two-variable reduction of the full mechanism under the following assumptions: (i) the enzyme-substrate complexes $E_1A_1$, $E_1^*A_1$, and $E_2A_2$ satisfy a quasi-steady-state approximations; (ii) the activation reactions involving $E_1$ and $E_1A_1$ are in partial equilibrium; and (iii) the interconversion of $A_4$ is also in partial equilibrium. The resulting reduced dynamics for the concentrations of $A_1$ and $A_2$ take the form
\begin{equation}
\begin{split}
\frac{dx}{dt} &= k - \bar{f}(x,y), \quad
\frac{dy}{dt} = \bar{f}(x,y) - \bar{g}(y)
\end{split}
\label{eq:det_odes}
\end{equation} 
where $x$ and $y$ denote the concentrations of F6P and ADP, respectively. 
However, the derivation of this reduced model was presented only heuristically, and conditions under which the approximation captures the oscillatory behavior of the full reaction network have not been rigorouly analyzed.

The first part of the paper will focus on how to rigorously derive a reduced-order ODE model of the form \eqref{eq:det_odes} starting from a Markov chain formulation of the reaction system in \Cref{fig:full_gly} in a suitable scaling regime.

% In the remaining of this paper, we introduce a stochastic description of the glycolytic mechanism based on a continuous-time Markov chain and provide a rigorous derivation of the reduced ODE system analogous to~\eqref{eq:det_odes}. 

\subsection{Stochastic Description}
For each $n \geq 1$, viewed as a scaling parameter that encodes differences in species abundance and reaction speeds, we denote by $\nX$ the stochastic process representing the species vector,
$$\nX = \left(\nX_{A_1}, \nX_{A_1}, \nX_{A_3}, \nX_{A_4}, \nX_{E_1}, \nX_{E^*_1}, \nX_{E_1A_1}, \nX_{E^*_1A_1}, \nX_{E_2}, \nX_{E_2A_2}\right).$$
Denote the reaction index set as
\begin{eqnarray*}
%\mathcal{S}_I &=& \{A_1,A_2,A_3,A_4,E_1,E_1^*,E_1A_1,E_1^*A_1,E_2,E_2A_2\},\\
\mathcal{R} &=& \{0,\pm1,2,\pm3,4,\pm5,6,\pm7,\pm7',\pm8\}.
\end{eqnarray*}
For each $k \in \mathcal{R}$, define an increasing pure jump process $\nR_k$ by
\begin{align} \label{eq:react-num}
\dst \nR_{k}(t) = \int_{[0,\infty)\times[0,t]} \indic_{[0, \nlambda_k(\nX(s-))]}(u)\xi_k(du\times ds),
\end{align}
where the $\xi_k$ are Poisson random measures (PRMs) on $[0,\infty)\times[0,\infty)$ with mean measure as product of Lebesgue measures, $\leb\ot\leb$.
$\nR_k(t)$ is the number of occurrences of the $k$-th reaction with propensity functions $\nlambda_k$ in $[0,t]$. The propensity functions $\nlambda_k$ are defined as follows: $\nlambda_0(x) \equiv \nkappa_0$, 
and 
\begin{align}\label{eq:prop-X}
\begin{alignedat}{4}
\nlambda_1(x)   &= \nkappa_1 x_{E_1}x_{A_1}, &\quad 
\nlambda_{-1}(x) &= \nkappa_{-1} x_{E_1A_1}, &\quad 
\nlambda_{2}(x)  &= \nkappa_2 x_{E_1A_1}, \\
\nlambda_{3}(x)  &= \nkappa_3 x_{E_1^*}x_{A_1}, &\quad 
\nlambda_{-3}(x) &= \nkappa_{-3} x_{E_1^*A_1}, &\quad 
\nlambda_{4}(x)  &= \nkappa_4 x_{E_1^*A_1}, \\ 
\nlambda_{5}(x)  &= \nkappa_5 x_{E_2}x_{A_2}, &\quad 
\nlambda_{-5}(x) &= \nkappa_{-5} x_{E_2A_2}, &\quad 
\nlambda_{6}(x) &= \nkappa_{6} x_{E_2A_2}, \\
\nlambda_{7}(x)  &= \nkappa_7 x_{E_1}x_{A_3}, &\quad 
\nlambda_{-7}(x) &= \nkappa_{-7} x_{E_1^*}, &\quad 
\nlambda_{7'}(x) &= \nkappa_7 x_{E_1A_1}x_{A_3}, \\ 
\nlambda_{-7'}(x) &= \nkappa_{-7} x_{E_1^*A_1}, &\quad 
\nlambda_{8}(x)  &= \nkappa_8 x_{A_2}\!\left(x_{A_2}-1\right), &\quad 
\nlambda_{-8}(x) &= \nkappa_{-8} x_{A_3}x_{A_4}. & & 
\end{alignedat}
\end{align}
Here, for convenience of tracking, a typical state $x \in \Z_{\geq 0}^{10}$ will be denoted as
$$x=(x_{A_1}, x_{A_2}, x_{A_3}, x_{A_4},x_{E_1}, x_{E_1^*}, x_{E_1A_1}, x_{E_1^*A_1}, x_{E_2}, x_{E_2A_2})$$
instead of the more typical $x=(x_1,x_2,\hdots, x_{10}).$

It is now clear from the reaction system in \Cref{fig:full_gly} that the trajectories of $\nX$ are given by 
\begin{align}
	\begin{aligned}
		\nX_{A_1}(t) &= \nX_{A_1}(0) +  \nR_0(t) - \nR_1(t) + \nR_{-1}(t) - \nR_3(t) + \nR_{-3}(t), \\
		\nX_{A_2}(t) &= \nX_{A_2}(0) + \nR_2(t) + \nR_4(t) - \nR_5(t) + \nR_{-5}(t) - 2\nR_8(t) + 2\nR_{-8}(t), \\
		\nX_{A_3}(t) &= \nX_{A_3}(0) - \nR_7(t) + \nR_{-7}(t) - \nR_{7'}(t) +\nR_{-7'}(t) + \nR_8(t) - \nR_{-8}(t), \\
		\nX_{A_4}(t) &= \nX_{A_4}(0) + \nR_8(t) - \nR_{-8}(t) , \\
		\nX_{E_1}(t) &= \nX_{E_1}(0) - \nR_1(t) + \nR_{-1}(t) + \nR_2(t) - \nR_7(t) +\nR_{-7}(t), \\ 
		\nX_{E_1^*}(t) &= \nX_{E_1^*}(0) - \nR_3(t) + \nR_{-3}(t) + \nR_4(t) + \nR_7(t) - \nR_{-7}(t), \\ 
		\nX_{E_1A_1}(t) &= \nX_{E_1A_1}(0) + \nR_1(t) - \nR_{-1}(t) - \nR_2(t) - \nR_{7'}(t) + \nR_{-7'}(t), \\ 
		\nX_{E_1^*A_1}(t) &=\nX_{E_1^*A_1}(0) +\nR_3(t) - \nR_{-3}(t) - \nR_4(t) + \nR_{7'}(t) -\nR_{-7'}(t), \\ 
        \nX_{E_2}(t) &= \nX_{E_2}(0) - \nR_5(t) + \nR_{-5}(t) + \nR_6(t), \\
		\nX_{E_2A_2}(t)&= \nX_{E_2A_2}(0) + \nR_5(t) - \nR_{-5}(t) -\nR_6(t),
	\end{aligned}
	\label{eq:full_gly_unscaled_eq}
\end{align}
which constitutes a system of SDEs driven by Poisson random measures (PRM) written in integral form.

In order to study the averaging phenomena, we consider the $[0,\infty)^{10}$-valued scaled process vector $\nZ$  defined by 
\begin{align} \label{eq:species-scaling}
\nZ_{i}(t) =& n^{-\alpha_{i}} \nX_{i}(t), \quad i \in \{A_1, A_2, A_3, A_4, E_1, E_1^*, E_1A_1, E_1^*A_1, E_2, E_2A_2\}.
\end{align}
The $\alpha_i \in \R$ are scaling exponents that describe variation in species abundance. We also introduce the scaling exponents $\{\beta_k\}$ for reactions rates capturing  variation in reaction speeds:
\begin{align}\label{eq:reaction-rate-scaling}
\kappa_k = n^{-\beta_k} \kappa^{(n)}_k, \quad k =0, \pm 1, 2, \pm 3, 4, \pm 5, 6, \pm 7, \pm 8.
\end{align}
As before, for convenience of tracking, a typical state $z \in [0,\infty)^{10}$ of the process $\nZ$ will be denoted as
$$z=(z_{A_1}, z_{A_2}, z_{A_3}, z_{A_4},z_{E_1}, z_{E_1^*}, z_{E_1A_1}, z_{E_1^*A_1}, z_{E_2}, z_{E_2A_2})$$
instead of the more typical $z=(z_1,z_2,\hdots, z_{10}).$

In this paper we operate in the following scaling regime:
\begin{align} \label{eq:scaling-regime}
\begin{aligned}
\alpha_{A_4}&=2, \quad \alpha_{A_1} = \alpha_{A_2} = 1, \quad \alpha_{A_3} = \alpha_{E_1} = \alpha_{E^*_1} = \alpha_{E_1A_1}= \alpha_{E_1^*A_1}= \alpha_{E_2}= \alpha_{E_2A_2} = 0\\
\beta_0 &= \beta_{-1}= \beta_{2}= \beta_{-3}= \beta_{4}= \beta_{-5}= \beta_{6}=1, \quad  \beta_{7}= \beta_{-7} = \frac{1}{2}, \quad \beta_{1}= \beta_{3}= \beta_{5}=0, \\
 \beta_{8}&= \beta_{-8} = -1.
\end{aligned}
\end{align}
With the above choice of the scaling parameters, the scaled process $\nZ$ satisfies the following system of equations: 
\begin{equation}
	\begin{split}
		\nZ_{A_1}(t) &= \nZ_{A_1}(0) + n^{-1}\left[ \nR_0(t) - \nR_1(t) + \nR_{-1}(t) - \nR_3(t) + \nR_{-3}(t)\right], \\
		\nZ_{A_2}(t) &= \nZ_{A_2}(0) + n^{-1}\left[ \nR_2(t) + \nR_4(t) - \nR_5(t) + \nR_{-5}(t) - 2\nR_8(t) + 2\nR_{-8}(t)\right], \\
		\nZ_{A_3}(t) &= \nZ_{A_3}(0) - \nR_7(t) + \nR_{-7}(t) - \nR_{7'}(t) +\nR_{-7'}(t) + \nR_8(t) - \nR_{-8}(t), \\
		\nZ_{A_4}(t) &= \nZ_{A_4}(0) + n^{-2}\left[ \nR_8(t) - \nR_{-8}(t) \right], \\
		\nZ_{E_1}(t) &= \nZ_{E_1}(0) - \nR_1(t) + \nR_{-1}(t) + \nR_2(t) - \nR_7(t) +\nR_{-7}(t), \\ 
		\nZ_{E_1^*}(t) &= \nZ_{E_1^*}(0) - \nR_3(t) + \nR_{-3}(t) + \nR_4(t) + \nR_7(t) - \nR_{-7}(t), \\ 
		\nZ_{E_1A_1}(t) &= \nZ_{E_1A_1}(0) + \nR_1(t) - \nR_{-1}(t) - \nR_2(t) - \nR_{7'}(t) + \nR_{-7'}(t), \\ 
		\nZ_{E_1^*A_1}(t) &=\nZ_{E_1^*A_1}(0) +\nR_3(t) - \nR_{-3}(t) - \nR_4(t) + \nR_{7'}(t) -\nR_{-7'}(t), \\ 
        \nZ_{E_2}(t) &= \nZ_{E_2}(0) - \nR_5(t) + \nR_{-5}(t) + \nR_6(t), \\
		\nZ_{E_2A_2}(t)&= \nZ_{E_2A_2}(0) + \nR_5(t) - \nR_{-5}(t) -\nR_6(t).
	\end{split}
	\label{eq:full_gly_scaled_eq}
\end{equation}
{\em Conservation Laws:} Notice that the following conservation laws hold: for all $t\geq 0$,
\begin{align}
        %\begin{aligned}
          &\nZ_{E_1}(t)+\nZ_{E_1A_1}(t)+\nZ_{E_1^*}(t)+\nZ_{E_1^*A_1}(t) \equiv\ \nconsconst_1, \qquad \nZ_{E_2}(t) + \nZ_{E_2A_2}(t)  \equiv\ \nconsconst_2. 
        %\end{aligned}
        \label{eq:cons_laws}
    \end{align}
Denote the species index-sets $S$ and $F$ by 
\begin{align}
    \label{eq:species-set}
  S=\{A_1,A_2,A_4\}, \qquad   F=\{A_3, E_1, E_1^*, E_1A_1, E_1^*A_1, E_2, E_2A_2\},
\end{align}
and write $\nZ = (\nZ_S, \nZ_F)$, where
\begin{align}\label{eq:slow-fast}
\begin{aligned}
\nZ_S \dfeq&\ \left(\nZ_{A_1},\nZ_{A_2}, \nZ_{A_4}\right), \qquad
\nZ_F \dfeq\ \left(\nZ_{A_3}, \nZ_{E_1}, \nZ_{E_1^*}, \nZ_{E_1A_1}, \nZ_{E_1^*A_1}, \nZ_{E_2}, \nZ_{E_2A_2}\right).
\end{aligned}
\end{align}
When necessary, a typical state $z \in [0,\infty)^{10}$ of $\nZ = (\nZ_S, \nZ_F)$ will be split as $z=(z_S,z_F)$ with $z_S=(z_{A_1}, z_{A_2}, z_{A_4}) \in [0,\infty)^3$ and $z_F=(z_{A_3},z_{E_1}, z_{E_1^*}, z_{E_1A_1}, z_{E_1^*A_1}, z_{E_2}, z_{E_2A_2}) \in [0,\infty)^7.$

In the scaling regime defined by \eqref{eq:scaling-regime}, $\nZ_F$ will be the fast process and $\nZ_S$ will be the slow process. Our primary goal is to show that as $\nZ_S \nrt Z_S$ in $D([0,T], [0,\infty)^3)$ where $Z_S$ is the solution of a possibly random ODE. We present this stochastic averaging result in the next section.

\section{Stochastic Averaging}\label{sec:stochavg}
We start with the necessary assumptions for the main convergence result to hold.

\begin{assumption} \label{assum:main} 
The following conditions hold.
\begin{enumerate}[label=(\alph*), ref=(\alph*)]
\item \label{item:mmt-bd-initial} $ \sup_n\EE\|\nZ(0)\|_1 < \infty$;
\item \label{item:cons-const} $\sup_n \EE\lf[(\nconsconst_1)^2\ri] \vee  \sup_n \EE\lf[(\nconsconst_2)^2\ri]  < \infty.$
\item \label{item:initial-other} 
\begin{itemize}
	\item for some  random variables $\consconst_1, \consconst_2, Z_{A_1}(0), Z_{A_2}(0)$ and $Z_{A_4}(0),$ as $n\rt \infty,$
	$$ \big(\nconsconst_1,\nconsconst_2, \nZ_{A_1}(0), \nZ_{A_2}(0), \nZ_{A_4}(0)\big) \prt  \big(\consconst_1, \consconst_2, Z_{A_1}(0), Z_{A_2}(0), Z_{A_4}(0)\big);$$ 
	\item $Z_{A_4}(0) > \delta_0$ a.s. for some $\delta_0>0$.
	\item $\lf\{n^{-1}\big(\nZ_{A_3}(0)\big)^{\aexp}\ri\}$ is tight for some $\aexp > 2$.
\end{itemize}

\end{enumerate}
\end{assumption}
% {\crd AG: The last assumption in \ref{item:initial-other} is a bit interesting. If we assume \ref{item:mmt-bd-initial}, which implies $\{\nZ_{A_3}(0)\}$ is tight, then that obviously implies $n^{-1}\big(\nZ_{A_3}(0)\big)^{\aexp} \prt 0$ (and hence it is tight). But I am thinking whether to preclude $\nZ_{A_3}(0)$ from \ref{item:mmt-bd-initial} and only  assume $\lf\{n^{-1}\big(\nZ_{A_3}(0)\big)^{\aexp}\ri\}$ is tight.}

The first step to prove the convergence of the slow process $\nZ_S$ is to establish its tightness in $D([0,T],  [0,\infty)^3)$. To analyze  the rapid movement of the fast process $\nZ_F$, we introduce its measure  $\nocc_F$ as
\begin{align*}
	\nocc_{F}(\SC{A}\times[0,t]) = \int_0^t \indic_{\SC{A}}(\nZ_F(s))\ \diff{s}, \quad \SC{A} \subset [0,\infty)^{7}.
\end{align*}
Notice that $\nocc_F$ is a random measure taking values in $\spmeas_T([0,\infty)^{7}\times[0,T])$ (see definition in  {\em Notational conventions} in Introduction)	and $\nocc_F([0,\infty)^7\times [0,t]) =t$ for any $t \in [0,T]$. The following proposition establishes the necessary tightness of $(\nocc_F,\nZ_S)$.

 \begin{proposition} \label{prop:rel_compactness} 
	Suppose that \Cref{assum:main} holds. Then for any $T>0$, the sequence $(\nocc_F,\nZ_S)$ is relatively compact as $\spmeas_T([0,\infty)^{7}\times[0,T])\times D([0, T], [0,\infty)^3)$-valued random variables. Furthermore, the limit points of $\nZ_S$ are almost surely in $C([0, T],  [0,\infty)^3)$. 
\end{proposition}

We are now ready to state our main convergence result.

\begin{theorem}	\label{thm:FLLN-Glyco}
 Suppose that \Cref{assum:main} holds. Then, as $n \rt \infty$, $\nZ_S \RT Z_S = (Z_{A_1}, Z_{A_2}, Z_{A_4})$, where the path-space of of $(Z_{A_1}, Z_{A_2})$ is $C([0,T], [0,\infty)^2)$,  for $\PP$-a.a. $\om \in \Omega$,   $Z_{A_4}(t, \om) \equiv\ Z_{A_4}(0)$ and  $(Z_{A_1}(\cdot, \om), Z_{A_2}(\cdot, \om))$ solves the (random) ODE
  \begin{align}\label{eq:FLLN_Glyco}
		\begin{aligned}
	   \f{\diff Z_{A_1}(t)}{\diff t} =&\  \kappa_0 - f\left(Z_{A_1}(t),Z_{A_2}(t), Z_{A_4}(0)\right)\\
	   \f{\diff Z_{A_2}(t)}{\diff t}=&\  f\left(Z_{A_1}(t),Z_{A_2}(t), Z_{A_4}(0)\right) 
	   -g\left(Z_{A_2}(t)\right)\\
	   Z_{A_4}(t) \equiv&\ Z_{A_4}(0)
	   \end{aligned}
	\end{align}
	with initial condition $Z_S(0, \om) = (Z_{A_1}(0,\om), Z_{A_2}(0,\om) , Z_{A_4}(0,\om))$ at time zero. Here
	\begin{align}
	\label{eq:FLLN-ODE-fun}
	\begin{aligned}
	f(z_{A_1},z_{A_2}, z_{A_4}) &=\frac{1}{K_1z_{A_4}+z_{A_2}^2}\left[\frac{\consconst^{\bullet}_1 K_1z_{A_1}z_{A_4}}{K_{M_1}+z_{A_1}}+\frac{\consconst^\star_1 z_{A_1}z_{A_2}^2}{K_{M_1}^\star+z_{A_1}}\right], \\
	g(z_{A_2}) &= \frac{\consconst^{\bullet}_2 z_{A_2}}{K_{M_2}+z_{A_2}},
	\end{aligned}
\end{align}	
and 
\begin{align}\label{eq:rates}
	\begin{alignedat}{4}
	 K_1 & \equiv \frac{\kappa_{-7}\kappa_{-8} }{\kappa_7 \kappa_8},& \qquad
	 K_{M_1}& \equiv \frac{\kappa_{-1}+\kappa_{2}}{k_1},& \qquad
	 K_{M_1}^\star & \equiv \frac{\kappa_{-3}+\kappa_{4}}{\kappa_3},& \qquad
	 K_{M_2} & \equiv \frac{\kappa_{-5}+\kappa_{6}}{\kappa_5}\\
	 \consconst^{\bullet}_1& \equiv \kappa_2\consconst_1, & \qquad	
	 \consconst^\star_1 & \equiv \kappa_4\consconst_1, & \qquad	
	 \consconst^{\bullet}_2& \equiv \kappa_6\consconst_2. & &
	\end{alignedat}
\end{align}

 In particular, if $Z_S(0), \consconst_1$ and $\consconst_2$ are deterministic (non-random), then $Z_S$ is deterministic, and hence as $n\rt \infty$, $\nZ_S \prt Z_S$ in $C([0,T],[0,\infty))$, that is, 
	\begin{align*}
		\sup_{t\leq T}\|\nZ_S(t)-Z_S(t)\| \prt 0.
	\end{align*}
	\end{theorem}	

We now toward the proofs of these results, which are technically involved and require a number of delicate estimates relying on a careful analysis of the reaction network structure.

To this end, we first observe that the reaction-number process $\nR_k$ defined by \eqref{eq:react-num}, and its `centered version' 
$$\tnR_k(\cdot) \dfeq \int_{[0,\infty)\times[0,\cdot]} \indic_{[0, \nlambda_k(\nX(s-))]}(u)\tilde\xi_k(du\times ds) = \nR_k(\cdot)-\int_0^\cdot \nlambda_k(\nX(s))\ ds$$
can  be written in terms of the $\nZ$ process as follows:
\begin{equation}\label{eq:react-process-1}
	\begin{aligned}
		&  \Big(\nR_{k}(t), \tnR_{k}(t)\Big) \\
		&= 
		\begin{cases}
			\lf(\xi_0\big(n\kappa_0 t\big),\tilde\xi_0\big(n\kappa_0 t\big)\ri), 
			& k=0, \\[6pt]
			\displaystyle 
			\int_{[0,\infty)\times[0,t]} \indic_{[0, n \l_k(\nZ(s-))]}(u)\,
			\lf(\xi_k(du\times ds),\ \tilde\xi_k(du\times ds)\ri), 
			& k= \pm 1, 2, \pm 3, 4,  \pm 5, 6, \pm 8, \\
			\displaystyle 
			\int_{[0,\infty)\times[0,t]} \indic_{[0, n^{1/2}\l_k(\nZ(s-))]}(u)\,
			\lf(\xi_k(du\times ds),\ \tilde\xi_k(du\times ds)\ri), 
			& k=\pm7,\pm7',
		\end{cases}
	\end{aligned}
\end{equation}
where recall that for each $k \in \mathcal{R}$,  $\xi_k$ is  the PRM on $[0,\infty)\times[0,\infty)$ with mean measure $\leb\ot\leb$, and $\tilde{\xi}_k$ defined by $\tilde{\xi}_k(A\times [0,t]) - \leb(A)t$ is the corresponding compensated PRM.

%
%\begin{align}\label{eq:react-process-1}
%	\begin{aligned}
%\lf(\nR_{k}(t), \tnR_{k}(t)\ri) = \begin{cases}
%\lf(\xi_0\left(n \kappa_0 t\right), \tilde\xi_0\left(n \kappa_0 t\right)\ri), & \  k=0\\
%\displaystyle \int_{[0,\infty)\times[0,t]} \indic_{[0, n \l_k(\nZ(s-))]}(u) \lf(\xi_k(du\times ds), \tilde\xi_k(du\times ds)\ri), & \   k=\pm 1, 2, \pm 3, 4,\\
%& \hs{1.2cm} \pm 5, 6, \pm 8\\
%\displaystyle \int_{[0,\infty)\times[0,t]} \indic_{[0, n^{1/2} \l_k(\nZ(s-))]}(u) \xi_k(du\times ds), \lf(\xi_k(du\times ds), \tilde\xi_k(du\times ds)\ri), & \   k=\pm 7, \pm 7'.
%\end{cases}  
%\end{aligned}
%\end{align}
The $\l_k$ appearing in \eqref{eq:react-process-1} represent `scaled' propensities, and are defined as follows: $\l_0(z) \equiv \kappa_0$ and 
\begin{align}\label{eq:scaled-prop}
\begin{alignedat}{4}
\l_1(z)   &= \kappa_1 z_{E_1}z_{A_1}, &\quad 
\l_{-1}(z) &= \kappa_{-1} z_{E_1A_1}, &\quad 
\l_{2}(z)  &= \kappa_2 z_{E_1A_1}, \\
\l_{3}(z)  &= \kappa_3 z_{E_1^*}z_{A_1}, &\quad 
\l_{-3}(z) &= \kappa_{-3} z_{E_1^*A_1}, &\quad 
\l_{4}(z)  &= \kappa_4 z_{E_1^*A_1}, \\ 
\l_{5}(z)  &= \kappa_5 z_{E_2}z_{A_2}, &\quad 
\l_{-5}(z) &= \kappa_{-5} z_{E_2A_2}, &\quad 
\l_{6}(z) &= \kappa_{6} z_{E_2A_2}, \\
\l_{7}(z)  &= \kappa_7 z_{E_1}z_{A_3}, &\quad 
\l_{-7}(z) &= \kappa_{-7} z_{E_1^*}, &\quad 
\l_{7'}(z) &= \kappa_7 z_{E_1A_1}z_{A_3}, \\ 
\l_{-7'}(z) &= \kappa_{-7} z_{E_1^*A_1}, &\quad 
\l_{8}(z)  &= \kappa_8 z_{A_2}\!\left(z_{A_2}-n^{-1}\right), &\quad 
\l_{-8}(z) &= \kappa_{-8} z_{A_3}z_{A_4}. & & 
\end{alignedat}
\end{align}
For convenience, we introduce the following notation:
\begin{align}
	\label{eq:mean_scaled_reaction-0}
		\nmeanreact_k(t) = \int_0^t \l_k(\nZ(s))\ ds, \quad k=0, \pm 1, 2, \pm 3, 4, \pm 5, 6, \pm 8, \pm 7, \pm 7'.
\end{align}
Clearly,
\begin{align}\label{eq:react-split}
\nR_k = 
\begin{cases}
	\tnR_k + n \,\nmeanreact_k, & k = 0, \pm 1, 2, \pm 3, 4, \pm 5, 6, \pm 8,\\[2mm]
	\tnR_k + n^{1/2} \,\nmeanreact_k, & k = \pm 7, \pm 7',
\end{cases}
\end{align}
and since each $\tnR_k$ is a martingale, 
\begin{align}
\begin{aligned}
\label{eq:exp-react-sq}
\EE\lf[\nR_k(t)\ri]  =&\
\begin{cases}
n \EE\lf[\nmeanreact_k(t)\ri], & \quad k=\pm 1, 2, \pm 3, 4, \pm 5, 6, \pm 8\\
n^{1/2} \EE\lf[\nmeanreact_k(t)\ri], & \quad k=\pm 7, \pm 7'.
\end{cases}
\\[1em]
\EE\lf[(\nR_k(t))^2\ri] = &\ 
\begin{cases}
\EE\lf[\lf(n \nmeanreact_k(t)\ri)^2 + n \nmeanreact_k(t)\ri], & \quad k=\pm 1, 2, \pm 3, 4, \pm 5, 6, \pm 8\\
\EE\lf[\lf(n^{1/2} \nmeanreact_k(t)\ri)^2 + n^{1/2} \nmeanreact_k(t)\ri], & \quad k=\pm 7, \pm 7'.
\end{cases}
\end{aligned}
\end{align}

To prove \Cref{prop:rel_compactness}, we begin with the following lemma, which provides the necessary moment estimates and establishes $C$-tightness (see \Cref{def:C-tight}) for some of the scaled reaction-count processes.
\begin{lemma}\label{lem:aux-tight}
Under  \Cref{assum:main}, the following assertions hold:
\begin{enumerate}[label=(\roman*), ref=(\roman*)]
\item \label{item:react-tight-1} For $k = 0, -1, 2, -3, 4, -5, 6, -7, -7'$, $\sup_n \EE\lf[(\nmeanreact_k(T))^2\ri]<\infty$ and the sequences of  processes $\{\nmeanreact_k\}$ are tight in $C([0,T],[0,\infty))$.  For $k = 0, -1, 2, -3, 4, -5, 6$, the sequences of  processes  $\{n^{-1} \nR_k\}$, and for $k = -7, -7'$,  $\{n^{-1/2} \nR_k\}$ are $C$-tight in $D([0,T], [0,\infty))$. 

\item \label{item:A1-A2-tight} $\sup_n\EE\lf[\sup_{t\leq T}\lf(\nZ_{A_1}(t)\ri)^2\ri]\ \vee\ \sup_n\EE\lf[\sup_{t\leq T}\lf(\nZ_{A_2}(t)\ri)^2\ri]  < \infty$. Furthermore, \\ $\sup_n\EE\lf[\sup_{t\leq T}\lf(\nZ_{A_3}(t)/n\ri)^2\ri] <\infty$.

\item \label{item:react-tight-2} For $k=1,3,5,8$, $\sup_n \EE\lf[\nmeanreact_k(T)\ri]<\infty$, the sequences of processes $\lf\{\nmeanreact_k\ri\}$ are tight in $ C([0,T], [0,\infty))$, and $\lf\{n^{-1} \nR_k\ri\}$ are $C$-tight in $D([0,T], [0,\infty))$.

\item \label{item:react-tight-7-7'} The sequence of processes  $\lf\{n^{-1/2}(\nR_{7}+\nR_{7'})\ri\}$ is $C$-tight in $D([0,T], [0,\infty))$. Furthermore,
\begin{align}
\label{eq:mmt-react-bd-7-7'-8}
\sup_n\EE\lf[\Big(\nmeanreact_{7}(T)\Big)^2\ri] \vee \sup_n\EE\lf[\Big(\nmeanreact_{7'}(T)\Big)^2\ri]< \infty, \quad \sup_n\EE\lf[\nmeanreact_{-8}(T)\ri] < \infty.
\end{align}
%In particular, for every fixed $t>0$, the sequences of random variables, $ \{n^{-1/2}\nR_{7}(t)\}, \{n^{-1/2}\nR_{7'}(t)\}, \{\nmeanreact_7(t)\}, \{\nmeanreact_{7'}(t)\}$ are tight. Furthermore, for each $t>0$, $\{n^{-1}\nR_{-8}(t)\}$ and $\{\nmeanreact_{-8}(t)\}$ are tight.

\item  \label{item:conv-A4} $\EE\lf(\sup_{t\leq T}|\nZ_{A_4}(t) - \nZ_{A_4}(0)| \ri)\rt 0,$ and $\sup_{t\leq T}|\nZ_{A_4}(t) - Z_{A_4}(0)| \prt 0$ as $n\rt \infty$

%\item \label{item:A3-int-tight} For any $t>0$, $\{\int_0^t \nZ_{A_3}(s)\ ds\}$ is tight.
\end{enumerate}
\end{lemma}

\begin{proof}
\ref{item:react-tight-1} We show tightness of $\nmeanreact_2$ and $\{n^{-1} \nR_2\}$. The tightness of other processes follow similarly. Since $\sup_{s\leq T}\nZ_{E_1A_1}(s) \leq \nconsconst_1$, and $\sup_n \EE\lf[(\nconsconst_1)^2\ri] < \infty$, it follows that 
$$\sup_n \EE\lf[(\nmeanreact_2(T))^2\ri] \leq \sup_n \EE\lf[(\nconsconst_1)^2\ri] \kappa_2^2 T^2 <\infty.$$
In particular, $\{\sup_{t\leq T}\nmeanreact_2(t) \equiv \nmeanreact_2(T)  \}$ is tight, and  the tightness of the process $\{ \nmeanreact_2\}$ in $C([0,T], \R)$  follows from \Cref{lem:int-eq-tight}. Next write (c.f. \eqref{eq:react-split})
\begin{align*}
n^{-1}\nR_2(t) = n^{-1} \tnR_2(t) + \nmeanreact_2(t)
\end{align*}
where recall that $\tnR_2$, defined by $\tnR_2(t) =\nR_2(t) -n\nmeanreact_2(t) = \int_{[0,\infty)\times[0,t]} \indic_{[0,   n \kappa_2 Z_{E_1A_1}^N(s)]}(u) \tilde\xi_2(du\times ds) $, is a zero-mean martingale. The $C$-tightness of $\{n^{-1}\nR_2\}$ will follow once 
we show that $ n^{-1} \tnR_2 \prt 0$ in $D([0,T],\R)$. In fact we show the following stronger statement: as $n \rt \infty$, $ n^{-1} \sup_{t\leq T}|\tnR_2(t)| \rt 0$ in $L^2(\Om, \PP)$.
To this end, observe that $\<n^{-1}\tnR_2\>$, the predictable quadratic variation of $n^{-1}\tnR_2$, satisfies
 $$\<n^{-1}\tnR_2\>_t = n^{-2} \int_0^t n \kappa_2 \nZ_{E_1A_1}(s)\,ds  \leq n^{-1} \nconsconst_1 T.$$
 By Burkholder-Davis-Gundy (BDG) inequality and the assumption on $\{\nconsconst_1\}$, it follows that $\<n^{-1}\tnR_2\>_T  \rt 0$ in $L^2(\Om, \PP)$  as $n\rt \infty$, 
   \begin{align}
            \label{eq:conv-pred-quadvar-R2}
  \EE\lf[\sup_{t\leq T}|  n^{-1}\tnR_2(t)|^2\ri] \leq \cnst_2 \EE\<n^{-1}\tnR_2\>_T \nrt 0.
   \end{align} 
   where $\cnst_2$ is the Burkholder's constant for exponent 2.

%\begin{align}\label{eq:mart-conv-R2}
%n^{-1} \sup_{t\leq T}|\tnR_2(t)| \prt 0.
%\end{align}
%To this end, observe that $\<n^{-1}\tnR_2\>$, the predictable quadratic variation of $n^{-1}\tnR_2$, satisfies
% $$\<n^{-1}\tnR_2\>_t = n^{-2} \int_0^t n \kappa_2 \nZ_{E_1A_1}(s)\,ds  \leq n^{-1} \nconsconst_1 T.$$
%By the assumption on $\{\nconsconst_1\}$, it follows that $\<n^{-1}\tnR_2\>_T  \rt 0$ in $L^2(\Om, \PP)$  as $n\rt \infty$, 
%        \begin{align}
%            \label{eq:conv-pred-quadvar-R2}
%            \<n^{-1}\tnR_2\>_T \prt 0.
%        \end{align} 
%  {\crd This can be covered by the general theorem}  Next for any positive $\vep, \eta$, the  Lenglart--Rebolledo inequality (see \cite[Lemma 3.7]{Whitt2007MCLT}, and \cite[Remark 4.17]{karatzas1991brownian}) gives us 
%        \begin{align}
%            \PP\lf[\sup_{t\leq T}|n^{-1}\tnR_2(t)|>\eta\ri]\leq \vep+\PP\lf[\<n^{-1}\tnR_2\>_T > \vep\eta^2\ri]. 
%            \label{eq:Lenglart-Rebolledo}
%        \end{align}
%        Because of the convergence in \eqref{eq:conv-pred-quadvar-R2}, letting $n\rt \infty$ in the \eqref{eq:Lenglart-Rebolledo} we get 
%        \begin{align*}
%              \lim_{n\rt \infty}   \PP\lf[\sup_{t\leq T}|n^{-1}\tnR_2(t)|>\eta\ri]\leq \vep,
%        \end{align*}
%        which establishes \eqref{eq:mart-conv-R2}. 
%    {\crd Shall we spin the tightness of $n^{-1}\nR_2$ as a general theorem?}
 
 \np
\ref{item:A1-A2-tight}     Notice that   
\begin{align} \label{eq:A1-ineq}
	0\leq \nZ_{A_1}(t) & \leq \nZ_{A_1}(0) + n^{-1}\left[ \nR_0(t)  + \nR_{-1}(t) + \nR_{-3}(t)\right],
\end{align}   
and
\begin{align} \label{eq:A2-A3-ineq}
\begin{aligned}
	\nZ_{A_2}(t)+2n^{-1}\nZ_{A_3}(t) =&\ \nZ_{A_2}(0)+2n^{-1}\nZ_{A_3}(0) +n^{-1}\Big[ \nR_2(t) + \nR_4(t) - \nR_5(t) + \nR_{-5}(t) \\ 
	&\hs{.2cm} - 2\nR_8(t)+ 2\nR_{-8}(t)\Big] +2n^{-1}\Big[- \nR_7(t) + \nR_{-7}(t) - \nR_{7'}(t) + \nR_{-7'}(t) \\
	&\hs{.2cm} + \nR_8(t) - \nR_{-8}(t)\Big]\\
	=&\ \nZ_{A_2}(0)+2n^{-1}\nZ_{A_3}(0) +n^{-1}\Big[ \nR_2(t) + \nR_4(t) - \nR_5(t) + \nR_{-5}(t) \\
	& \hs{.2cm} 2\big(- \nR_7(t) + \nR_{-7}(t) - \nR_{7'}(t) + \nR_{-7'}(t) \big)\Big]\\
	\leq &\ \nZ_{A_2}(0)+2n^{-1}\nZ_{A_3}(0) +n^{-1}\Big[ \nR_2(t) + \nR_4(t) +\nR_{-5}(t)\\
	&\hs{.2cm} +2\big( \nR_{-7}(t)  + \nR_{-7'}(t)\big)\Big].
\end{aligned} 
\end{align}  
The assertion now follows by squaring both sides in \eqref{eq:A1-ineq} and \eqref{eq:A2-A3-ineq}, taking expectations, noting that $\sup_{t\le T}\nR_k(t)=\nR_k(T)$ (since $\nR_k$ is nondecreasing), and then applying \eqref{eq:exp-react-sq} together with \ref{item:react-tight-1}.

%Again, as a special case of \ref{item:react-tight-1}, for $k \in \{2, 4, -5, -7, -7'\}$, the sequences $\{n^{-1} \nR_k(T)\}$ are tight in $[0,\infty]$, which proves that  $\{\sup_{t\leq T} \nZ_{A_2}(t)\}$ is tight.

\np
\ref{item:react-tight-2} The proof follows from \ref{item:react-tight-1}, \ref{item:A1-A2-tight} and \Cref{lem:int-eq-tight}.

\np 
\ref{item:react-tight-7-7'} Notice that 
\begin{align}\label{eq:react7+7'}
	\begin{aligned}
		n^{-1/2}(\nR_{7}(t)+\nR_{7'}(t)) =&\ n^{-1/2}\lf(\nZ_{E_1}(0)+ \nZ_{E_1A_1}(0) -\nZ_{E_1}(t)- \nZ_{E_1A_1}(t)\ri)\\
		& \hs{.2cm}+n^{-1/2}(\nR_{-7}(t)+\nR_{-7'}(t)).
	\end{aligned}
\end{align}
Now since $\nconsconst_1$ is tight,
\begin{align*}
	\sup_{t\leq T}n^{-1/2}\lf|\nZ_{E_1}(0)+ \nZ_{E_1A_1}(0) -\nZ_{E_1}(t)- \nZ_{E_1A_1}(t)\ri| \leq 2 n^{-1/2}\nconsconst_1 \prt 0
\end{align*}
as $n \rt \infty$. Since we already proved $n^{-1/2}\nR_{-7}$ and $n^{-1/2}\nR_{-7'}$ are $C$-tight, it follows that $n^{-1/2}(\nR_{7}+\nR_{7'})$ is $C$-tight in $D([0,T], [0,\infty))$.

Now squaring both sides of \eqref{eq:react7+7'}, taking expectations and using \eqref{eq:exp-react-sq}, we have
\begin{align}\label{eq:mean-react7+7'}
	\begin{aligned}
		\EE\Big[\Big(\nmeanreact_{7}(T)\Big)^2+\Big(\nmeanreact_{7'}(T)\Big)^2\Big] \leq &\ 4\Big(n^{-1}\EE\Big[(\nconsconst_1)^2\Big]
		+\EE\Big[\Big(\nmeanreact_{-7}(T)\Big)^2+n^{-1/2}\nmeanreact_{-7}(T)\Big]\\
		&\ +  \EE\Big[\Big(\nmeanreact_{-7'}(T)\Big)^2+n^{-1/2}\nmeanreact_{-7'}(T)\Big]\Big), % \leq 2n^{-1/2}\EE(\nconsconst_1) + t\EE(\nconsconst_1),
	\end{aligned}
\end{align}
which establishes the first inequality in \eqref{eq:mmt-react-bd-7-7'-8}, because of \ref{item:react-tight-1} and \Cref{assum:main}-\ref{item:mmt-bd-initial}. 
Next, notice that from the equation of $\nZ_{A_3}$ in \eqref{eq:full_gly_scaled_eq},
\begin{align*}
\nR_{-8}(t) \leq \nZ_{A_3}(0)+ \nR_{-7}(t) +\nR_{-7'}(t) + \nR_8(t).
\end{align*}
Taking expectation and dividing by $n$ we get
\begin{align*}
\EE(\nmeanreact_{-8}(T)) \leq &\  n^{-1}\EE(\nZ_{A_3}(0))+n^{-1/2} \EE(\nmeanreact_{-7}(T)) + n^{-1/2} \EE(\nmeanreact_{-7'}(T))+ \EE(\nmeanreact_8(T)).
\end{align*}
It now follows from \ref{item:react-tight-1}, \ref{item:react-tight-2} and \Cref{assum:main}-\ref{item:mmt-bd-initial} that
the second inequality in \eqref{eq:mmt-react-bd-7-7'-8} holds.

\np
\ref{item:conv-A4} 
%From the equation of $\nZ_{A_4}$ we get for $p \geq 1$
%\begin{align*}
%|\nZ_{A_4}(t) - \nZ_{A_4}(0)|^p \leq & \ 4^{p-1}\lf[n^{-2p}\lf((\tnR_8(t))^p +(\tnR_{-8}(t))^p\ri) + n^{-p}\lf(\nmeanreact_{8}(t))^p+(\nmeanreact_{-8}(t))^p\ri)\ri]
%\end{align*}
%It now follows by BDG inequality
%\begin{align*}
%\EE\lf(\sup_{t\leq T}|\nZ_{A_4}(t) - \nZ_{A_4}(0)|^p\ri) \leq & \ 4^{p-1}\Big[n^{-2p+p/2}\lf(\EE(\nmeanreact_8(t))^{p/2} +\EE(\nmeanreact_{-8}(t))^{p/2}\ri) \\
%& \hs{.2cm}+ n^{-p}\lf(\EE(\nmeanreact_{8}(t))^p+\EE(\nmeanreact_{-8}(t))^p\ri)\Big]\\
%\nrt &\ 0.
%\end{align*} 
%
From the equation of $\nZ_{A_4}$ in \eqref{eq:full_gly_scaled_eq}, and using \ref{item:react-tight-2} and \ref{item:react-tight-7-7'}, we obtain
\begin{align*}
\EE\lf(\sup_{t\leq T}|\nZ_{A_4}(t) - \nZ_{A_4}(0)|\ri) \leq  n^{-1} \EE\lf(\nmeanreact_{8}(T)+\nmeanreact_{-8}(T)\ri) \nrt 0.
\end{align*}
The second convergence is immediate from \Cref{assum:main}:\ref{item:initial-other}.
%{\grn In (v), we need to show for $\EE\lf(\sup_{t\leq T}|\nZ_{A_4}(t) - \nZ_{A_4}(0)|^p\ri) \nrt 0$ for $p>2$?}

\end{proof}

\begin{lemma}\label{lem:A3-int-tight}
Under  \Cref{assum:main}, the sequence of random variables $\dst \lf\{\int_0^T \nZ_{A_3}(s)\ ds\ri\}$ is tight.
\end{lemma}

\begin{proof}
To prove the assertion, we need to show that for any $\ep>0$, there exists a $K_1\equiv K_1(\vep)$ and $N_0 = N_0(\vep)$  such that for all $n\geq N_0(\vep)$
\begin{align} \label{eq:tight-int-A3}
\SC{P}^n(K_1(\ep)) \dfeq  	\PP\lf(\int_0^T  \nZ_{A_3}(s) ds \geq K_1(\vep)\ri) \leq \ep.
\end{align}
Since 
$\sup_n\EE\lf(\int_0^T \nZ_{A_3}(s)\nZ_{A_4}(s)\ ds \ri) < \infty$ by  \eqref{eq:mmt-react-bd-7-7'-8},
% the sequence $\lf\{\int_0^T \nZ_{A_3}(s)\nZ_{A_4}(s)\ ds\ri\}$ is tight. 
by Markov inequality, choose $K_1(\ep)$ such that
\begin{align}\label{eq:K1-choice}
\PP \lf( \f{2}{\delta_0}\int_0^T  \nZ_{A_3}(s)\nZ_{A_4}(s)\ ds \geq K_1(\ep)\ri)  \leq \ep/3,
\end{align}
where $\delta_0$ is as in \Cref{assum:main}-\ref{item:initial-other}. 
 Now by \Cref{lem:aux-tight}:\ref{item:conv-A4} and the assumption $\nZ_{A_4}(0) \prt Z_{A_4}(0)$, let $N_0$ be such that for all $n \geq N_0$
\begin{align*}
	\PP\lf(  \sup_{s\leq T}|\nZ_{A_4}(s) - \nZ_{A_4}(0)| > \delta_0/4\ri) \vee \PP\lf( |\nZ_{A_4}(0) - Z_{A_4}(0)| > \delta_0/4\ri) \leq \ep/3.
\end{align*}
Therefore, for all $n \geq N_0$,
\begin{align*}
& \SC{P}^n(K_1(\vep)) \leq\ \PP\lf(\int_0^T  \nZ_{A_3}(s) ds \geq K_1(\ep),\ \sup_{s\leq T}|\nZ_{A_4}(s) - \nZ_{A_4}(0)| \leq \delta_0/4\ri)+\ep/3\\
&\leq	\PP\lf(\int_0^T  \nZ_{A_3}(s) ds \geq K_1(\ep),\  \sup_{s\leq T}|\nZ_{A_4}(s) - \nZ_{A_4}(0)| \leq \delta_0/4,  |\nZ_{A_4}(0) - Z_{A_4}(0)| \leq \delta_0/4\ri)+2\ep/3
\end{align*}
Now writing $\int_0^T  \nZ_{A_3}(s) ds = \int_0^T \nZ_{A_3}(s)\nZ_{A_4}(s) /\nZ_{A_4}(s) \ ds$, we see that  on the event 
$$\lf\{ \sup_{s\leq T}|\nZ_{A_4}(s) - \nZ_{A_4}(0)| \leq \delta_0/4, \  |\nZ_{A_4}(0) - Z_{A_4}(0)| \leq \delta_0/4\ri\}$$
we have
\begin{align*}
	\int_0^T  \nZ_{A_3}(s) ds \ \leq\ \f{1}{Z_{A_4}(0)-\delta_0/2}\int_0^T  \nZ_{A_3}(s)\nZ_{A_4}(s)\ ds \leq\ \f{2}{\delta_0}\int_0^T  \nZ_{A_3}(s)\nZ_{A_4}(s)\ ds,
\end{align*}
where for the second inequality we have used the assumption that $Z_{A_4}(0) \geq \delta_0$ a.s.
Consequently, because of  the choice of $K_1(\ep)$ in \eqref{eq:K1-choice}, we have for all $n \geq N_0$,
\begin{align*}
\SC{P}^n(K_1(\ep)) \leq \ \PP \lf( \f{2}{\delta_0}\int_0^T  \nZ_{A_3}(s)\nZ_{A_4}(s)\ ds \geq K_1(\ep)\ri)+2\ep/3 \leq \ep.
%\leq&\ \EE\lf(\int_0^T  \nZ_{A_3}(s)\nZ_{A_4}(s)\ ds\ri) \Big/ (Z_{A_4}(0)-\delta) K_1(\ep)+2\ep/3 \leq \ep.
\end{align*}

\end{proof}

\begin{proposition} \label{prop:A3-tight-prop}
Under \Cref{assum:main}, the following hold:
\begin{enumerate}[label=(\roman*), ref=(\roman*)]
%\item \label{item:A3-int-tight} The sequence of random variables $\dst \lf\{\int_0^T \nZ_{A_3}(s)\ ds\ri\}$ is tight.
\item \label{item:A3-conv} the sequence of random variables $\lf\{n^{-1}\sup_{t\leq T}\big(\nZ_{A_3}(t)\big)^{\aexp}\ri\}$ is tight (where $\aexp$ is as in \Cref{assum:main}: \ref{item:initial-other})
\item \label{item:A3-pwr-int-tight}  the sequence of random variables $\dst \lf\{\int_0^T (\nZ_{A_3}(s))^{\aexp}\ ds\ri\}$ is tight.
\end{enumerate}
\end{proposition}

\begin{remark} \label{rem:A3-conv-0}
\Cref{prop:A3-tight-prop}:\ref{item:A3-conv} implies that for any $\aexp_1<\aexp$, $n^{-1}\sup_{t\leq T}\big(\nZ_{A_3}(t)\big)^{\aexp_1} \prt 0$ as $n\rt \infty.$
\end{remark}

\begin{proof}[Proof of \Cref{prop:A3-tight-prop}] Notice that without loss of generality we can assume $\aexp$ to be positive integer valued. We will prove \ref{item:A3-conv}   and \ref{item:A3-pwr-int-tight}  by induction, that is, assuming that $\lf\{n^{-1}\sup_{t\leq T}\big(\nZ_{A_3}(t)\big)^{\aexp'}\ri\}$ and $ \lf\{\int_0^T (\nZ_{A_3}(s))^{\aexp'}\ ds\ri\}$ are tight  for any positive integer $\aexp' < \aexp$, we will show that   \ref{item:A3-conv}   and \ref{item:A3-pwr-int-tight} hold. Observe that the case of  $\aexp'=1$ follows from \Cref{lem:aux-tight}:\ref{item:A1-A2-tight} and \Cref{lem:A3-int-tight}. 
%Now assume that \ref{item:A3-conv}   and \ref{item:A3-pwr-int-tight} hold for all positive integers $\aexp' < \aexp$.

By It\^o's lemma,
\begin{align} \label{eq:A3-sq}
	(\nZ_{A_3}(t))^p=& (\nZ_{A_3}(0))^p + \hat{R}^{(n,\aexp)}_{-7}(t) +  \hat{R}^{(n,\aexp)}_{-7'}(t) + \hat{R}^{(n,\aexp)}_{8}(t) + \hat{R}^{(n,\aexp)}_{7}(t)      	+  \hat{R}^{(n,\aexp)}_{7'}(t) +  \hat{R}^{(n,\aexp)}_{-8}(t),
\end{align}	
where
\begin{align*}
 \hat{R}^{(n,\aexp)}_{-7}(t)  =&\ \int_{[0,\infty)\times [0,t]}\lf((\nZ_{A_3}(s-)+1)^{\aexp}-(\nZ_{A_3}(s-))^{\aexp}\ri) \indic_{[0,n^{1/2} \kappa_{-7} \nZ_{E_1^*}(s-)]}(u) \xi_{-7}(du\times ds)\\
 \hat{R}^{(n,\aexp)}_{-7'}(t)  =&\ \int_{[0,\infty)\times [0,t]}\lf((\nZ_{A_3}(s-)+1)^{\aexp}-(\nZ_{A_3}(s-))^{\aexp}\ri) \indic_{[0,  n^{1/2} \kappa_{-7} \nZ_{E_1^*A_1}(s-)]}(u) \xi_{-7'}(du\times ds)\\
 \hat{R}^{(n,\aexp)}_{8}(t)  =&\ \int_{[0,\infty)\times [0,t]}\lf((\nZ_{A_3}(s-)+1)^{\aexp}-(\nZ_{A_3}(s-))^{\aexp}\ri) \indic_{[0,  n \kappa_8 \nZ_{A_2}(s-)(\nZ_{A_2}(s-)-n^{-1})]}(u) \xi_{8}(du\times ds) \\
 \hat{R}^{(n,\aexp)}_{7}(t)  =&\ \int_{[0,\infty)\times [0,t]}\lf((\nZ_{A_3}(s-)-1)^{\aexp}-(\nZ_{A_3}(s-))^{\aexp}\ri) \indic_{[0, n^{1/2} \kappa_7 \nZ_{E_1}(s-)\nZ_{A_3}(s-)]}(u) \xi_{7}(du\times ds)\\
 \hat{R}^{(n,\aexp)}_{7'}(t)  =&\ \int_{[0,\infty)\times [0,t]}\lf((\nZ_{A_3}(s-)-1)^{\aexp}-(\nZ_{A_3}(s-))^{\aexp}\ri) \indic_{[0,  n^{1/2} \kappa_7 \nZ_{E_1A_1}(s-)\nZ_{A_3}(s-)]}(u) \xi_{7'}(du\times ds) \\
 \hat{R}^{(n,\aexp)}_{-8}(t)  =&\ \int_{[0,\infty)\times [0,t]}\lf((\nZ_{A_3}(s-)-1)^{\aexp}-(\nZ_{A_3}(s-))^{\aexp}\ri) \indic_{[0,  n \kappa_{-8} \nZ_{A_3}(s-)\nZ_{A_4}(s-)]}(u) \xi_{-8}(du\times ds).
\end{align*}
Binomial expansion shows that for $y \geq 0$
\begin{align*}
(y+1)^{\aexp}-y^{\aexp} = g^{(\aexp-1)}(y), \quad (y-1)^{\aexp}-(y)^{\aexp} \leq  -\aexp y^{\aexp-1}+\tilde g^{(\aexp-2)}(y) \leq \tilde g^{(\aexp-2)}(y) ,
\end{align*}
where $g^{(\aexp-1)}(\cdot)$ and $ \tilde g^{(\aexp-2)}(\cdot)$ are nonnegative polynomials of degree $\aexp-1$ and $\aexp-2$, respectively. Thus, there is a constant $\const_{\aexp}$ such that
\begin{align}
0\leq g^{(\aexp-1)}(y) \leq \const_{\aexp}(1+y^{\aexp-1}), \quad 0 \leq \tilde g^{(\aexp-2)}(y) \leq \const_{\aexp}(1+y^{\aexp-2}), \quad y \geq 0
\end{align}
Since the $\xi_k$ are non-negative measures, we have
\begin{align*}
\sup_{t\leq T}(\nZ_{A_3}(t))^p \leq&\ (\nZ_{A_3}(0))^{\aexp} + G^{(n,p-1)}_{-7}(T) + G^{(n,p-1)}_{-7'}(T) + G^{(n,p-1)}_{8}(T) \\
&\ +\tilde G^{(n,p-2)}_{7}(T) + \tilde G^{(n,p-2)}_{7'}(T) +\tilde G^{(n,p-2)}_{-8}(T),
%\leq &\ (\nZ_{A_3}(0))^2 + G^{(n,p-1)}_{-7}(T) + G^{(n,p-1)}_{-7'}(T) + G^{(n,p-1)}_{8}(T) \\
%&\ +|\tilde G^{(n,p-2)}_{7}(T)| + |\tilde G^{(n,p-2)}_{7'}(T)| +|\tilde G^{(n,p-2)}_{-8}(T)|,
\end{align*} 
where %for ${\grn k=\pm7,\pm7',\pm8}$
\begin{align*}
G^{(n,p-1)}_{k}(t) =&
     \begin{cases}
      \int_{[0,\infty)\times [0,t]} g^{(\aexp-1)}(\nZ_{A_3}(s-)) \indic_{[0, n^{1/2} \l_k(\nZ(s-))]}(u) \xi_k(du\times ds), & \quad k=- 7, - 7',\\
       \int_{[0,\infty)\times [0,t]} g^{(\aexp-1)}(\nZ_{A_3}(s-)) \indic_{[0, n \l_k(\nZ(s-))]}(u) \xi_k(du\times ds), & \quad k =8,
      \end{cases}\\
\tilde G^{(n,p-2)}_{k}(t) =&
\begin{cases}
  \int_{[0,\infty)\times [0,t]} \tilde g^{(\aexp-2)}(\nZ_{A_3}(s-)) \indic_{[0, n^{1/2} \l_k(\nZ(s-))]}(u) \xi_k(du\times ds), & \quad k= 7,  7',\\
   \int_{[0,\infty)\times [0,t]} \tilde g^{(\aexp-2)}(\nZ_{A_3}(s-)) \indic_{[0, n \l_k(\nZ(s-))]}(u) \xi_k(du\times ds), & \quad k=-8.
  \end{cases}
  \end{align*}

The assertion  \ref{item:A3-conv} for exponent $\aexp$ follows once we show  for $k= -7, - 7', 8$, the sequences $\{n^{-1}G^{(n,p-1)}_{k}(T)\}$ are tight and for $k=7,7', -8$, the sequences$\{n^{-1}\tilde G^{(n,p-2)}_{k}(T)\}$ are tight. We will show this for the case of $k=8$; that is, we show $\{n^{-1}G^{(n,p-1)}_{8}(T)\}$ is tight. The proofs for the others follow similarly and are actually simpler. To this end, write
\begin{align*}
G^{(n,p-1)}_{8}(t) =&  \gamma^{(n,p-1)}_8(t)+  \mart^{(n,p-1)}_8(t),
\end{align*}
where
\begin{align*}
\gamma^{(n,p-1)}_8(t) =&\ n \int_0^t  g^{(\aexp-1)}(\nZ_{A_3}(s))\l_k(\nZ(s))\ ds=  n \kappa_8 \int_0^tg^{(\aexp-1)}(\nZ_{A_3}(s))  \nZ_{A_2}(s)(\nZ_{A_2}(s)-n^{-1}) \ ds,
\end{align*}
and the martingale, $\mart^{(n,p-1)}_8$, is defined by
\begin{align*}
\mart^{(n,p-1)}_8(t) = &\ \int_{[0,\infty]\times [0,T]} g^{(\aexp-1)}(\nZ_{A_3}(s-)) \indic_{[0,  n \kappa_8 \nZ_{A_2}(s-)(\nZ_{A_2}(s-)-n^{-1})]}(u) \tilde \xi_{8}(du\times ds).
\end{align*}
Since $\lf\{\sup_{t\leq T}\lf(\nZ_{A_2}(t)\ri)^2\ri\}$ and $\lf\{\int_0^T(\nZ_{A_3}(s)^{p-1}\ ds\ri\}$ are tight sequences of random variables by \Cref{lem:aux-tight}:\ref{item:A1-A2-tight} and the induction hypothesis, respectively, it follows that $\{n^{-1} \gamma^{(n,p-1)}_8(T)\}$, satisfying the inequality
\begin{align*}
n^{-1} \gamma^{(n,p-1)}_8(T) \leq &\  \const_{\aexp} \kappa_8 \sup_{t\leq T}\lf(\nZ_{A_2}(t)(\nZ_{A_2}(t)+1)\ri) \int_0^T \lf(1+\nZ_{A_3}(s)\ri)^{\aexp-1} \ ds,
\end{align*}
is tight. 
Next notice that  $\<\mart^{(n,p-1)}_8\>$, the predictable quadratic variation of the martingale $\widehat\mart^{(n)}_8$, is given by
\begin{align*}
\<\mart^{(n,p-1)}_8\>_t = &\ n \kappa_8 \int_0^t \lf(g^{(\aexp-1)}(\nZ_{A_3}(s))\ri)^2 \nZ_{A_2}(s)(\nZ_{A_2}(s)-n^{-1}) \ ds.
\end{align*}
Therefore, 
\begin{align*}
\<n^{-1}\mart^{(n,p-1)}_8\>_T &=\ n^{-2}\<\mart^{(n,p-1)}_8\>_T \leq\ n^{-1}\const_{\aexp}\kappa_8 \int_0^t \lf(1+\nZ_{A_3}(s)\ri)^{2(\aexp-1)} \nZ_{A_2}(s)(\nZ_{A_2}(s)-n^{-1}) \ ds\\
&\leq  \const_{\aexp}\kappa_8 \lf(n^{-1}\sup_{t\leq T} (1+\nZ_{A_3}(s))^{p-1}\ri) \sup_{t\leq T}\lf(\nZ_{A_2}(t)(\nZ_{A_2}(t)+1)\ri) \int_0^T   \lf(1+\nZ_{A_3}(s)\ri)^{\aexp-1}\!\!\! ds.
\end{align*}
Again since $\lf\{\sup_{t\leq T}\lf(\nZ_{A_2}(t)\ri)^2\ri\}, \lf\{n^{-1}\sup_{t\leq T}(\nZ_{A_3}(t))^{\aexp-1} \ri\}$ and $\lf\{\int_0^T(\nZ_{A_3}(s)^{p-1}\ ds\ri\}$ are tight by  \Cref{lem:aux-tight}:\ref{item:A1-A2-tight} and  induction-hypothesis, it follows that $\{\<n^{-1}\mart^{(n,p-1)}_8\>_T\}$ is tight. Hence,  by  \Cref{lem:sup-mart-tight}  $\{\sup_{t\leq T}|n^{-1}\mart^{(n,p-1)}_8(t)|\}$ is tight, and thus so is $\{n^{-1} G^{(n,p-1)}_{8}(T) \}$. This establishes   \ref{item:A3-conv} for exponent $\aexp$.

To prove \ref{item:A3-pwr-int-tight}, notice that 
 $$ \hat{R}^{(n,\aexp)}_{-8}(t) \leq -\aexp\int_0^t \big(\nZ_{A_3}(s-)\big)^{\aexp-1}\indic_{[0,  n \kappa_{-8} \nZ_{A_3}(s-)\nZ_{A_4}(s-)]}(u) \xi_{-8}(du\times ds)+ \tilde G^{(n,p-2)}_{-8}(t)$$
  and then rearranging \eqref{eq:A3-sq}  gives
\begin{align*}
\aexp\int_0^t & \big(\nZ_{A_3}(s)\big)^{\aexp-1}\indic_{[0,  n \kappa_{-8} \nZ_{A_3}(s-)\nZ_{A_4}(s-)]}(u) \xi_{-8}(du\times ds)\\
 \leq &\  (\nZ_{A_3}(0))^{\aexp} + G^{(n,p-1)}_{-7}(T) + G^{(n,p-1)}_{-7'}(T) + G^{(n,p-1)}_{8}(T)  +\tilde G^{(n,p-2)}_{7}(T) + \tilde G^{(n,p-2)}_{7'}(T) +\tilde G^{(n,p-2)}_{-8}(T)
\end{align*}
which, by the results established while proving \ref{item:A3-conv}, shows that the sequence 
$$\lf\{n^{-1}\int_0^t \big(\nZ_{A_3}(s-)\big)^{\aexp-1}\indic_{[0,  n \kappa_{-8} \nZ_{A_3}(s-)\nZ_{A_4}(s-)]}(u) \xi_{-8}(du\times ds)\ri\}$$ is tight.
Now observe that
\begin{align}\label{eq:int-A3pwr-A4}
\begin{aligned}
\kappa_{-8}\int_0^t \big(\nZ_{A_3}(s)\big)^{\aexp}\nZ_{A_4}(s)\ ds =&  - n^{-1}\bar \mart^{(n,\aexp-1)}_{-8}(t) + n^{-1}\int_0^t \big(\nZ_{A_3}(s-)\big)^{\aexp-1}\\
& \hs{3cm} \times\indic_{[0,  n \kappa_{-8} \nZ_{A_3}(s-)\nZ_{A_4}(s-)]}(u) \xi_{-8}(du\times ds),
\end{aligned}
\end{align}
where the martingale term $\bar \mart^{(n,\aexp-1)}_{-8}$ is defined by
\begin{align*}
\bar \mart^{(n,\aexp-1)}_{-8}(t)= \int_0^t \big(\nZ_{A_3}(s-)\big)^{\aexp-1}\indic_{[0,  n \kappa_{-8} \nZ_{A_3}(s-)\nZ_{A_4}(s-)]}(u) \tilde\xi_{-8}(du\times ds)
\end{align*}
The quadratic variation $\<n^{-1}\bar \mart^{(n,p-1)}_{-8}\>$ of the martingale $n^{-1}\bar \mart^{(n,p-1)}_{-8}$  is given by
\begin{align*}
\<n^{-1}\bar \mart^{(n,\aexp-1)}_{-8}\>_t=&\ n^{-1}\kappa_{-8}\int_0^t (\nZ_{A_3}(s))^{2\aexp-1}\nZ_{A_4}(s)\ ds\\
 \leq& \kappa_{-8}\lf(n^{-1}\sup_{s\leq t} (\nZ_{A_3}(s))^{\aexp}\ri) \sup_{s\leq t} \nZ_{A_4}(s) \int_0^t \big(\nZ_{A_3}(s)\big)^{\aexp-1}\ ds.
\end{align*}
Since $\{\sup_{s\leq t} \nZ_{A_4}(s)\}$ and $\{n^{-1}\sup_{s\leq t} (\nZ_{A_3}(s))^{\aexp}\}$ are tight by \Cref{lem:aux-tight}:\ref{item:conv-A4} and just proven assertion \ref{item:A3-conv}, and $\lf\{\int_0^t \big(\nZ_{A_3}(s)\big)^{\aexp-1}\ ds\ri\}$ is tight by the induction hypothesis, $\{\<n^{-1}\bar \mart^{(n,\aexp-1)}_{-8}\>_T\}$ is tight, and  so is $\lf\{n^{-1}\sup_{t\leq T}|\bar \mart^{(n,\aexp-1)}_{-8}(t)|\ri\}$ by \Cref{lem:sup-mart-tight}. We conclude from \eqref{eq:int-A3pwr-A4} that the sequence\\ $\lf\{\int_0^t \big(\nZ_{A_3}(s)\big)^{\aexp}\nZ_{A_4}(s)\ ds\ri\}$ is tight. Now by essentially the same techniques used in the proof of \Cref{lem:A3-int-tight}, it follows that the sequence $\lf\{\int_0^T (\nZ_{A_3}(s))^{\aexp}\ ds\ri\}$ is tight, which proves \ref{item:A3-pwr-int-tight}.

\end{proof}

\begin{corollary} \label{cor:react-tight-minus8}
Under  \Cref{assum:main}, the sequence of processes $\lf\{n^{-1} \nR_{-8}\ri\}$ is $C$-tight in $D([0,T], [0,\infty))$.
\end{corollary}

\begin{proof}
Rearranging the equation of $\nZ_{A_3}$ in \eqref{eq:full_gly_scaled_eq} gives
\begin{align}
\nR_{-8}(t) =&\ \nZ_{A_3}(0) -\nZ_{A_3}(t)- \nR_7(t) + \nR_{-7}(t) - \nR_{7'}(t) +\nR_{-7'}(t) + \nR_8(t).
\end{align}
We already know from \Cref{lem:aux-tight}:\ref{item:react-tight-2} that the sequence $\lf\{n^{-1}\nR_8\ri\}$ is $C$-tight in $D([0,T], [0,\infty))$, and by \Cref{rem:A3-conv-0}, $n^{-1}\sup_{t\leq T}\nZ_{A_3}(t) \prt 0$ . The assertion now follows from the observation that   $\sup_n n^{-1}\EE\lf[\sup_{t\leq T}\nR_k(t)\ri] = \sup_n n^{-1}\EE\lf[\nR_k(T)\ri] \nrt 0$ for $k=\pm 7, \pm 7'$ by \Cref{lem:aux-tight}:\ref{item:react-tight-1} and \eqref{eq:mmt-react-bd-7-7'-8}.
\end{proof}

\np
We are now ready to prove \Cref{prop:rel_compactness}.
\begin{proof}[Proof of \Cref{prop:rel_compactness}]
The $C$-tightness of $\lf\{\nZ_S \equiv (\nZ_{A_1},\nZ_{A_2}, \nZ_{A_4})\ri\}$ follows immediately from \Cref{lem:aux-tight}:\ref{item:react-tight-1}, \ref{item:react-tight-2} \& \ref{item:conv-A4} and \Cref{cor:react-tight-minus8}.

To show that $\lf\{\nocc_{F}\ri\}$ is tight, we need to show that for every $\ep>0$, there exists a compact set $\SC{K} \subset \spmeas_T([0,\infty)^{7}\times[0,T])$ such that
$\limsup_{n\rt \infty}\PP\lf(\nocc_{ F} \notin \SC{K}\ri) \leq \ep$.

By Markov inequality and Prohorov's theorem it easily follows that  for every constant $c>0$, 
$$\SC{K}_{c} \dfeq \lf\{\mu \in \spmeas_T([0,\infty)^7\times [0,T]): \int _{[0,\infty)^7\times [0,T]} \|z\|_1 \mu(dz\times ds) \leq c\ri\}$$ is a  compact set of $\spmeas_T([0,\infty)^7\times [0,T])$. Fix an $\ep>0$. Now because of the assumptions on $\{\nconsconst_1\}$ and $\{\nconsconst_2\}$, there exists an $K_2\equiv K_2(\vep/3)>0$ such that
\begin{align}\label{eq:tight-cons-const}
	\inf_n\PP\lf[0\leq \nconsconst_1, \nconsconst_2 \leq K_2(\vep/3)\ri]\geq 1-\vep/3.
\end{align}   
Now let $K_1(\cdot), N_0(\cdot)$ and $\SC{P}^n(K_1(\cdot))$ be as in the proof of \Cref{lem:A3-int-tight}, so that   \eqref{eq:tight-int-A3} holds.
We now show that for all $n \geq N_0(\ep/3)$, 
$$\tilde{\SC{P}}^n_F(\ep) \dfeq \PP\lf(\nocc_{ F} \notin \SC{K}_{K_1(\ep/3)+2K_2(\ep/3)T}\ri) \leq \ep.$$
 To this end, first notice that by  \eqref{eq:cons_laws},  
$$0\leq \nZ_{E_1}, \nZ_{E_1^*}, \nZ_{E_1A_1}, \nZ_{E_1^*A_1} \leq \  \nconsconst_1,  \quad
 0\leq \nZ_{E_2}, \nZ_{E_2A_2} \leq \nconsconst_2.$$ 

Now writing $\nZ_F =(\nZ_{A_3}, \nZ_{\tilde F})$ with 
\begin{align}
	\label{eq:fast-minus-A3}
	\nZ_{\tilde F}=(\nZ_{E_1}, \nZ_{E_1^*}, \nZ_{E_1A_1}, \nZ_{E_1^*A_1}, \nZ_{E_2}, \nZ_{E_2A_2}),
\end{align}
 we have for all $n \geq N_0(\vep/3)$
%Now denoting  the index set $\tilde{F}$ by $\tilde{F} =\{E_1, E_1^*, E_1A_1, E_1^*A_1, E_2, E_2A_2\}$, we have for all $n \geq N_0(\vep/3)$
\begin{align*}
	\tilde{\SC{P}}^n_F(\ep)=&\ \PP\lf(\int_{[0,\infty)^{7}\times[0,T]}\|z_F\|_1 
     \nocc_{F}(dz_F\times ds) \geq K_1(\ep/3)+2K_2(\ep/3)T\ri)\\
	 =&\ \PP\lf(\int_0^T \left( \nZ_{A_3}(s)+ \|\nZ_{\tilde F}(s)\|_1\ \right) ds \geq K_1(\ep/3)+2K_2(\ep/3)T\ri)\\
	 \leq &\ \SC{P}^n(K_1(\ep/3))+ \PP\lf(\int_0^{T}\|\nZ_{\tilde F}(s)\|_1 \ ds \geq 2K_2(\ep/3)T \ri)\\
	 \leq &\ \SC{P}^n(K_1(\ep/3))+ \PP\lf(T(\nconsconst_1+\nconsconst_2) \geq 2K_2(\ep/3)T\ri)\\
	 \leq &\  \SC{P}^n(K_1(\ep/3))+ \PP(\nconsconst_1 \geq K_2(\ep/3))+ \PP(\nconsconst_2 \geq K_2(\ep/3)) \leq \ep.
\end{align*}

%  
%Once we get these estimates we work with $\modu(\nPhi_S,T,\delta)$, the modulus of continuity of $\nPhi_S$,  given by
%\begin{align*}
%	\modu(\nPhi_S,T,\delta) \dfeq \sup_{\substack{t_1,t_2\in [0,T], \\ |t_1-t_1|\leq \delta}}|\nPhi_S(t_1)-\nPhi_S(t_2)|, 
%\end{align*}
%It now easily follows that for some const $\const$ (independent of $n$)
%\begin{align*}
%\EE\modu(\nPhi_S,T,\delta) \leq \const \delta.
%\end{align*}
%Markov inequality now implies that for $\vep>0, \ \eta>0$, there exist $0<\delta<1$ and $n_0>0$ such that
%\begin{align}
%	\label{eq:modulus-of-cont}
%	\sup_{n\geq n_0}\PP\lf(\modu(\nPhi_S,T,\delta)\geq \eta\ri)\leq \vep,
%\end{align}
%
%Since the sequence $\{|\nPhi_S(0)|\}$ is tight, $C$-tightness of $\nPhi_S$ follows by \cite[Theorem 7.3]{Billingsley1999Convergence},

\end{proof}

Continuing with our convention of denoting states of the processes, we represent a typical state $z_F \in [0,\infty)^7$ of $\nZ_F = (\nZ_{A_3}, \nZ_{\tilde F})$ (see \eqref{eq:fast-minus-A3}) as $z_F = (z_{A_3}, z_{\tilde F})$ with $z_{\tilde F} =(z_{E_1}, z_{E_1^*}, z_{E_1A_1}, z_{E_1^*A_1}, z_{E_2}, z_{E_2A_2})$.

The following corollary is now immediate.
\begin{corollary} \label{cor:phi-prelim-int-tight}
 Let $\phi:[0,\infty)^{10}\times [0,T] \rt \R$ be a  function satisfying the following growth conditions: there exist a constant $\const_{\phi}$ and exponents $\aexp_0 \geq 0,\ \aexp_1 =\aexp$ (with $\aexp$  as in \Cref{assum:main}: \ref{item:initial-other}) such that  for $ (z_{A_3}, z_{\tilde F}) \in [0,\infty)^{7} , z_S \in [0,\infty)^3$, 
   \begin{align}
   \label{eq:phi-growth}
   |\phi(z_F,z_S)| \leq\ \const_{\phi} \lf(1+\|z_{\tilde F}\|_1\ri)^{\aexp_0}\lf(1+z_{A_3}\ri)^{\aexp_1}\lf(1+\|z_S\|_1\ri)^{\aexp_0}   
   \end{align}
Then under   \Cref{assum:main}, $\lf\{ \int_0^t \phi(\nZ_F(s),\nZ_S(s))\ ds\ \equiv\  \int_{[0, t]\times [0,\infty)^7}\phi(z_F,\nZ_S(s)) \nocc_F(dz_F\times ds)\ri\}$ is tight.

\end{corollary}

%\begin{proof}
%Since
% \begin{align*}
% \int_0^t \phi(\nZ_F(s),\nZ_S(s))\ ds \leq  \const_{\phi} \int_0^T \lf(1+\|\nZ_{\tilde F}(s)\|_1\ri)^{\aexp_0}\lf(1+\nZ_{A_3}(s)\ri)^2(1+\|\nZ_S(s)\|_1)^{\aexp_0}ds\\
% \leq &\ \const_{\phi}(1+\nconsconst_1+\nconsconst_2)^{\aexp_0} \sup_{s\leq t}(1+\|\nZ_S(s)\|_1)^{\aexp_1}  \int_0^T \big(\nZ_{A_3}(s)\big)^2\ ds,
% \end{align*}
% the assertion follows from 
%\end{proof}

\begin{lemma} \label{lem:conv-occ-ZS}
   Suppose that \Cref{assum:main} holds  and $(\nocc_F, \nZ_S)\RT  (\occ_F, Z_S)$ as $n\rt \infty$. Let $\phi:[0,\infty)^{10}\times [0,T] \rt \R$ be a continuous function satisfying the  growth condition in \eqref{eq:phi-growth} with $0\leq \aexp_1 <\aexp$ and the following continuity condition:  there exist a constant $\const_{\phi}$ and exponents $\aexp_0 \geq 0$,  such that 
  for $(z_{A_3}, z_{\tilde F}) \in [0,\infty)^{ 7} , z_S, z_S' \in [0,\infty)^3$,
  \begin{align}
  \label{eq:phi-lip}
  |\phi(z_F,z_S, t)- \phi(z_F,z'_S, t)| \leq\ \const_{\phi} \lf(1+\|z_{\tilde F}\|_1\ri)^{\aexp_0}\lf(1+z_{A_3}\ri)^{\aexp}\lf(1+\|z_S\|_1+\|z'_S\|_1\ri)^{\aexp_0}\|z_S - z'_S\|_1.
  \end{align}
  Then as $n\rt \infty$
    \begin{align*}%\label{eq:occ-int}
   % \begin{aligned}
        \int_0^t \phi(\nZ_F(s),\nZ_S(s))\ ds\ \equiv&\  \int_{[0, t]\times [0,\infty)^7}\phi(z_F,\nZ_S(s)) \nocc_F(dz_F\times ds)\\
         \RT&\ \int_{[0, t]\times [0,\infty)^7 }\phi(z_F,Z_S(s)) \occ_F(dz_F\times ds).
     %  \end{aligned}
    \end{align*}
 %If $(\nocc_F, \nZ_S)\rt  (\occ_F, Z_S)$ a.s. or in probability,  then the above convergence  holds in probability. 
\end{lemma}

\begin{proof}
Notice that by \Cref{assum:main}:\ref{item:cons-const}, \Cref{prop:rel_compactness} and \Cref{prop:A3-tight-prop}:\ref{item:A3-pwr-int-tight}, the sequence \\ $\lf\{\lf(\nocc_F, \nZ_S, \nconsconst_1, \nconsconst_2, \int_0^t \big(1+\nZ_{A_3}(s)\big)^{\aexp}ds\ri) \ri\}$ is relatively compact, and let $\lf(\occ_F, Z_S, \consconst_1, \consconst_2, V\ri)$, for some $[0,\infty)$-valued random variables $\consconst_1, \consconst_2, V$,  be its limit point. By Skorohod's theorem, we can assume that 
\begin{align}\label{eq:conv-skoro}
\lf(\nocc_F, \nZ_S, \nconsconst_1, \nconsconst_2, \int_0^t \big(1+\nZ_{A_3}(s)\big)^{\aexp} ds\ri) \nrt \lf(\occ_F, Z_S, \consconst_1, \consconst_2, V\ri), \quad \text{a.s.}
\end{align}
along a subsequence. By, a slight abuse of notation, we continue to denote the subsequence by $\{n\}$.

Recall from \Cref{prop:rel_compactness}, the paths of $Z_S$ are almost surely in $C([0, T], [0,\infty)^3)$ (in fact, from \Cref{lem:aux-tight}:\ref{item:conv-A4}, $Z_4(s) \equiv Z_4(0)$). In particular, $\sup_{s\leq T}\|Z_S(s)\|_1 < \infty$ a.s., and 
	\begin{align}\label{eq:unif-conv-Z_V}
		\sup_{t\leq T}\|\nZ_S(t) - Z_S(t)\|_1 \stackrel{n\rt\infty} \Rt 0, \quad \text{a.s.}
	\end{align}
Next write
	\begin{align*}
 \int_0^t \phi(\nZ_F(s),\nZ_S(s))\ ds =&\  \int_0^t \phi(\nZ_F(s),Z_S(s))\ ds + \cal{E}^{(n,0)}(t).
	\end{align*}
Here, because of  \eqref{eq:cons_laws} amd \eqref{eq:phi-lip},  the term $\cal{E}^{(n,0)}(t) \dfeq \int_0^t \lf(\phi(\nZ_F(s),\nZ_S(s)) -\phi(\nZ_F(s),Z_S(s))\ri) ds $ can be estimated as 
\begin{align*}
\begin{aligned}
	|\cal{E}^{(n,0)}(t)| & \leq \ \const_{\phi}(1+\nconsconst_1+\nconsconst_2)^{\aexp_0} \sup_{s\leq t}(1+\|Z_S(s)\|_1+\|\nZ_S(s)\|_1)^{\aexp_0} \sup_{s\leq t}\|\nZ_S(s) - Z_S(s)\|_1\\
	&\ \times  \int_0^t \big(1+\nZ_{A_3}(s)\big)^{\aexp}\ ds\ \nrt\ \ 0, \quad \text{a.s.}
	\end{aligned}
\end{align*}
where the convergence holds because of \eqref{eq:conv-skoro} and \eqref{eq:unif-conv-Z_V}.

%Now since path space of $Z_S$ is $C([0,T], [0,\infty]^3)$, $\sup_{s\leq T}\|Z_S(s)\|_1 < \infty$ a.s. 
%Because of \Cref{assum:main}:\ref{item:cons-const}, \Cref{lem:aux-tight}:\ref{item:A1-A2-tight} \& \ref{item:conv-A4} (which shows   $ \sup_n\EE(\sup_{s\leq T}\|\nZ_S(s)\|_1) < \infty$)
%and \Cref{prop:A3-tight-prop}:\ref{item:A3sq-int-tight}, it is now clear that \eqref{eq:unif-conv-Z_V} implies, as $n\rt \infty$,
%$\sup_{t\leq T}|\cal{E}^{(n)}(t)| \prt 0.$

The assertion will follow once we now show that 
\begin{align}
\label{eq:phi-int-conv-1}
\begin{aligned}
\int_0^t \phi(\nZ_F(s),Z_S(s))\ ds &\ \equiv \int_{[0, t]\times [0,\infty)^7}\!\!\!\!\phi(z_F, Z_S(s)) \nocc_F(dz_F\times ds)\\ &\ \nrt \int_{[0, t]\times [0,\infty)^7}\!\!\!\!\phi(z_F, Z_S(s)) \occ_F(dz_F\times ds).
\end{aligned}
\end{align}
Toward this end, we first observe that
\begin{align}
\label{eq:int-phi-fin}
\int_{[0, T]\times [0,\infty)^7}|\phi(z_F, Z_S(s))| \occ_F(dz_F\times ds) < \infty, \quad \text{a.s.}
\end{align}
Indeed, by Fatou's lemma and the growth assumption on $\phi$,
\begin{align*}
\int_{[0, T]\times [0,\infty)^7}& |\phi(z_F, Z_S(s))| \occ_F(dz_F\times ds)\\
\leq  & \ \const_{\phi}\int_{[0, T]\times [0,\infty)^7}\lf(1+\|z_{\tilde F}\|_1\ri)^{\aexp_0}\lf(1+z_{A_3}\ri)^{\aexp_1}\lf(1+\|Z_S(s)\|_1\ri)^{\aexp_0} \occ_F(dz_F\times ds) \\
\leq  & \ \const_{\phi} \liminf_{n\rt \infty} \int_{[0, T]\times [0,\infty)^7}\lf(1+\|z_{\tilde F}\|_1\ri)^{\aexp_0}\lf(1+z_{A_3}\ri)^{\aexp_1}\lf(1+\|Z_S(s)\|_1\ri)^{\aexp_0}  \nocc_F(dz_F\times ds) \\
=&\  \const_{\phi} \liminf_{n\rt \infty} \int_0^T \lf(1+\|\nZ_{\tilde F}(s)\|_1\ri)^{\aexp_0}\lf(1+\nZ_{A_3}(s)\ri)^{\aexp_1}(1+\|Z_S(s)\|_1)^{\aexp_0}ds \\
\leq &\ \liminf_{n\rt \infty} \const_{\phi}(1+\nconsconst_1+\nconsconst_2)^{\aexp_0} \sup_{s\leq t}(1+\|Z_S(s)\|_1)^{\aexp_0}  \lf(\int_0^{ T} \big(1+\nZ_{A_3}(s)\big)^{\aexp}\ ds\ri)^{\aexp_1/\aexp}{ T}^{1-\aexp_1/\aexp}\\
= &\  \const_{\phi}(1+\consconst_1+\consconst_2)^{\aexp_0} \sup_{s\leq t}(1+\|Z_S(s)\|_1)^{\aexp_0}  V^{\aexp_1/\aexp}{ T}^{1-\aexp_1/\aexp}
< \ \infty, \quad \text{a.s.},
\end{align*}
where the fourth step used H\"older's inequality.  
Define the compact set $\cal{K}_B \subset [0,\infty)^7$ by
\begin{align*}
\cal{K}_B = \lf\{z_{\tilde F} \in [0,\infty)^6: \|z_{\tilde F}\|_1 \leq B \ri\} \times [0,B]
\end{align*}
Now \eqref{eq:int-phi-fin} implies that 
\begin{align}
\label{eq:int-phi-tail}
\int_{[0, T]\times \lf(\cal{K}_B\ri)^c}|\phi(z_F, Z_S(s))| \occ_F(dz_F\times ds) \stackrel{B \rt \infty}\Rt 0, \quad \text{a.s.}
\end{align}
and hence also in probability, there exists a $B_0>0$ such that for any $B>B_0$,
\begin{align}
\label{eq:phi-int-tail}
\PP\lf(\lf| \int_{[0, T]\times \lf([0,B]^7\ri)^c}|\phi(z_F, Z_S(s))| \occ_F(dz_F\times ds)\ri|>\eta/6\ri) \leq \ep/3.
\end{align}
By Urysohn's Lemma,  \cite[Page 122]{foll99}, let $\tilde \phi_{B} \in C([0,\infty)^{10}, \R)$ be such that $\lf|\tilde \phi_{B}\ri| \leq |\phi|$ pointwise, and
    \begin{align*}
        \tilde \phi_{B}(z_F,z_S) = 
        \begin{cases}
            \phi(z_F,z_S),& \quad \text{for } (z_F,z_S) \in \cal{K}_B, \\
        0, & \quad \text{for }  (z_F,z_S) \in (\cal{K}^\circ_{B+1})^c.
        \end{cases}
    \end{align*} 
   where $A^\circ$ denotes  interior of a set $A$. 
 Notice that the mapping $(z_F,s) \in \cal{K}_B\times [0,T]  \rt \tilde \phi_{B}(z_F,Z_S(s))$  is continuous and bounded. Since $\nocc_F \rt \occ_F$ a.s  in the weak topology by \eqref{eq:conv-skoro}, we have  
\begin{align}
\label{eq:occ-weak-conv}
\int_{[0, T]\times [0,\infty)^7}  \tilde \phi_{B}(z_F,Z_S(s)) (\nocc_F -\occ_F)(dz_F\times ds) \rt 0 \text{ a.s.}.
\end{align}

% and hence in probability.
 Now finally write
 \begin{align}\label{eq:occ-lim-diff}
\int  \phi(z_F,Z_S(s)) (\nocc_F -\occ_F)(dz_F\times ds)
=\int  \tilde \phi_{B}(z_F,Z_S(s)) (\nocc_F -\occ_F)(dz_F\times ds) + \cal{E}^{(n,1)}_{\phi,B}(t)- \cal{E}^{(2)}_{\phi,B}(t),
 \end{align}
 where  
 $$\lf(\cal{E}^{(n,1)}_{\phi,B}(t), \cal{E}^{(2)}_{\phi,B}(t)\ri)  \dfeq \int_{[0, T]\times [0,\infty)^7}  (\phi - \tilde \phi_{B})(z_F,Z_S(s)) \lf(\nocc_F,\occ_F\ri)(dz_F\times ds).$$ 
 Now observe that for any $t \in [0,T]$
 \begin{align*}
|&\cal{E}^{(n,1)}_{\phi,B}( t)| \leq\ 2\int_{[0, T]\times (\cal{K}_B)^c}|\phi(z_F,Z_S(s))|\nocc_F (dz_F\times ds) \\
 &\leq\ 2\int_{[0, T]\times [0,\infty)^7}  |\phi(z_F,Z_S(s))| \lf(\indic_{\{\|z_{\tilde F}\|_1\geq B\}}+\indic_{\{z_{A_3}\geq B\}}\ri)\nocc_F (dz_F\times ds)\\
 &\leq\ \const_{\phi}(1+\nconsconst_1+\nconsconst_2)^{\aexp_0} \sup_{s\leq T}(1+\|Z_S(s)\|_1)^{\aexp_0} \Bigg\{ \lf(\int_0^T \big(1+\nZ_{A_3}(s)\big)^{\aexp}\ ds\ri)^{\aexp_1/\aexp}\!\! T^{1-\aexp_1/\aexp}\indic_{\{\nconsconst_1+\nconsconst_2 \geq B\}} \\
 & \hs{1cm}+ \int_0^T \big(1+\nZ_{A_3}(s)\big)^{\aexp_1}\indic_{\{\nZ_{A_3}(s)\geq B\}} ds\Bigg\}\\
&\leq\ \const_{\phi}(1+\nconsconst_1+\nconsconst_2)^{\aexp_0} \sup_{s\leq T}(1+\|Z_S(s)\|_1)^{\aexp_0} \Bigg\{ \lf(\int_0^T \big(1+\nZ_{A_3}(s)\big)^{ p}\ ds\ri)^{\aexp_1/\aexp}\!\! T^{1-\aexp_1/\aexp}\indic_{\{\nconsconst_1+\nconsconst_2 \geq B\}}\\
 & \hs{1cm}+ (1+B)^{-(\aexp-\aexp_1)}\int_0^T \big(1+\nZ_{A_3}(s)\big)^{\aexp} ds\Bigg\}.
 \end{align*}
 By \eqref{eq:conv-skoro} and the upper semicontinuity of the indicator function $\indic_{[B,\infty)}(\cdot)$, it follows that
 \begin{align*}
 \limsup_{n \rt \infty}|\cal{E}^{(n,1)}_{\phi,B}(t)| \leq&\  \const_{\phi}(1+\consconst_1+\consconst_2)^{\aexp_0} \sup_{s\leq T}(1+\|Z_S(s)\|_1)^{\aexp_0}\lf\{V^{\aexp_1/\aexp}T^{1-\aexp_1/\aexp}\indic_{\{\consconst_1+\consconst_2 \geq B\}})+(1+B)^{-(\aexp-\aexp_1)}V\ri\}\\
 \dfeq&\ |\cal{E}^{(1)}_{\phi,B}(T)|.
 \end{align*} 
This, together with \eqref{eq:occ-weak-conv} and \eqref{eq:occ-lim-diff}, gives
 \begin{align}\label{eq:phi-int-conv-2}
 \limsup_{n\rt \infty}\lf|\int_{[0, T]\times [0,\infty)^7}  \phi(z_F,Z_S(s)) (\nocc_F -\occ_F)(dz_F\times ds)\ri| \leq |\cal{E}^{(1)}_{\phi,B}(T)| +|\cal{E}^{(2)}_{\phi,B}(T)|. 
 \end{align}
Since $\aexp_1<\aexp$ by assumption, $|\cal{E}^{(1)}_{\phi,B}(T)| \stackrel{B\rt \infty}\Rt 0$. Moreover, by \eqref{eq:int-phi-tail},
\[
|\cal{E}^{(2)}_{\phi,B}(T)| \leq 2 \int_{[0, T]\times (\cal{K}_B)^c}|\phi(z_F,Z_S(s))|\occ_F (dz_F\times ds) \stackrel{B\rt \infty}\Rt 0.
\]
\eqref{eq:phi-int-conv-1} now follows by taking $B \rt \infty$ in \eqref{eq:phi-int-conv-2}.

\end{proof}

\np
{\em Generator of the fast process:} For a fixed $z_S = (z_{A_1}, z_{A_2},z_{A_4}) \in [0,\infty)^3$, define the operator $\fgen_{z_S}$ by
% \begin{align}
% \label{eq:lim-fast-gen}
% \begin{aligned}
% \fgen_{z_S}g(z_F)=&\ \l_1(z_F, z_S) \lf(g\big(z_F-e^{(7)}_2+e^{(7)}_4\big) - g(z_F)\ri) + \l_{-1}(z_F, z_S) \lf(g\big(z_F+e^{(7)}_2-e^{(7)}_4\big) - g(z_F)\ri)\\
% & +  \l_{2}(z_F, z_S) \lf(g\big(z_F+e^{(7)}_2-e^{(7)}_4\big) - g(z_F)\ri)+ \l_{3}(z_F, z_S) \lf(g\big(z_F-e^{(7)}_3+e^{(7)}_5\big) - g(z_F)\ri)\\
% & +  \l_{-3}(z_F, z_S)  \lf(g\big(z_F+e^{(7)}_3-e^{(7)}_5\big) - g(z_F)\ri)+ \l_{4}(z_F, z_S) \lf(g\big(z_F+e^{(7)}_3-e^{(7)}_5\big) - g(z_F)\ri)\\
% & +  \l_{5}(z_F, z_S)  \lf(g\big(z_F-e^{(7)}_6+e^{(7)}_7\big) - g(z_F)\ri)+ \l_{-5}(z_F, z_S) \lf(g\big(z_F+e^{(7)}_6-e^{(7)}_7\big) - g(z_F)\ri)\\
% & +  \l_{6}(z_F, z_S)  \lf(g\big(z_F+e^{(7)}_6-e^{(7)}_7\big) - g(z_F)\ri)+ \tilde\l_{8}(z_F, z_S) \lf(g\big(z_F+e^{(7)}_1\big) - g(z_F)\ri)\\
% & + \l_{-8}(z_F, z_S) \lf(g\big(z_F-e^{(7)}_1\big) - g(z_F)\ri),
% \end{aligned}
% \end{align}
\begin{align}
\label{eq:lim-fast-gen}
\begin{aligned}
\fgen_{z_S}g(z_F)=&\ \kappa_1 z_{E_1}z_{A_1} \lf(g\big(z_F-e^{(7)}_2+e^{(7)}_4\big) - g(z_F)\ri) + \kappa_{-1} z_{E_1A_1} \lf(g\big(z_F+e^{(7)}_2-e^{(7)}_4\big) - g(z_F)\ri)\\
& +  \kappa_2 z_{E_1A_1} \lf(g\big(z_F+e^{(7)}_2-e^{(7)}_4\big) - g(z_F)\ri)+ \kappa_3 z_{E_1^*}z_{A_1} \lf(g\big(z_F-e^{(7)}_3+e^{(7)}_5\big) - g(z_F)\ri)\\
& +  \kappa_{-3} z_{E_1^*A_1}  \lf(g\big(z_F+e^{(7)}_3-e^{(7)}_5\big) - g(z_F)\ri)+ \kappa_4 z_{E_1^*A_1} \lf(g\big(z_F+e^{(7)}_3-e^{(7)}_5\big) - g(z_F)\ri)\\
& +  \kappa_5 z_{E_2}z_{A_2}  \lf(g\big(z_F-e^{(7)}_6+e^{(7)}_7\big) - g(z_F)\ri)+ \kappa_{-5} z_{E_2A_2} \lf(g\big(z_F+e^{(7)}_6-e^{(7)}_7\big) - g(z_F)\ri)\\
& +  \kappa_{6} z_{E_2A_2} \lf(g\big(z_F+e^{(7)}_6-e^{(7)}_7\big) - g(z_F)\ri)+ \kappa_8z_{A_2}^2 \lf(g\big(z_F+e^{(7)}_1\big) - g(z_F)\ri)\\
& + \kappa_{-8} z_{A_3}z_{A_4} \lf(g\big(z_F-e^{(7)}_1\big) - g(z_F)\ri),
\end{aligned}
\end{align}
where $e^{(7)}_i, \ i=1,2,\hdots 7$ are canonical unit vectors in $\R^7$. $\fgen_{z_S}$ can be interpreted as the generator of the fast process when the slow component is frozen at the state $z_S$.

We are now ready to prove \Cref{thm:FLLN-Glyco}.
\subsection{Proof of \Cref{thm:FLLN-Glyco}}
Let $(\occ_F, Z_S =(Z_{A_1}, Z_{A_2}, Z_{A_4}))$ be a limit point of $\{(\nocc_F, \nZ_S)\}$, which exists by \Cref{prop:rel_compactness}. 
Thus there exists a subsequence along which $(\nocc_F, \nZ_S)\RT (\occ_F, Z_S)$.
 % By a slight abuse of notation we continue to denote the subsequence by $(\occ_n, Z_V^{(n)}).$  We will show the limit point $(\occ, Z_V) = (\pi_{Z_S}\star \leb, (Z_s,Z_P)) $ is unique and is independent of the subsequence, and thus, the convergence holds along the entire sequence. Furthermore, since the limit is deterministic, convergence holds also in probability.
By Skorokhod representation theorem, we can assume, without loss of generality, that $(\nocc_F, \nZ_S) \rt (\occ_F, Z_S)$ a.s. in  $\spmeas\lf([0,\infty)^7  \times [0, T]\ri)\times D([0, T], [0,\infty)^3)$ along this subsequence. Notice that by $C$-tightness of $\nZ_S$ established in \Cref{prop:rel_compactness}, $Z_S$ has continuous paths, and 
\begin{align}\label{eq:unif-conv-Z_V-1}
\sup_{t\leq T}\|\nZ_S(t) - Z_S(t)\| \stackrel{n\rt\infty} \Rt 0, \quad \text{a.s.}
\end{align}
Here, by a slight abuse of notation we continued to denote the subsequence by $\{(\nocc_F, \nZ_S)\}$. Notice that by \Cref{lem:aux-tight}:\ref{item:conv-A4} and \Cref{assum:main}:\ref{item:initial-other}, $Z_{A_4}(t) \equiv Z_{A_4}(0)$ for $t \in [0,T]$.

Let $g:[0,\infty)^7 \rt \R$ be a measurable function satisfying
  \begin{align}
   \label{eq:gen-testfn-growth}
   |g(z_F)| \leq\ \const_{g} \lf(1+\|z_{\tilde F}\|_1\ri)^{\aexp_0}\lf(1+z_{A_3}\ri)^{\aexp_1}, 
   \end{align}
  with the constant $\const_{g} \geq 0$ and exponents $\aexp_0 \geq 0$ and $0\leq \aexp_1< \aexp-1$ 
Then It\^o's lemma gives 
\begin{align*}
g(\nZ_F(t))- g(\nZ_F(0)) =&\ n \int_0^t \fgen^{(n)}_{\nZ_S(s)}g(\nZ_F(s))\ ds+ \nmart_{F,g}(t),
%=&\ n \int_{[0,t]\times[0,\infty)} \fgen^{(n)}_{\nZ_S(s)}g(z_F)\ \nocc(dz_F \times ds)+ \nmart_{F,g}(t),
\end{align*}
that is,
\begin{align} \label{eq:fgen-prelim}
\int_0^t \fgen^{(n)}_{\nZ_S(s)}g(\nZ_F(s))\ ds = n^{-1}\lf(g(\nZ_F(t))- g(\nZ_F(0))\ri)-n^{-1}\nmart_{F,g}(t),
\end{align}
where for a fixed $z_S \in [0,\infty)^3$,
%{\grn 
% HW: What about we write (3.39) as $\fgen^{(n)}_{z_S}g(z_F)=\fgen_{z_S}g(z_F)+n^{-1/2}(\cdots)$?
%}
\begin{align}
\label{eq:prelim-fast-gen}
\begin{aligned}
\fgen^{(n)}_{z_S}&g(z_F)=\ \l_1(z_F, z_S) \lf(g\big(z_F-e^{(7)}_2+e^{(7)}_4\big) - g(z_F)\ri) + \l_{-1}(z_F, z_S) \lf(g\big(z_F+e^{(7)}_2-e^{(7)}_4\big) - g(z_F)\ri)\\
& +  \l_{2}(z_F, z_S) \lf(g\big(z_F+e^{(7)}_2-e^{(7)}_4\big) - g(z_F)\ri)+ \l_{3}(z_F, z_S) \lf(g\big(z_F-e^{(7)}_3+e^{(7)}_5\big) - g(z_F)\ri)\\
& +  \l_{-3}(z_F, z_S)  \lf(g\big(z_F+e^{(7)}_3-e^{(7)}_5\big) - g(z_F)\ri)+ \l_{4}(z_F, z_S) \lf(g\big(z_F+e^{(7)}_3-e^{(7)}_5\big) - g(z_F)\ri)\\
& +  \l_{5}(z_F, z_S)  \lf(g\big(z_F-e^{(7)}_6+e^{(7)}_7\big) - g(z_F)\ri)+ \l_{-5}(z_F, z_S) \lf(g\big(z_F+e^{(7)}_6-e^{(7)}_7\big) - g(z_F)\ri)\\
& +  \l_{6}(z_F, z_S)  \lf(g\big(z_F+e^{(7)}_6-e^{(7)}_7\big) - g(z_F)\ri)+ n^{-1/2}\l_{7}(z_F, z_S) \lf(g\big(z_F-e^{(7)}_1-e^{(7)}_2+e^{(7)}_3\big) - g(z_F)\ri)\\
& + n^{-1/2}\l_{-7}(z_F, z_S) \lf(g\big(z_F+e^{(7)}_1+e^{(7)}_2-e^{(7)}_3\big) - g(z_F)\ri)\\
& +n^{-1/2}\l_{7'}(z_F, z_S) \lf(g\big(z_F-e^{(7)}_1-e^{(7)}_4+e^{(7)}_5\big) - g(z_F)\ri)\\
& + n^{-1/2}\l_{-7'}(z_F, z_S) \lf(g\big(z_F+e^{(7)}_1+e^{(7)}_4-e^{(7)}_5\big) - g(z_F)\ri)+\l_{8}(z_F, z_S) \lf(g\big(z_F+e^{(7)}_1\big) - g(z_F)\ri)\\
& + \l_{-8}(z_F, z_S) \lf(g\big(z_F-e^{(7)}_1\big) - g(z_F)\ri),
\end{aligned}
\end{align}
and the martingale $\nmart_{F,g}$ is given by
\begin{align*}
&\nmart_{F,g}(t)=\ \int_{[0,t]\times [0,\infty)} \lf(g\big(\nZ_F(s-)-e^{(7)}_2+e^{(7)}_4\big) - g(\nZ_F(s-))\ri)\indic_{[0,n\l_1(\nZ_F(s-), \nZ_S(s-))]}(u) \tilde \xi_{1}(du\times ds)\\
&  \int_{[0,t]\times [0,\infty)} \lf(g\big(nZ_F(s-)+e^{(7)}_2-e^{(7)}_4\big)  - g(\nZ_F(s-))\ri)\indic_{[0,n\l_{-1}(\nZ_F(s-), \nZ_S(s-))]}(u) \tilde \xi_{-1}(du\times ds)\\
&\ \int_{[0,t]\times [0,\infty)} \lf(g\big(\nZ_F(s-)+e^{(7)}_2-e^{(7)}_4\big)  - g(\nZ_F(s-))\ri)\indic_{[0,n\l_{2}(\nZ_F(s-), \nZ_S(s-))]}(u) \tilde \xi_{2}(du\times ds)\\
&\ \int_{[0,t]\times [0,\infty)} \lf(g\big(\nZ_F(s-)-e^{(7)}_3+e^{(7)}_5\big)  - g(\nZ_F(s-))\ri)\indic_{[0,n\l_{3}(\nZ_F(s-), \nZ_S(s-))]}(u) \tilde \xi_{3}(du\times ds)\\
&\ \int_{[0,t]\times [0,\infty)} \lf(g\big(\nZ_F(s-)+e^{(7)}_3-e^{(7)}_5\big)  - g(\nZ_F(s-))\ri)\indic_{[0,n\l_{-3}(\nZ_F(s-), \nZ_S(s-))]}(u) \tilde \xi_{-3}(du\times ds)\\
&\ \int_{[0,t]\times [0,\infty)} \lf(g\big(\nZ_F(s-)+e^{(7)}_3-e^{(7)}_5\big)  - g(\nZ_F(s-))\ri)\indic_{[0,n\l_{4}(\nZ_F(s-), \nZ_S(s-))]}(u) \tilde \xi_{4}(du\times ds)\\
&\ \int_{[0,t]\times [0,\infty)} \lf(g\big(\nZ_F(s-)-e^{(7)}_6+e^{(7)}_7\big)  - g(\nZ_F(s-))\ri)\indic_{[0,n\l_{5}(\nZ_F(s-), \nZ_S(s-))]}(u) \tilde \xi_{5}(du\times ds)\\
&\ \int_{[0,t]\times [0,\infty)} \lf(g\big(\nZ_F(s-)+e^{(7)}_6-e^{(7)}_7\big)  - g(\nZ_F(s-))\ri)\indic_{[0,n\l_{-5}(\nZ_F(s-), \nZ_S(s-))]}(u) \tilde \xi_{-5}(du\times ds)\\
&\ \int_{[0,t]\times [0,\infty)} \lf(g\big(\nZ_F(s-)+e^{(7)}_6-e^{(7)}_7\big)  - g(\nZ_F(s-))\ri)\indic_{[0,n\l_{6}(\nZ_F(s-), \nZ_S(s-))]}(u) \tilde \xi_{6}(du\times ds)\\
&\ \int_{[0,t]\times [0,\infty)} \lf(g\big(\nZ_F(s-)-e^{(7)}_1-e^{(7)}_2+e^{(7)}_3\big)  - g(\nZ_F(s-))\ri)\indic_{[0,n^{1/2}\l_{7}(\nZ_F(s-), \nZ_S(s-))]}(u) \tilde \xi_{7}(du\times ds)\\
&\ \int_{[0,t]\times [0,\infty)} \lf(g\big(\nZ_F(s-)+e^{(7)}_1+e^{(7)}_2-e^{(7)}_3\big)  - g(\nZ_F(s-))\ri)\indic_{[0,n^{1/2}\l_{-7}(\nZ_F(s-), \nZ_S(s-))]}(u) \tilde \xi_{-7}(du\times ds)\\
&\ \int_{[0,t]\times [0,\infty)} \lf(g\big(\nZ_F(s-)-e^{(7)}_1-e^{(7)}_4+e^{(7)}_5\big)  - g(\nZ_F(s-))\ri)\indic_{[0,n^{1/2}\l_{7'}(\nZ_F(s-), \nZ_S(s-))]}(u) \tilde \xi_{7'}(du\times ds)\\
&\ \int_{[0,t]\times [0,\infty)} \lf(g\big(\nZ_F(s-)+e^{(7)}_1+e^{(7)}_4-e^{(7)}_5\big)  - g(\nZ_F(s-))\ri)\indic_{[0,n^{1/2}\l_{-7'}(\nZ_F(s-), \nZ_S(s-))]}(u) \tilde \xi_{-7'}(du\times ds)\\
&\ \int_{[0,t]\times [0,\infty)} \lf(g\big(\nZ_F(s-)+e^{(7)}_1\big)  - g(\nZ_F(s-))\ri)\indic_{[0,n\l_{8}(\nZ_F(s-), \nZ_S(s-))]}(u) \tilde \xi_{8}(du\times ds)\\
&\ \int_{[0,t]\times [0,\infty)} \lf(g\big(\nZ_F(s-)-e^{(7)}_1\big)  - g(\nZ_F(s-))\ri)\indic_{[0,n\l_{-8}(\nZ_F(s-), \nZ_S(s-))]}(u) \tilde \xi_{-8}(du\times ds).
\end{align*} 
By \Cref{lem:conv-occ-ZS}, 
\begin{align}\label{eq:fgen-conv}
 \int_0^t \fgen^{(n)}_{\nZ_S(s)}g(\nZ_F(s))\ ds \nrt \int_{[0,t]\times[0,\infty)} \fgen_{Z_S(s)}g(z_F)\ \occ_F(dz_F \times ds).
\end{align}
Also notice that by \Cref{cor:phi-prelim-int-tight}, $\<n^{-1}\nmart_{F,g}\>_T =  n^{-2} \<\nmart_{F,g}\>_T \rt 0$ as $n \rt \infty$. Therefore by \Cref{lem:sup-mart-tight}, $n^{-1}\sup_{t\leq T}|\nmart_{F,g}(t)| \rt 0$ as $n\rt \infty$. Moreover, because of the assumption on $g$ in \eqref{eq:gen-testfn-growth}, \Cref{assum:main} and \Cref{rem:A3-conv-0}, $\sup_{t\leq T}n^{-1}\lf|g(\nZ_F(t))- g(\nZ_F(0))\ri| \prt 0$ as $n \rt \infty$. It follows from \eqref{eq:fgen-prelim} and \eqref{eq:fgen-conv} that 
\begin{align}\label{eq:occ-meas-char}
\int_{[0,t]\times[0,\infty)} \fgen_{Z_S(s)}g(z_F)\ \occ_F(dz_F \times ds)=0.
\end{align}
Since $\nocc_F([0,\infty)^7\times [0,t]) =t$ for any $n\geq 1$ and $t>0$, $\occ_F([0,\infty)^7\times [0,t]) =t$. Therefore, we can split $\occ_F$ as $\occ_F(dz_F\times ds) \equiv \occ_F^{(2|1)}(dz_F|s) ds$. It is now clear from \eqref{eq:occ-meas-char} that for a.a. $s$, $\occ_F^{(2|1)}(dz_F|s)$ is an invariant distribution of $\fgen_{Z_S(s)}$; that is, $\occ_F^{(2|1)}(dz_F|s) = \inv_{Z_S(s)}$, where for a fixed $z_S \in [0,\infty)^3$, $\inv_{z_S}$ solves
\begin{align}
\label{eq:stat-dist-char}
\int_{[0,\infty)^7} \fgen_{z_S}g(z_F) \inv_{z_S}(dz_F)=0
\end{align}
with $g$ satisfying \eqref{eq:gen-testfn-growth}.
%It follows that $\occ_F^{(2|1)}(dz_F|s) = \inv_{Z_S(s)}$, where for $z_S \in [0,\infty)^3$, $\inv_{z_S}$ is an invariant distribution of the generator $ \fgen_{z_S}$. 
Thus $\occ_F$ has the form $\occ_F(dz_F\times ds) \equiv \pi_{Z_S(s)}(dz_F)ds$, with $\pi_{z_S}$ being an invariant distribution of $\fgen_{z_S}$.

Then by \Cref{lem:conv-occ-ZS}, it follows that for each $k$,
\begin{align}\label{eq:conv-int+centrd-react}
		\nmeanreact_k(t) = \int_0^t \l_k(\nZ_F(s), \nZ_S(s))\ ds & \nrt \int_{[0,t]\times [0,\infty)^7} \!\! \! \!  \l_k(z_F, Z_S(s))\pi_{Z_S(s)}(dz_F)ds \equiv \int_0^t \bar \l_k(Z_S(s))ds,
\end{align}
where the averaged intensity functions $\bar \l_k: [0,\infty)^3 \rt [0,\infty)$ are defined as 
\begin{align*}
	\bar \l_k(z_S) = 
	\begin{cases}
		\int_{[0,\infty)^7} \l_k(z_F, z_S) \inv_{z_S}(dz_F),& \quad k \neq 8,\\
		\kappa_8 (z_{A_2}\big)^2,& \quad k=8.
	\end{cases}
\end{align*}
It follows that
\begin{align}
	\label{eq:centrd-react-conv}
	\begin{aligned}
	\<n^{-1}\tnR_k\>_T &\nrt 0, 
	&& k  = 0, \pm 1, 2, \pm 3, 4, \pm 5, 6, \pm 8, \\
	\<n^{-1/2}\tnR_k\>_T &\nrt 0, 
	&& k = \pm 7, \pm 7'.
	\end{aligned}
\end{align}
and thus from \eqref{eq:react-split}
\begin{align}
	\label{eq:react-conv}
	\begin{aligned}
		n^{-1}\nR_k(t) &\nrt  \int_0^t \bar \l_k(Z_S(s))ds, 
		&& k  = 0, \pm 1, 2, \pm 3, 4, \pm 5, 6, \pm 8, \\
		n^{-1/2}\nR_k(t) &\nrt  \int_0^t \bar \l_k(Z_S(s))ds, 
		&& k = \pm 7, \pm 7'.
	\end{aligned}
\end{align}

%Now write the equation of $\nZ_{A_1}$ and $\nZ_{A_2}$ as 
%\begin{align*}
%\nZ_{A_1}(t) &= \nZ_{A_1}(0) + n^{-1}\nmart_{A_1}(t) +  \kappa_0t - \nmeanreact_1(t) + \nmeanreact_{-1}(t) - \nmeanreact_3(t) + \nmeanreact_{-3}(t), \\
%		\nZ_{A_2}(t) &= \nZ_{A_2}(0) + n^{-1}\nmart_{A_2}(t) + \nmeanreact_2(t) + \nmeanreact_4(t) - \nmeanreact_5(t) + \nmeanreact_{-5}(t) - 2\nmeanreact_8(t) + 2\nmeanreact_{-8}(t), \\
%\end{align*}
%where the martingales $\nmart_{A_1}$ and $\nmart_{A_2}$ are given by
%\begin{align*}
%\nmart_{A_1}(t)= &\ \tnR_0(t) - \tnR_1(t) + \tnR_{-1}(t) - \tnR_3(t) + \tnR_{-3}(t)\\
%\nmart_{A_2}(t)= &\ \tnR_2(t) + \tnR_4(t) - \tnR_5(t) + \tnR_{-5}(t) - 2\tnR_8(t) + 2\tnR_{-8}(t)
%\end{align*}
%with $\tnR_k$ defined in \eqref{eq:react-process-1}.
%
%Thus for $i=1,2$,  $\<n^{-1}\nmart_{A_i}\>_T \nrt 0$  which implies $n^{-1}\sup_{t\leq T}|\nmart_{A_i}| \nrt 0$. 

It immediately follows from the equation for $(\nZ_{A_1}, \nZ_{A_2})$ in \eqref{eq:full_gly_scaled_eq} that the limit point $(Z_{A_1}, Z_{A_2})$ satisfies the ODE
\begin{align}
	\label{eq:limit-pt-ode}
	\begin{aligned}
	Z_{A_1}(t) &= Z_{A_1}(0) +   \kappa_0t - \int_0^t  \bar \l_1(Z_S(s))ds + \int_0^t  \bar \l_{-1}(Z_S(s))ds - \int_0^t  \bar \l_{3}(Z_S(s))ds + \int_0^t  \bar \l_{-3}(Z_S(s))ds, \\
	&= Z_{A_1}(0) +   \kappa_0t -\kappa_{1}\int_0^t Z_{A_1}(s) \mean_{E_1}(Z_S(s))ds+ \kappa_{-1}\int_0^t \mean_{E_1A_1}(Z_S(s))ds\\
	& \hs{0.6cm} - \kappa_{3}\int_0^t   \mean_{E_1^*}(Z_S(s))Z_{A_1}(s)ds +\kappa_{-3} \int_0^t  \mean_{E_1^*A_1}(Z_S(s))ds\\
	Z_{A_2}(t) &= Z_{A_2}(0) +\int_0^t  \bar \l_{2}(Z_S(s))ds+\int_0^t  \bar \l_{4}(Z_S(s))ds -\int_0^t  \bar \l_{5}(Z_S(s))ds+\int_0^t  \bar \l_{-5}(Z_S(s))ds,\\
    & \hs{0.6cm} -2\int_0^t  \bar \l_{8}(Z_S(s))ds+2\int_0^t  \bar \l_{-8}(Z_S(s))ds\\
    &= Z_{A_2}(0) +  \kappa_2\int_0^t  \mean_{E_1A_1}(Z_S(s))ds+ \kappa_4\int_0^t  \mean_{E_1^*A_1}(Z_S(s))ds - \int_0^t  \kappa_5 Z_{A_2}(s)\mean_{E_2}(Z_S(s))ds\\
	& \hs{0.6cm} + \kappa_{-5}\int_0^t  \mean_{E_2A_2}(Z_S(s))ds - 2\kappa_8\int_0^t \big(Z_{A_2}(s)\big)^2ds + 2\kappa_{-8}Z_{A_4}(0)\int_0^t   \mean_{A_3}(Z_S(s))ds, 
	\end{aligned}
\end{align}
where for $z_S$ in the state-space of the process $Z_S$, we define the coordinate-wise means $\mean_{\alpha}(z_S)$ and first order mixed moments  $\cov_{\alpha, \alpha'}(z_S)$ with $\alpha, \alpha' \in F=\{A_3, E_1, E_1^*, E_1A_1, E_1^*A_1, E_2, E_2A_2\}$  as
\begin{align*}
	\mean_{\alpha}(z_S) \dfeq \int_{[0,\infty)^7} z_{\alpha} \inv_{z_S}(dz_F),\quad  \cov_{\alpha,\alpha'}(z_S) \dfeq  \int_{[0,\infty)^7} z_{\alpha}z_{\alpha'} \inv_{z_S}(dz_F), \quad \alpha, \alpha' \in F.
\end{align*}

Now to claim that $\nZ_S \equiv (\nZ_{A_1}, \nZ_{A_2}, \nZ_{A_4})$ converges to $Z_S \equiv (Z_{A_1}, Z_{A_2}, Z_{A_4}(0))$ along the full sequence, we need to show that  $Z_S$ is uniquely defined. However, note that for a given $z_S$ in the state space of $Z_S$, $\fgen_{z_S}$ does not admit a unique invariant distribution. We show that despite this, the $\mean_\alpha(\cdot), \ \alpha \in  F$ have unique closed-form expressions.

We start by showing that the probability measure $\inv_{z_S}$ in a limit point $\occ(dz_F \times ds) = \inv_{Z_S(s)}(dz_F)ds$ satisfies additional constraint besides \eqref{eq:stat-dist-char}. To this end, notice that taking $n\rt \infty$ in \eqref{eq:react7+7'} and using \eqref{eq:conv-int+centrd-react} gives for any $t \in [0,T]$,
\begin{align*}
\kappa_7\int_0^t  \cov_{E_1, A_3}(Z_S(s))+\cov_{E_1A_1, A_3}(Z_S(s))\ ds= \kappa_{-7} \int_0^t \mean_{E_1^*}(Z_S(s))+\mean_{E_1^*A_1}(Z_S(s))\ ds;
\end{align*}
that is, for $z_S$ in the state space of $Z_S$, $\inv_{z_S}$ must be such that
\begin{align}\label{eq:stat-dist-constrnt}
	\kappa_7\lf( \cov_{E_1, A_3}(z_S)+\cov_{E_1A_1, A_3}(z_S)\ri)= \kappa_{-7}  \lf(\mean_{E_1^*}(z_S)+\mean_{E_1^*A_1}(z_S)\ri).
\end{align}
 Furthermore for any $t>0$, integrating both sides of the conservation laws in \eqref{eq:cons_laws} from $[0,t]$ and taking $n\rt \infty$ shows that 
 \begin{align}
        \begin{aligned}
          \mean_{E_1}(z_S)+\mean_{E_1A_1}(z_S)+\mean_{E_1^*}(z_S)+\mean_{E_1^*A_1}(z_S) =& \consconst_1, \\
        \mean_{E_2}(z_S) + \mean_{E_2A_2}(z_S)  =&\ \consconst_2.
        \end{aligned}
        \label{eq:cons_laws-stat}
    \end{align}
 
We now proceed to find the specific expressions for $\mean_{\alpha}(Z_S(s))$  needed for the closed form expressions of the above limiting ODE. To this end we use the fact that for a fixed $z_S \in [0,\infty)^3$ and any $g$ satisfying \eqref{eq:gen-testfn-growth}, \eqref{eq:stat-dist-char}  holds.

{\em Step 1:} Using \eqref{eq:stat-dist-char} with  $g$ given by $g(z_F)=z_{E_2}$  gives the  equation
\begin{align*}
-\kappa_5 \mean_{E_2}(z_S)z_{A_2}+(\kappa_{-5}+\kappa_6)\mean_{E_2A_2}(z_S)=&\ 0.
\end{align*}
Solving this together with the second equation in \eqref{eq:cons_laws-stat} gives
\begin{align}
\label{eq:mean-E2-E2A2}
\mean_{E_2}(z_S) = \f{\consconst_2(\kappa_{-5}+\kappa_6)}{(\kappa_{-5}+\kappa_6)+\kappa_5z_{A_2}}, \quad \mean_{E_2A_2}(z_S) =\f{\kappa_5 \mean_{E_2}(z_S)z_{A_2}}{\kappa_{-5}+\kappa_6} = \f{\consconst_2\kappa_{5}z_{A_2}}{(\kappa_{-5}+\kappa_6)+\kappa_5z_{A_2}}.
\end{align}
Next taking  $g(z_F) = z_{A_3}$ in \eqref{eq:stat-dist-char}, we get $\kappa_8\big(z_{A_2}\big)^2-\kappa_{-8}\mean_{A_3}(z_S)z_{A_4}  =0$, that is,
\begin{align}
\label{eq:mean} 
\mean_{A_3}(z_S) = \f{\kappa_8 \big(z_{A_2}\big)^2 }{\kappa_{-8}z_{A_4}}.
\end{align}
Now    \eqref{eq:stat-dist-char} with $g$ given by $g(z_F) = z_{A_3}(z_{E_1}+z_{E_1A_1})$, together with \eqref{eq:stat-dist-constrnt} and \eqref{eq:cons_laws-stat} gives
\begin{align*}
\kappa_8\lf(z_{A_2}\ri)^2\lf(\mean_{E_1}(z_S)+\mean_{E_1A_1}(z_S)\ri) =&\ \kappa_{-8}z_{A_4}\lf(\cov_{E_1,A_3}(z_S)+\cov_{E_1A_1, A_3}(z_S)\ri)\\
=& \f{\kappa_{-8}\kappa_{-7}}{\kappa_7}z_{A_4}\lf(J_1 - \lf(\mean_{E_1}(z_S)+\mean_{E_1A_1}(z_S)\ri)\ri).
\end{align*}
 Solving the equations, we get
\begin{align*}
\mean_{E_1}(z_S)+\mean_{E_1A_1}(z_S) = \f{\kappa_{-8}\kappa_{-7}\consconst_1 z_{A_4}}{\kappa_{8}\kappa_{7}(z_{A_2})^2+\kappa_{-8}\kappa_{-7}z_{A_4}} = \consconst_1 - \lf(\mean_{E_1^*}(z_S)+\mean_{E_1^*A_1}(z_S)\ri).
\end{align*}
Finally,    \eqref{eq:stat-dist-char} with $g$ given by $g(z_F) = z_{E_1}$ and $g(z_F) = z_{E_1^*}$  respectively give
\begin{align*}
-\kappa_1 \mean_{E_1}(z_S)z_{A_1}+(\kappa_{-1}+\kappa_{2})\mean_{E_1A_1}(z_S)=&\ 0,\\
-\kappa_3 \mean_{E^*_1}(z_S)z_{A_1}+(\kappa_{-3}+\kappa_{4})\mean_{E^*_1A_1}(z_S)=&\ 0.
\end{align*}
We thus get 
\begin{align}
\label{eq:mean-E1-E1*}
\begin{aligned}
\mean_{E_1}(z_S) =&\ \f{(\kappa_{-1}+\kappa_{2})\kappa_{-8}\kappa_{-7}\consconst_1 z_{A_4}}{(\kappa_{-1}+\kappa_{2}+\kappa_{1}z_{A_1})(\kappa_{8}\kappa_{7}(z_{A_2})^2+\kappa_{-8}\kappa_{-7}z_{A_4})},\\
\mean_{E_1A_1}(z_S)=& \f{\kappa_{1}\kappa_{-8}\kappa_{-7}\consconst_1 z_{A_4}z_{A_1}}{(\kappa_{-1}+\kappa_{2}+\kappa_{1}z_{A_1})(\kappa_{8}\kappa_{7}(z_{A_2})^2+\kappa_{-8}\kappa_{-7}z_{A_4})},\\
\mean_{E^*_1}(z_S) =&\ \f{(\kappa_{-3}+\kappa_{4})\kappa_{8}\kappa_{7}\consconst_1(z_{A_2})^2}{(\kappa_{-3}+\kappa_{4}+\kappa_{3}z_{A_1})(\kappa_{8}\kappa_{7}(z_{A_2})^2+\kappa_{-8}\kappa_{-7}z_{A_4})},\\
\mean_{E^*_1A_1}(z_S) =&\ \f{\kappa_{3}\kappa_{8}\kappa_{7}\consconst_1(z_{A_2})^2z_{A_1}}{(\kappa_{-3}+\kappa_{4}+\kappa_{3}z_{A_1})(\kappa_{8}\kappa_{7}(z_{A_2})^2+\kappa_{-8}\kappa_{-7}z_{A_4})}.
\end{aligned}
\end{align}
Plugging in the expressions for $\mean_{\alpha}$, $\alpha \in F$, into \eqref{eq:limit-pt-ode}, we see that it simplifies to the ODE \eqref{eq:FLLN_Glyco}.  Since this ODE admits a unique solution, $Z_S$ is uniquely defined, and the assertion follows.
%Notice that 
%\begin{align}\label{eq:ZV-split}
%            Z^{(n)}_S(t)  =&\ \nPhi_S(t)+ \nmart_St),
%        \end{align}
%where $\nPhi_S(t) =(\nPhi_1(t), \nPhi_2(t), \nPhi_4(t))$ with 
%\begin{align}\label{eq:Z-S-cont}
%\begin{aligned}
%\nPhi_1(t) =&\ Z_{A_1}^N(0)+\kappa_0t-\kappa_1 \int_0^tZ_{E_1}^N(s)Z_{A_1}^N(s)\ ds+\kappa_{-1} \int_0^tZ_{E_1A_1}^N(s)(s)\ ds\\
%&\  -\kappa_3 \int_0^tZ_{E*_1}^N(s)Z_{A_1}^N(s)\ ds +\kappa_{-3} \int_0^tZ_{E*_1}^N(s)Z_{A_1}^N(s)\ ds
%%
%\end{aligned}
%\end{align}

%\end{proof}

\subsection{Discussion of the reduced order ODE system}

\begin{figure}
\centering
\includegraphics[width=\linewidth]{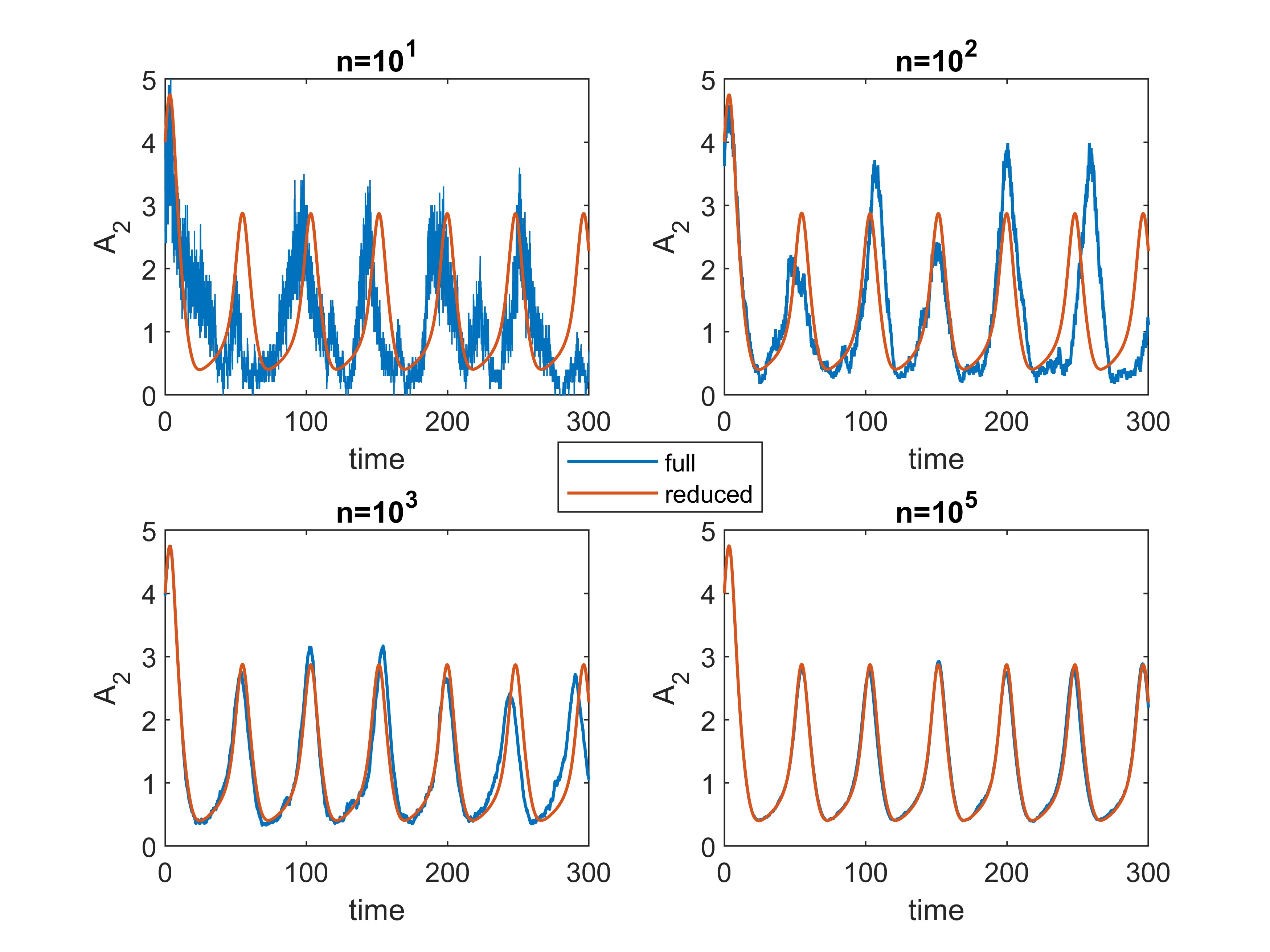}
\includegraphics[width=\linewidth]
{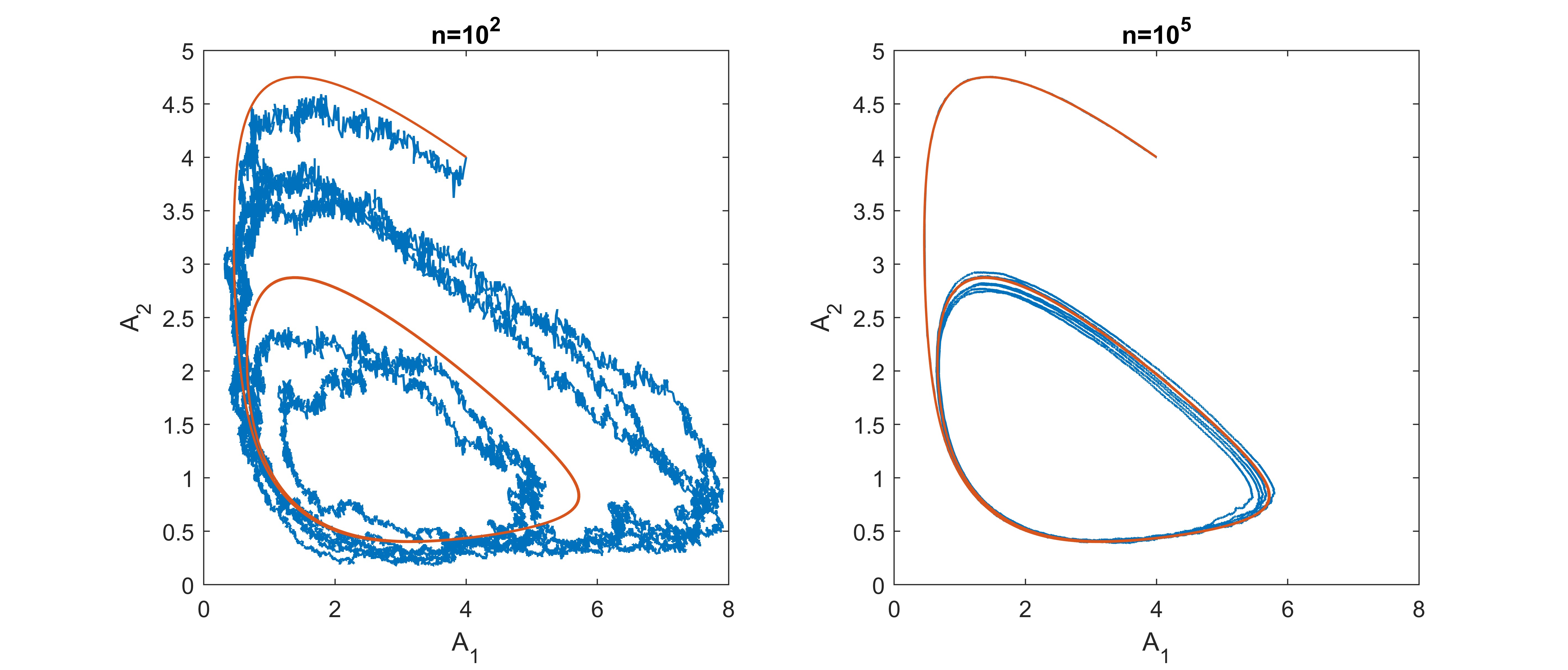}
\caption{Comparison between the behavior of the normalize molecular counts for species $A_2$ in the full CTMC model and the reduced ODE model with varying scaling parameter values $n$.}
\label{fig:stoch_n}
\end{figure}

We compare the normalized abundance of $A_2$ in the full stochastic model with that of the reduced ODE model. The normalized stochastic dynamics follow \eqref{eq:full_gly_scaled_eq}, while the reduced dynamics are governed by \eqref{eq:FLLN_Glyco}-\eqref{eq:rates}.
By Theorem~\ref{thm:FLLN-Glyco}, the scaled stochastic process converges to the reduced ODE solution as $n\to\infty$, implying consistency between two models for large $n$.

Figure \ref{fig:stoch_n} illustrates this convergence as the scaling parameter varies from $n=10$ to $n=10^5$. For $n=10$, the stochastic trajectory of $A_2$ (blue) exhibits large fluctuations and noticeable phase shifts relative to the ODE trajectory (red). However, for $n=10^2$, $10^3$, and $10^5$, the stochastic trajectories closely track the ODE trajectory. 
The bottom panel of Figure \ref{fig:stoch_n} presents phase-plane comparisons for $n=10^2$ and $n=10^5$: the reduced ODE trajectory (red) converges toward the limit cycle over time, while the corresponding stochastic trajectory (blue) oscillates around it. As $n$ increases, the stochastic trajectories remain increasingly close to the ODE limit cycle.

Overall, these results indicate that the reduced ODE model provides an accurate approximation of the full stochastic dynamics for large--and even moderately large--scaling values, with reliable agreement already evident at $n=10^2$.

\section{Parameter estimation}\label{sec:para}
Let $\bm{\kappa} = ( \kappa_0,\kappa_{\pm 1}, \kappa_2, \kappa_{\pm 3}, \kappa_4, \kappa_{\pm 5}, \kappa_6, \kappa_{\pm 7}, \kappa_{\pm 8}) \in (0, \infty)^{ 14}$ be the vector of (scaled) reaction rates of the original (scaled) system $\nZ$ in \eqref{eq:full_gly_scaled_eq}. In the full stochastic description of the glycolytic pathway, the reaction rates $\kappa_i^{(n)}$ govern both the slow and fast reactions. Learning these rates directly from data would in principle characterize the entire biochemical network, but in practice the fast variables are unobserved, and the data consist only of the slow components $(\nZ_{A_1}, \nZ_{A_2})$. Because the full system is high-dimensional and exhibits multiscale stiffness, estimating all $\kappa_i^{(n)}$ from such partial observations is infeasible. The functional law of large numbers (averaging principle) in \Cref{thm:FLLN-Glyco} justifies the reduced-order model \eqref{eq:FLLN_Glyco} as an asymptotically accurate description of the slow dynamics. Consequently, statistical inference can be meaningfully performed at the level of this reduced order model, where the effective parameter vector 
\begin{align}\label{eq:theta-form}
\theta \equiv \mathcal{T}(\bm{\kappa}) \dfeq (\kappa_0, K_1, K_{M_1}, K^\star_{M_1}, K_{M_2}, \consconst^{\bullet}_1, \consconst^{\star}_1, \consconst^{\bullet}_2)
\end{align}
encodes the aggregate influence of the underlying reaction rates. Here, we assumed that $\consconst_1, \consconst_2$ are deterministic, so that $\theta$ is deterministic parameter. Let $\Theta \subset  (0,\infty)^{8}$ be the parameter space for $\theta$. The function $\mathcal{T}: (0,\infty)^{ 14} \rt \Theta $ signifies that the effective parameter $\theta$ is a (deterministic)  dimension-reducing transformation of the full  parameter vector $\bm{\kappa}$. Given a trajectory of the slow components $(\nZ_{A_1}, \nZ_{A_2})$ over the interval $[0,T]$, our goal in this section is to establish the accuracy of the estimator $\hat{\theta}^{(n)}$ computed from the reduced-order model. Mathematically, this amounts to proving the consistency of $\hat{\theta}^{(n)}$ as $n \to \infty$.

We now formalize the mathematical framework for consistency in our case and state the precise result. First, for a fixed $\theta$ having the form \eqref{eq:theta-form}, we denote the solution of the ODE \eqref{eq:FLLN_Glyco} with initial condition $z_{S,0} =(z_{A_1,0}, z_{A_2,0}, z_{A_4,0}) \in [0,\infty)^3$ as 
$$Z_S(z_{S,0}, \theta,\cdot) \equiv \Big(Z_{A_1}(z_{S,0}, \theta, \cdot), Z_{A_2}(z_{S,0}, \theta, \cdot), Z_{A_4}(z_{S,0}, \theta, \cdot)\equiv z_{A_4,0}\Big)$$ 
to highlight its dependence on the parameter $\theta$ and the initial condition $z_{S,0}$.

Next, denote the  probability measure on  $(\Om, \cal{F})$ by $\PP_{\bm{\kappa}}$ to emphasize its dependence on the full parameter vector $\bm{\kappa}$. In other words, under $\PP_{\bm{\kappa}}$, the process $\nZ$ satisfies \eqref{eq:full_gly_scaled_eq} with the reaction rates given by $\bm{\kappa},$ and if under $\PP_{\bm{\kappa}}$ \Cref{assum:main} holds,  \Cref{thm:FLLN-Glyco} guarantees that $\nZ_S$ converges (in distribution) to $Z_S(Z_S(0), \theta,\cdot)$  with $\theta = \cal{T}(\bm{\kappa}).$

\begin{remark}
Even though we assumed that $\consconst_1, \consconst_2$ (limit of the conservation constants in \eqref{eq:cons_laws}) are deterministic, $Z_S(\theta,\cdot)$ can be random (that is, a stochastic process) as the limiting initial values $Z_S(0)$ are assumed to be potentially random. Note that in this case the randomness enters the ODE \eqref{eq:FLLN_Glyco} only through the initial values $Z_S(0)$.
\end{remark}

\begin{definition}\label{def:consistency}
An estimator $\hat \theta^{(n)}$ is consistent for  $\theta^* \in \Theta$, if $\hat \theta^{(n)} \stackrel{\PP_{\bm{\kappa}^*}}\Rt \theta^*$ for any $\bm{\kappa}^* \in \cal{T}^{-1}(\theta^*) = \{\bm{\kappa}^*: \cal{T}(\bm{\kappa}^*) = \theta^*\}$.
\end{definition}
Our data set comprises of $\mathcal{D}^{(n)}_T \equiv \{(\nZ_{A_1}(t), \nZ_{A_2}(t), \nZ_{A_4}(0)): t \in [0,T]\}$.
To estimate $\theta$, we introduce the loss function
\begin{align}
\label{eq:loss-fun-glyco}
\loss^{(n)}_T(\theta \mid \mathcal{D}^{(n)}_T) = \int_0^T \lf|\nZ_{A_1}(t) - Z_{A_1}\lf(\nZ_{S}(0), \theta,t\ri)\ri|^2+\lf|\nZ_{A_2}(t) - Z_{A_2}\lf(\nZ_{S}(0),\theta,t\ri)\ri|^2\ dt,
\end{align}
where $Z_S(\nZ_{S}(0), \theta,\cdot) \equiv \Big(Z_{A_1}(\nZ_{S}(0), \theta, \cdot), Z_{A_2}(\nZ_{S}(0), \theta, \cdot),  \nZ_{A_4}(0)\Big)$ satisfies \eqref{eq:FLLN_Glyco} with initial condition $\nZ_S(0).$
The loss $\loss^{(n)}_T(\theta \mid \mathcal{D}^{(n)}_T)$ quantifies the mismatch between the observed evolution of the slow components of the original system and the  prediction by the reduced-order model across the  time horizon $[0,T]$.  We estimate $\theta$  as
\begin{align}\label{eq:def-est}
\hat \theta^{(n)} \equiv \hat \theta^{(n)}_T(\mathcal{D}^{(n)}_T) = \argmin_{\theta \in \Theta} \loss^{(n)}_T(\theta \mid \mathcal{D}^{(n)}_T).
\end{align}
Thus $\hat \theta^{(n)}$  gives the parameter value for which the reduced-order model best reproduces the observed trajectory. 

Before, we state our consistency result, we need to ensure that the parameter $\theta$ of the ODE system \eqref{eq:FLLN_Glyco} is identifiable in an appropriate sense.  Write $\theta =(\kappa_0, \theta^{(1)}, \theta^{(2)})$ where $\theta^{(1)} = (K_1, K_{M_1}, K^\star_{M_1}, \consconst^{\bullet}_1, \consconst^{\star}_1)$ and $\theta^{(2)}=( K_{M_2}, \consconst^{\bullet}_2)$. Notice that in \Cref{thm:FLLN-Glyco} $f$ depends only on $\theta^{(1)}$ and $g$ only depends on $\theta^{(2)}$. To show the explicit dependence of the functions $f$ and $g$ of \eqref{eq:FLLN-ODE-fun} on the parameter $(\kappa_0, \theta^{(1)}, \theta^{(2)})$, we now write them as $f_{\theta^{(1)}}$ and $g_{\theta^{(2)}}$. 

\begin{remark}\label{rem:fg-lip-bdd} Suppose that $\Theta \subset (0,\infty)^8$ is compact. 
It is easy to see that the mappings $(\theta^{(1)}, z_S=(z_{A_1}, z_{A_2}, z_{A_4})) \rt f_{\theta^{(1)}}(z_S)$ and $(\theta^{(2)},  z_{A_2}) \rt g_{\theta^{(2)}}(z_{A_2})$ are Lipschitz continuous and bounded for $z_{A_4}$ bounded away from zero. Specifically, for $\delta_0>0$, there exist constants $\const, L > 0$ such that
\begin{align*}
 &\sup_{\theta \in \Theta} \lf(\sup_{z_S=(z_{A_1}, z_{A_2}, z_{A_4}) \in [0,\infty)^2\times[\delta_0, \infty)}f_{\theta^{(1)}}(z_S) \vee \sup_{z_{A_2} \in [0,\infty)} g_{\theta^{(2)}}(z_{A_2})\ri) \leq\ \const,\\
 &|f_{\theta^{(1)}}(z_{S}) - f_{\tilde\theta^{(1)}}(\tilde z_{S})|
 \leq\ \  L \lf(\|\theta^{(1)} - \tilde \theta^{(1)}\|_1 +\|z_{S}-\tilde z_{S}\|_1\ri),\\
 &|g_{\theta^{(2)}}(z_{A_2}) - g_{\tilde\theta^{(2)}}(\tilde z_{A_2})| \leq\ L \lf(\|\theta^{(2)} - \tilde \theta^{(2)}\|_1 +\|z_{A_2}-\tilde z_{A_2}\|_1\ri).
\end{align*}
It now easily follows that for some constants $\const' \equiv \const'(z_{S,0}) > 0$
\begin{align}\label{eq:Z-bounds}
\sup_{\theta \in \Theta}\sup_{t\leq T}Z_{A_i}(z_{S,0},\theta,t) \leq \const'(z_{S,0}) T, \quad i=1,2.
\end{align}
Furthermore, by Gronwall's inequality, it follows that the mapping $(z_{S,0},\theta) \in [0,\infty)^2\times[\delta_0, \infty)\times \Theta \rt (Z_{A_1}(z_{S,0}, \theta,t), Z_{A_2}(z_{S,0}, \theta,t))$ is Lipschitz continuous, that is, there exists a  constant $L'(T)$ such that for $z_{S,0}=(z_{A_1,0}, z_{A_2,0}, z_{A_4,0}) \in [0,\infty)^2\times[\delta_0, \infty)$ and $\theta\in \Theta$,
\begin{align}\label{eq:Z-para-initial-Lip}
\sup_{t\leq T}\sum_{i=1}^2|Z_{A_i}(z_{S,0},\theta,t) - Z_{A_i}(\tilde z_{S,0}, \tilde\theta,t)|   \leq L'(T)(\|z_{S,0}-\tilde z_{S,0}\|_1+ \|\theta-\tilde \theta\|_1).
\end{align}

\end{remark}

%\begin{align}
%	\label{eq:FLLN-ODE-fun}
%	\begin{aligned}
%		f_{\theta^{(1)}}(z_{A_1},z_{A_2}, z_{A_4}) &=\frac{1}{K_1z_{A_4}+z_{A_2}^2}\left[\frac{\consconst^{\bullet}_1 K_1z_{A_1}z_{A_4}}{K_{M_1}+z_{A_1}}+\frac{\consconst^\star_1 z_{A_1}z_{A_2}^2}{K_{M_1}^\star+z_{A_1}}\right], \quad  %F_{(\kappa_0, \theta^{(1)})} = \kappa_0 -f_{\theta^{(1)}},
%		g_{\theta^{(2)}}(z_{A_2}) = \frac{\consconst^{\bullet}_2 z_{A_2}}{K_{M_2}+z_{A_2}},
%	\end{aligned}
%\end{align}	

\begin{definition}
Let $\theta = (\kappa_{0}, \theta^{(1)}, \theta^{(2)}) \in \Theta$ and fix $z_{A_4} > 0$.
A set $U \subset [0,\infty)^2$ is said to identify $\theta$,  if
\begin{align}\label{eq:f-g-para-equal} 
\begin{aligned}
		 % F_{(\kappa_{0,0}, \theta^{(1)}_0)}\left(z_{A_1}, z_{A_2}, z_{A_4}\right) \equiv \kappa_{0,0} - f_{\theta^{(1)}_0}\left(z_{A_1}, z_{A_2}, z_{A_4}\right) =&\ \tilde \kappa_0 - f_{\tilde\theta^{(1)}}\left(z_{A_1}, z_{A_2}, z_{A_4}\right)\equiv F_{(\tilde \kappa_{0}, \tilde\theta^{(1)}_0)}\left(z_{A_1}, z_{A_2}, z_{A_4}\right) \\
		  \kappa_{0} - f_{\theta^{(1)}}\left(z_{A_1}, z_{A_2}, z_{A_4}\right) =&\ \tilde \kappa_0 - f_{\tilde\theta^{(1)}}\left(z_{A_1}, z_{A_2}, z_{A_4}\right)\\ f_{\theta^{(1)}}\left(z_{A_1}, z_{A_2}, z_{A_4}\right) - g_{\theta^{(2)}}\left(z_{A_2}\right) =&\ f_{\tilde \theta^{(1)}}\left(z_{A_1}, z_{A_2}, z_{A_4}\right) - g_{\tilde \theta^{(2)}}\left(z_{A_2}\right)
 \end{aligned} 
 \end{align}
for all $(z_{A_1}, z_{A_2}) \in U$ and for some $\tilde\theta = (\tilde\kappa_0, \tilde\theta^{(1)}, \tilde\theta^{(2)}) \in \Theta$
implies $\tilde\theta = \theta$.

\end{definition}	

We now discuss sufficient conditions under which a set $U$ uniquely identifies $\theta$.
It is straightforward to verify, using the properties of analytic functions, that $U$ identifies $\theta$ whenever it has nonempty interior.
However, in our setting $U$  corresponds to the trajectory-set of an ODE and is typically  a curve in $\R^2$, and therefore has empty interior. Below we provide a condition that is easier to check and is typically satisfied by a trajectory of \eqref{eq:FLLN_Glyco}.

We introduce the following notation: for variables $x,y \in [0,\infty)$, define the row vector $V(x,y)$ as
\begin{align}\label{eq:Vand-vec}
V(x,y) \dfeq & \left(
\begin{array}{llllllllllll} % [c] center columns; [r] or [l] also available
1 & y & y^2 & x & xy & xy^2 & x^2 & x^2y & x^2y^2 & x^3 &  x^3 y & x^3y^2
\end{array} 
\right). 
\end{align}

% for column vectors $\bm{u}_{i, d} = (u_{i}(1), u_{i}(2), \hdots, u_{i}(d))^\top \in [0,\infty)^d,  i=1,2$, and exponents $p,q\geq 0$, denote $\bm{u}^p_{1,d}\bm{u}^q_{2,d}$ as the column vector of element-wise products, $u^p_{1}(k)u^q_{2}(k)$; that is, 
%$$\bm{u}^p_{1,d}\bm{u}^q_{2,d} = \lf(u^p_{1}(1)u^q_{2}(1), \hdots, u^p_{1}(d)u_{2}(d)\ri)^\top.$$

\begin{lemma} \label{lem:ident-0}
	Let $z_{A_4} \in (0,\infty)$ be fixed,  and  $\theta = (\kappa_{0}, \theta^{(1)}, \theta^{(2)}) \in \Theta\subset (0,\infty)^8$. Suppose $U \subset [0,\infty)^2$ contains $(0,0)$ and there exists a subset  $\{(z_{A_1}(k), z_{A_2}(k)): k=1,2,\hdots,12\} \subset U $ such that $z_{A_1}(k) > 0, \ k=1,\hdots,12$, and there exist $k_0, \ k'_0$ such that $z_{A_2}(k_0) \neq z_{A_2}(k'_0) > 0$. 
    % Let $\cal{A} = \{\alpha=(i,j):i=0,1,2,3, j=0,1,2\}$ be the ordered set of multi-indices, and 
    Define the $12\times 12$  generalized Vandermonde matrix 
$$\bm{V} \dfeq \begin{pmatrix} \vs*{.1cm}
V(z_{A_1}(1),z^2_{A_2}(1))\\ 
V(z_{A_1}(2),z^2_{A_2}(2))\\ 
\vdots\\
V(z_{A_1}(12),z^2_{A_2}(12))\\
\end{pmatrix}.	
$$
%$\bm{V}=\big[\, \bm{v}_\alpha \,\big]_{\alpha\in\mathcal{A}}$ with the column corresponding to $\alpha = (i,j)$ given by
%	$$\bm{v}_\alpha = \bm{z}^i_{A_1, 12}\bm{z}^{2j}_{A_2, 12} = \Big(z^i_{A_1}(1)z^{2j}_{A_2}(1), \hdots, z^i_{A_1}(12)z^{2j}_{A_2}(12)\Big)^\top , \quad \alpha = (i,j).$$
Assume that $\bm{V}$ is invertible. Then $U$ identifies $\theta$.
\end{lemma}

\begin{proof}
	Let  $\tilde\theta = (\tilde\kappa_0, \tilde\theta^{(1)}, \tilde\theta^{(2)}) \in \Theta$ be such that \eqref{eq:f-g-para-equal} holds for all $(z_{A_1}, z_{A_2}) \in U$. Taking $(z_{A_1}, z_{A_2}) =(0,0)$, we see $ \kappa_{0} = \tilde \kappa_{0}$. This shows for all $(z_{A_1}, z_{A_2}) \in U$
	\begin{align}\label{eq:f-g}
		f_{ \theta^{(1)}}\left(z_{A_1}, z_{A_2}, z_{A_4}\right) = f_{\tilde \theta^{(1)}}\left(z_{A_1}, z_{A_2}, z_{A_4}\right), \quad g_{ \theta^{(2)}}\left(z_{A_2}\right) = g_{\tilde \theta^{(2)}}\left(z_{A_2}\right) 
	\end{align}
For any  $(z_{A_1}, z_{A_2}) \in U$ with $z_{A_2}>0$, cancelling $z_{A_2}$ from both sides of the second equation gives
\begin{align*}
	 \frac{\consconst^{\bullet}_{2}}{K_{M_2}+z_{A_2}} =  \frac{\tilde \consconst^{\bullet}_2}{\tilde K_{M_2}+z_{A_2}}
\end{align*}
which implies
\begin{align*}
\consconst^{\bullet}_{2}\tilde K_{M_2} - \tilde \consconst^{\bullet}_2K_{M_2}     + (\consconst^{\bullet}_{2} - \tilde \consconst^{\bullet}_2)z_{A_2}=0.
\end{align*}
By the hypothesis this linear equation has two distinct solutions $z_{A_2}(k_0), z_{A_2}(k'_0)$. Therefore, it must be identically $0$, which readily implies that
$\theta^{(2)} \equiv (\consconst^{\bullet}_{2},  K_{M_2}) =  \tilde \theta^{(2)} \equiv (\tilde \consconst^{\bullet}_2, \tilde K_{M_2})$.

Next, for any  $(z_{A_1}, z_{A_2}) \in U$ with $z_{A_1}>$0, cancelling $z_{A_1}$ from the both sides of the first equation of \eqref{eq:f-g} gives
\begin{align}\label{eq:f-mult-poly}
Q_{\theta^{(1)}, \tilde \theta^{(1)}}(z_{A_1}, z^2_{A_2}) = Q_{\tilde \theta^{(1)}, { \theta^{(1)}}}(z_{A_1}, z^2_{A_2}),
\end{align}
where for $\theta^{(1)} = (K_1, K_{M_1}, K^\star_{M_1}, \consconst^{\bullet}_1, \consconst^{\star}_1)$, $\tilde \theta^{(1)} = (\tilde K_1, \tilde K_{M_1}, \tilde K^\star_{M_1}, \tilde \consconst^{\bullet}_1, \tilde \consconst^{\star}_1)$ and the fixed $z_{A_4}>0$ the polynomial $Q_{\theta^{(1)}, \tilde \theta^{(1)}}(x,y) \dfeq V(x,y)\gamma(\theta^{(1)}, \tilde \theta^{(1)})$
with the row vector $V(x,y)$ given by \eqref{eq:Vand-vec}
and the column vector $\gamma( \theta^{(1)}, \tilde \theta^{(1)} )$ given by
\begin{align*}
	\gamma(\theta^{(1)}, \tilde \theta^{(1)}) &\dfeq\ \Bigg(J_1^\bullet K_1 K^\star_{M_1}\,\tilde K_1\,\tilde K_{M_1}\,\tilde K^\star_{M_1}z^2_{A_4}, \ \tilde K_{M_1}\,\tilde K^\star_{M_1}\,\big(J_1^\bullet K_1 K^\star_{M_1} + J_1^\star K_{M_1}\,\tilde K_1\big)\, z_{A_4}, \\
	&  J_1^\star K_{M_1}\,\tilde K_{M_1}\,\tilde K^\star_{M_1}, \ J_1^\bullet K_1 \tilde K_1\,(K^\star_{M_1}\,\tilde K_{M_1} + K^\star_{M_1}\,\tilde K^\star_{M_1} + \tilde K_{M_1}\,\tilde K^\star_{M_1})\, z^2_{A_4}, \\
	&\  \Big(J_1^\bullet K_1 K^\star_{M_1}\,\tilde K_{M_1}
	+ J_1^\bullet K_1 K^\star_{M_1}\,\tilde K^\star_{M_1}
	+ J_1^\bullet K_1\,\tilde K_{M_1}\,\tilde K^\star_{M_1}	+ J_1^\star K_{M_1}\,\tilde K_1\,\tilde K_{M_1}
	+ J_1^\star K_{M_1}\,\tilde K_1\,\tilde K^\star_{M_1}\\
&\	+ J_1^\star\,\tilde K_1\,\tilde K_{M_1}\,\tilde K^\star_{M_1}
	\Big)\,z_{A_4}, \  J_1^\star\,(K_{M_1}\,\tilde K_{M_1} + K_{M_1}\,\tilde K^\star_{M_1} + \tilde K_{M_1}\,\tilde K^\star_{M_1}), \\
	&\ J_1^\bullet K_1 \tilde K_1\,(K^\star_{M_1} + \tilde K_{M_1} + \tilde K^\star_{M_1})\,z^2_{A_4}, \\
	&\  \Big(
	J_1^\bullet K_1 K^\star_{M_1}
	+ J_1^\bullet K_1\,\tilde K_{M_1}
	+ J_1^\bullet K_1\,\tilde K^\star_{M_1}
	+ J_1^\star K_{M_1}\,\tilde K_1
	+ J_1^\star\,\tilde K_1\,\tilde K_{M_1}
	+ J_1^\star\,\tilde K_1\,\tilde K^\star_{M_1}
	\Big)\,z_{A_4},\\
	&\ J_1^\star\,(K_{M_1} + \tilde K_{M_1} + \tilde K^\star_{M_1}), \  J_1^\bullet K_1 \tilde K_1 \,z^2_{A_4}, \  (J_1^\bullet K_1 + J_1^\star \tilde K_1)\,z_{A_4}, \  J_1^\star \Bigg)^\top.
\end{align*}
Applying  \eqref{eq:f-mult-poly} for each $(z_{A_1}(k), z_{A_2}(k)) \in U$ gives the equation
\begin{align*}
	\bm{V} \gamma(\theta^{(1)}, \tilde \theta^{(1)}) = \bm{V} \gamma(\tilde\theta^{(1)}, \theta^{(1)}).
\end{align*}
Since $\bm{V}$ is assumed to be invertible, we get $ \gamma(\theta^{(1)}, \tilde \theta^{(1)}) =  \gamma(\tilde\theta^{(1)}, \theta^{(1)})$, and simple algebra shows that $\theta^{(1)} = \tilde \theta^{(1)}$.

\end{proof}

\begin{remark} Note that if $U$ contains twelve distinct points, the (generalized) Vandermonde matrix $\bm{V}$ formed from these points is almost always invertible. The condition that $(0,0) \in U$ is not essential and can be omitted, provided a higher-dimensional Vandermonde matrix is invertible. On the other hand, if some of the parameters of the reduced model are known, invertibility of a lower-dimensional Vandermonde matrix in \Cref{lem:ident-0} is sufficient for $U$ to identify the remaining parameters.
\end{remark}

\np
We are now ready to state the main result of this section.

\begin{theorem}\label{th:consistency} Let  $\Theta$ be compact, and fix $\theta^* \in \Theta$. Suppose \Cref{assum:main} holds under $\PP_{\bm{\kappa}^*}$ for any $\bm{\kappa}^* \in \cal{T}^{-1}(\theta^*)$ with $\consconst_1, \consconst_2$  deterministic, and $\nZ_{A_4}(0)\geq \delta_0>0$, where $\delta_0$ is as in  \Cref{assum:main}:\ref{item:initial-other}.
Assume that  the trajectory (orbit) of $(Z_{A_1}(\theta^*, \cdot), Z_{A_2}(\theta^*, \cdot))$ up to time $T$, 
\[
\mathcal{O}_T \dfeq \{(Z_{A_1}(Z_S(0), \theta^*, t), Z_{A_2}(Z_S(0),\theta^*, t)): t \in [0,T]\},
\]
identifies $\theta^*$, almost surely, with respect to $\PP_{\bm{\kappa}^*}$ for any $\bm{\kappa}^* \in \cal{T}^{-1}(\theta^*)$. 
Then the estimator $\hat{\theta}^{(n)}$, defined by \eqref{eq:def-est}, is consistent for $\theta^*$.
\end{theorem}

\begin{proof}[Proof of \Cref{th:consistency}] By \Cref{thm:FLLN-Glyco}, under $\PP_{\bm{\kappa}^*}$, as $n\rt \infty$, $\nZ_S \RT Z_S(\theta^*, \cdot)$ and  by the Skorokhod representation theorem, we can assume for the purpose of this proof, that $\nZ_S \rt Z_S(\theta^*, \cdot)$ a.s.;  that is,  $\PP_{\bm{\kappa}^*}$ - a.s.
\begin{align}
\label{eq:conv-Ai-unif}
|\nZ_{A_4}(0) - Z_{A_4}(0)| \nrt 0, \quad \sup_{t\leq T}\lf|\nZ_{A_i}(t) - Z_{A_i}\lf(Z_S(0),\theta^*,t\ri)\ri| \nrt 0, \quad i=1,2.
\end{align}
% It is easy to see from the Lipschitz continuity of $f_{\theta^{(1)}}, g_{\theta^{(2)}}$ (see \Cref{rem:fg-lip-bdd}) and Gronwall's inequality that  the mapping $z_{S,0} \in [0,\infty)^3 \rt Z_{S}(z_{S,0},\theta,\cdot) \in C([0,T], \R^3)$ is Lipschitz continuous with Lipschitz constant independent of $\theta \in \Theta$. 
By \eqref{eq:Z-para-initial-Lip}, it follows that
\begin{align}
\label{eq:conv-unif-ode-initial}
\sup_{\theta \in \Theta}\sup_{t\leq T}\lf|Z_{A_i}\lf(\nZ_S(0),\theta, t\ri) - Z_{A_i}\lf(Z_S(0),\theta,t\ri)\ri| \rt 0, \quad i=1,2.
\end{align}

\np
{\em Step 1:} We show that as $n\rt \infty$, $\loss^{(n)}_T(\cdot \mid \mathcal{D}^{(n)}_T) \nrt \loss_T(\cdot, \theta^*)$ in   $C(\Theta, [0,\infty))$ a.s.; that is,  $\PP_{\bm{\kappa}^*}$ - a.s.,
\begin{align}\label{eq:loss-conv}
\sup_{\theta \in \Theta} \lf| \loss^{(n)}_T(\theta \mid \mathcal{D}^{(n)}_T) - \loss_T(\theta, \theta^*)\ri|\nrt 0, 
\end{align}
where the limiting (random) function $\loss_T(\cdot, \theta^*) \in C((0,\infty)^8, [0,\infty))$ is defined by
\begin{align}
\label{eq:loss-det}
\loss_T(\theta, \theta^*) \dfeq \int_0^T |Z_{A_1}(Z_S(0),\theta^*,t) - Z_{A_1}(Z_S(0),\theta,t)|^2+|Z_{A_2}(Z_S(0),\theta^*,t) - Z_{A_2}(Z_S(0),\theta,t)|^2\ dt .
\end{align}
It easily follows from \Cref{rem:fg-lip-bdd} that the mapping $\theta \rt \loss_T(\theta, \theta^*)$ is Lipschitz continuous. Also note that because of \eqref{eq:conv-Ai-unif},  $\limsup_{n \rt \infty}\sup_{t\leq T}\nZ_{A_i}(t) < \infty,$ $\PP_{\bm{\kappa}^*}$ - a.s.
It now follows from \Cref{lem:aux-tight}:\ref{item:react-tight-2}, \eqref{eq:conv-Ai-unif}, \eqref{eq:conv-unif-ode-initial} and \eqref{eq:Z-bounds} that under $\PP_{\bm{\kappa}^*}$
\begin{align*}
\sup_{\theta \in \Theta}\Bigg|\int_0^T& \Big( \lf|\nZ_{A_i}(t) -  Z_{A_i}\lf(\nZ_S(0),\theta,t\ri)\ri|^2  - \lf|Z_{A_i}(Z_S(0),\theta^*, t) - Z_{A_i}\lf(Z_S(0),\theta,t\ri)\ri|^2\Big)\ dt \Bigg|\\
 &\leq\   \lf( \sup_{t\leq T}\nZ_{A_i}(t)
+ \sup_{\theta \in \Theta}\sup_{t\leq T}\lf(Z_{A_i}(Z_S(0),\theta,t)+Z_{A_i}\lf(\nZ_S(0),\theta,t\ri)\ri) + \sup_{t\leq T}Z_{A_i}(Z_S(0), \theta^*,t)\ri)  \\
& \hs*{.5cm} \times \lf(\sup_{t\leq T}\lf|\nZ_{A_i}(t)  - Z_{A_i}(Z_S(0),\theta^*,t)\ri|+\sup_{\theta \in \Theta}\sup_{t\leq T}|Z_{A_i}(\nZ_S(0),\theta, t) - Z_{A_i}(Z_S(0),\theta,t)|\ri) T\\
&\nrt \  0, \quad a.s.,
\end{align*}
which readily implies \eqref{eq:loss-conv}.

{\em Step 2:} Let $\om$ be such that \eqref{eq:conv-Ai-unif} holds, and $\cal{O}_T \equiv \cal{O}_T(\om)$ identifies $\theta^*$. We now observe that $\loss_T(\cdot, \theta^*)(\om)$ defined in \eqref{eq:loss-det} has a unique minimum at $\theta^*$. For  convenience, we now suppress the fixed $\om$ in the notation. 

It is clear that $\theta^*$ is a minimizer, $ \loss_T(\theta^*, \theta^*)=0$. Now suppose there exists a $\tilde\theta$ such that 
\begin{align}
\label{eq:loss-min}
\loss_T(\tilde\theta, \theta^*)=0.
\end{align}
 We show that $\tilde\theta = \theta^*.$ To this end, notice that \eqref{eq:loss-min} and continuity of the mappings $t \in [0,T] \rt Z_{A_i}(Z_S(0),\theta, t), i=1,2, \theta \in (0,\infty)$ imply 
 \begin{align} \label{eq:Z-para-traj}
 Z_{A_1}(\theta^*, \cdot) \equiv  Z_{A_1}(\tilde\theta, \cdot), \quad  Z_{A_2}(\theta^*, \cdot) \equiv  Z_{A_2}(\tilde\theta, \cdot), \quad \text{ on } \ [0,T].
 \end{align}
 For notational convenience, we suppress the dependence on the initial position $Z_S(0)$ in $Z_{A_i}$ in the above equation.
 But because of continuity of the mappings $(z_{A_1}, z_{A_2}) \rt f_\theta(z_{A_1}, z_{A_2}, z_{A_4})$ ($z_{A_4}$ fixed) and $z_{A_2} \rt g_\theta(z_{A_2})$, \eqref{eq:Z-para-traj} in turn shows that \eqref{eq:f-g-para-equal} holds  with $z_{A_4} = Z_{A_4}(0)$ and any $(z_{A_1}, z_{A_2})$ in  
\begin{equation*}
  \cal{O}_T \equiv\  \lf\{(Z_{A_1}(\theta^*, t), Z_{A_2}(\theta^*, t)): t \in [0,T]\ri\}
  =\   \{(Z_{A_1}(\tilde\theta, t), Z_{A_2}(\tilde\theta, t)): t \in [0,T]\},
 \end{equation*} 
where the second equality is also because of \eqref{eq:Z-para-traj}.  Since $\cal{O}_T \equiv \cal{O}_T(\om)$ identifies $\theta^*$ by hypothesis,  $\theta^*=\tilde\theta$.

{\em Step 3:} We finally show that $\hat \theta^{(n)} \nrt \theta^*$,  $\PP_{\bm{\kappa}^*}$-a.s. For this we present an $\om$ by $\om$ argument. Let $\Om_0$ be such that $\PP_{\bm{\kappa}^*}(\Om_0)=1$ and for all $\om \in \Om_0$, (a) the convergence in \eqref{eq:loss-conv} holds, and (b) $\cal{O}_T(\om)$ identifies $\theta^*$. 

 Now fix an $\om_0 \in \Om_0.$ 
 Since $\Theta$ is compact, the sequence $\{\hat \theta^{(n)}(\om_0)\}$  has a limit point $\theta' $, that is, there exists a subsequence $\{n_k \equiv n_k(\om_0)\}$ such that $\hat \theta^{(n_k)}(\om_0) \stackrel{k\rt \infty}\longrightarrow \theta'(\om_0)$. Note that, at this stage, we do not \emph{ apriori} know whether $\theta'(\om_0)$ is deterministic.  We will actually show that $\theta'(\om_0) \equiv \theta^* = \argmin_\theta  \loss_T(\theta, \theta^*)$. Since  $\theta^* $ is the unique minimizer of  $\loss_T(\cdot, \theta^*)$ by Step 2,  the limit is independent of the subsequence (and obviously deterministic); hence, $\hat \theta^{(n)}(\om_0)\nrt \theta^*$ along the full sequence.  We now work toward establishing this.
 
For notational simplicity, we suppress the dependence of $n_k$ on $\om_0$, while keeping it explicit elsewhere.
The  continuity of the function $\loss_T(\cdot, \theta^*)(\om_0): \Theta \rt [0,\infty)$  implies that 
\begin{align}\label{eq:loss-est-conv}
\loss_T\lf(\hat\theta^{(n_k)}(\om_0), \theta^*\ri)(\om_0) \stackrel{k\rt \infty} \Rt \loss_T(\theta'(\om_0),\theta^*)(\om_0).
\end{align}
Now write 
\begin{align*}
\loss^{(n_k)}_T\lf(\hat\theta^{(n_k)} \mid \mathcal{D}^{(n_k)}_T\ri)& = \loss^{(n_k)}_T\lf(\hat\theta^{(n_k)} \mid \mathcal{D}^{(n_k)}_T\ri) - \loss_T\lf(\hat\theta^{(n_k)}, \theta^*\ri) + \loss_T\lf(\hat\theta^{(n_k)}, \theta^*\ri),
 \end{align*} 
and observe that by \eqref{eq:loss-conv}
\begin{align*}
\Big|\loss^{(n_k)}_T&\lf(\hat\theta^{(n_k)}(\om_0) \mid \mathcal{D}^{(n_k)}_T(\om_0)\ri) - \loss_T  \lf(\hat\theta^{(n_k)}(\om_0), \theta^*\ri)(\om_0) \Big|\\
& \leq \  \sup_{\theta \in \Theta} \lf|\loss^{(n_k)}_T(\theta \mid \mathcal{D}^{(n_k)}_T(\om_0)) - \loss_T(\theta, \theta^*)(\om_0)\ri|  \stackrel{k\rt 0}\Rt 0.
\end{align*}
Because of \eqref{eq:loss-est-conv}, we conclude that 
$$ \loss^{(n_k)}_T\lf(\hat\theta^{(n_k)}(\om_0) \mid \mathcal{D}^{(n_k)}_T(\om_0)\ri) \stackrel{k\rt \infty}\Rt \loss_T(\theta'(\om_0),\theta^*)(\om_0). $$
On the other hand, using the fact that $\hat\theta^{(n)}(\om_0)$ is a minimizer of $\loss^{(n)}_T(\cdot \mid \mathcal{D}^{(n)}_T(\om_0))$ and \eqref{eq:loss-conv}, we have
$$\loss^{(n_k)}_T\lf(\hat\theta^{(n_k)}(\om_0)\mid \mathcal{D}^{(n_k)}_T(\om_0)\ri) \leq  \loss^{(n_k)}_T\lf(\theta^* \mid \mathcal{D}^{(n_k)}_T(\om_0)\ri)\stackrel{k\rt \infty}\Rt \loss_T(\theta^*,\theta^*)(\om_0)=0, \quad \PP_{\bm{\kappa}^*}\text{\ -\ } a.s $$
It follows that $\loss_T(\theta'(\om_0), \theta^*)(\om_0) =0$, and thus by Step 2, $\theta'(\om_0) = \theta^*$.
\end{proof}

\subsection{Numerical Experiments}

% Based on our rigorous proofs for identifiability and consistency of the model parameters $\theta$, we use the reduced ODE model in \eqref{eq:FLLN_Glyco}-\eqref{eq:FLLN-ODE-fun} to perform parameter estimation.
% In many experimental settings, only slowly varying species are observable, while fast intermediate species are not directly measured~\cite{Pantea:2013:QCK,Acker:2014:CDR}. For the glycolytic mechanism considered here, we assume that molecular counts of $A_1$ and $A_2$ are observed at discrete time points, whereas the remaining chemical species are unobserved.  

For our numerical experiment, we consider the full CTMC model of the glycolytic pathway involving $10$ species, \eqref{eq:prop-X} and \eqref{eq:full_gly_unscaled_eq}, with true reaction rates given by 
$\kappa^{(n)}_0=0.5n$, 
$\kappa^{(n)}_1=1$,
$\kappa^{(n)}_{-1}=1.94n$,
$\kappa^{(n)}_2=0.06n$,
$\kappa^{(n)}_3=5$,
$\kappa^{(n)}_{-3}=4.6n$,
$\kappa^{(n)}_4=0.4n$,
$\kappa^{(n)}_5=1$,
$\kappa^{(n)}_{-5}=1.7n$,
$\kappa^{(n)}_6=0.3n$,
$\kappa^{(n)}_7=\sqrt{n}$,
$\kappa^{(n)}_{-7}=1.7321\sqrt{n}$,
$\kappa^{(n)}_{8}=n^{-1}$, and
$\kappa^{(n)}_{-8}=1.7321n^{-1}$.
Time-series data were generated by stochastic simulations of \eqref{eq:full_gly_unscaled_eq} using Gillespie's exact Stochastic Simulation Algorithm~\cite{Gillespie:1977:ESS,Gillespie:2007:SSC}. Five scaling regimes were considered, with the scaling parameters $n$ ranging from small to large values. 

For each scaling regime, a multi-start optimization procedure was applied to minimize the squared difference between the observed data and the ODE trajectories produced using a candidate parameter set. 
The optimization was performed using a modified Nelder-Mead method within a constrained parameter region, implemented in Julia.
The optimization procedure was repeated $m$ times ($m=3000$), each time using a different initial parameter vector as the starting point in order to increase the likelihood of identifying the global solution. Initial parameter values were sampled from prescribed parameter ranges using Latin hypercube sampling, ensuring that the entire parameter space was explored with approximately uniform coverage. 

Table 1 reports the best parameter estimates obtained from the multi-start algorithm for $n=10$, $10^2$, $10^3$, $10^4$, and $10^5$. As expected, the accuracy of the parameter estimates based on the reduced ODE model increases with $n$. Alongside the optimal estimates, we report the associated relative standard deviations (SD), defined as $\sigma_i/\bar{\hat{\theta}}_i$, where $\bar{\hat{\theta}}_i = \f{1}{m'}\sum_{j=1}^{m'} \hat \theta_{j,i}$ denotes the sample mean and $\sigma_i =\f{1}{\sqrt{m'-1}} \sqrt{\sum_{j=1}^{m'} (\hat\theta_{j,i} - \bar{\hat{\theta}}_i )^2}$ denotes the sample standard deviation of the $i$th parameter across the $m'$ optimization runs. An optimization run is classified as converged if it satisfies 
$$ 
\f{|\mathcal{L}-\mathcal{L}_{\text{best}}|}{|\mathcal{L}_{\text{best}}|}\le 0.1, 
$$ 
where $\mathcal{L}_{\text{best}}$ denotes the minimum loss value across all $m$ runs. Here, $m'$ represents the number of converged optimizations among the total $m$ runs. 

The relative errors are computed as
\begin{eqnarray*}
\frac{\|\theta_{\text{true}}-\theta_{\text{estimated}}\|_1}{\|\theta_{\text{true}}\|_1}\times 100\%.
\end{eqnarray*}
For $n=10$, $n=10^2$, and $n=10^3$, the relative errors are $63.9\%$, $31.1\%$, and $37.9\%$, respectively. As $n$ increases to $n=10^4$ and $n=10^5$, the relative errors decrease substantially to $7.2\%$ and $2.8\%$, respectively.
These results indicate that the relative error falls below $10\%$ when the scaling parameter is sufficiently large, specifically for $n=10^4$ and $n=10^5$.

%% constrained
\begin{table}[h!]
\begin{tabular}{|c|c|c|c|c|c|c|c|c|c|}
\hline
Cases & Parameters & $\kappa_0$ & $K_1$ & $K_{M_1}$ & $K^\star_{M_1}$ & $K_{M_2}$ & $\consconst^{\bullet}_1$ & $\consconst^{\star}_1$ & $\consconst^{\bullet}_2$ \\
\hline
\hline
& True values & $0.5$ & $3.0$ & $2.0$ & $1.0$ & $2.0$ & $0.3$ & $2.0$ & $1.5$ \\
& Intervals & $[0.01,1]$ & $[0.01,4]$ & $[0.01,3]$ & $[0.01,2]$ & $[0.01,3]$ & $[0.01,1]$ & $[0.01,3]$ & $[0.01,2]$ \\
\hline
$n=10^1$ & Estimate & $0.73$ & $0.03$ & $1.89$ & $1.15$ & $0.11$ & $0.43$ & $0.88$ & $0.93$ \\
 & Relative SD & $\pm 0.41$ & $\pm 0.83$ & $\pm 0.57$ & $\pm 0.87$ & $\pm 0.58$ & $\pm 0.62$ & $\pm 0.26$ & $\pm 0.18$ \\
\hline
$n=10^2$ & Estimate & $0.70$ & $2.04$ & $2.82$ & $0.30$ & $1.80$ & $0.48$ & $1.58$ & $1.84$ \\
 & Relative SD & $\pm 0.13$ & $\pm 0.37$ & $\pm 0.33$ & $\pm 0.46$ & $\pm 0.32$ & $\pm 0.19$ & $\pm 0.10$ & $\pm 0.13$ \\
\hline
$n=10^3$ & Estimate & $0.77$ & $1.79$ & $0.56$ & $0.69$ & $0.85$ & $0.46$ & $1.90$ & $1.46$ \\
& Relative SD & $\pm 0.16$ & $\pm 0.16$ & $\pm 0.40$ & $\pm 0.17$ & $\pm 0.19$ & $\pm 0.18$ & $\pm 0.06$ & $\pm 0.06$ \\
\hline
$n=10^4$ & Estimate & $0.53$ & $2.89$ & $1.52$ & $0.92$ & $1.84$ & $0.30$ & $1.98$ & $1.50$ \\
& Relative SD & $\pm 0.07$ & $\pm 0.04$ & $\pm 0.49$ & $\pm 0.08$ & $\pm 0.07$ & $\pm 0.06$ & $\pm 0.03$ & $\pm 0.03$ \\
\hline
$n=10^5$ & Estimate & $0.50$ & $2.97$ & $1.80$ & $1.04$ & $2.03$ & $0.29$ & $2.01$ & $1.51$ \\
 & Relative SD & $\pm 0.07$ & $\pm 0.07$ & $\pm 0.29$ & $\pm 0.13$ & $\pm 0.12$ & $\pm 0.07$ & $\pm 0.02$ & $\pm 0.01$ \\
\hline
\end{tabular}
\caption{Parameter estimation of eight parameters using a multi-start optimization algorithm with $3000$ initial points. The estimated values are provided with the best parameter and relative standard deviations of the converged values.}
\end{table}

\setcounter{section}{0}
\setcounter{theorem}{0}
\setcounter{equation}{0}
\renewcommand{\theequation}{\thesection.\arabic{equation}}

\appendix
\section{Appendix}	
  \begin{definition} \label{def:C-tight}
        A collection of stochastic processes $\{\nU :n \ge 1\}$ is said to be $C$-tight in $D([0, T], \R^d)$ if the collection is tight and hence relatively compact  in $D([0, T], \R^d)$, and the limit of every weakly convergent subsequence lies in the space $C([0, T], \R^d)$. 
    \end{definition}

\begin{lemma}\label{lem:int-eq-tight}
Let $\{Y^{(n)}\}$ be a sequence of $\R^d$-valued processes and define the process $V^{(n)}$ by 
     $$V^{(n)}(t)=\int_0^t Y^{(n)}(s) \ ds.$$
Assume that the  $\{\sup_{t\leq T}\|Y^{(n)}(t)\|\}$ is tight in $[0,\infty)$ Then $\{ V^{(n)}\}$ is tight in $C([0,T],\R^d)$. \end{lemma}
\begin{proof} 
Consider the modulus of continuity of $V^{(n)}$ given by
\begin{align*}
   \modu(V^{(n)},T,\delta) \dfeq \sup_{\substack{t_1,t_2\in [0,T], \\ |t_2-t_1|\leq \delta}}\|V^{(n)}(t_1)-V^{(n)}(t_2)\|. 
\end{align*}
Since the sequence $\{\|V^{(n)}(0)\| \equiv 0\}$ is trivially tight, by \cite[Theorem 7.3]{Billingsley1999Convergence}, we need to show that for $\vep>0,$ $ \lim_{\delta \rt 0}  \limsup_{n\rt \infty}\PP\lf(\modu(V^{(n)},T,\delta)\geq \vep\ri)=0.$
%\begin{align}
% \label{eq:modulus-of-cont}
% \lim_{\delta \rt 0}  \limsup_{n\rt \infty}\PP\lf(\modu(V^{(n)},T,\delta)\geq \vep\ri)=0.
%\end{align}
But this is an immediate consequence of the inequality, $\modu(V^{(n)},T,\delta) \leq \sup_{t\leq T}\|Y^{(n)}(t)\|\delta,$ and the assumption of tightness of $\{\sup_{t\leq T}\|Y^{(n)}(t) \|\}$.
%By tightness of $\{\sup_{t\leq T}|U^{(n)}(t)|\}$ in $[0,\infty)$, find $K(\vep)$ such that for all $n>0$, $$\PP\lf(\sup_{t\leq T}|U^{(n)}(t)| > K(\vep)\ri) \leq \vep.$$
%Now clearly, $\modu(V^{(n)},T,\delta) \leq \sup_{t\leq T}|U^{(n)}(t)|\delta.$
%Choose $\delta>0$ small enough such that $\eta\delta^{-1} \geq K(\vep)$. It now follows that for all $n>0$,
%\begin{align*}
%    \PP\lf(\modu(V^{(n)},T,\delta)\geq \eta\ri) \leq&\ \PP\lf(\sup_{t\leq T}|U^{(n)}(t)| \geq \eta\delta^{-1}\ri) \leq  \PP\lf(\sup_{t\leq T}|U^{(n)}(t)| \geq R(\vep)\ri)\\
%    \leq&\ \vep.
%\end{align*}

\end{proof}

\begin{lemma}\label{cor:int-eq-tight}
For each $n\geq 0$, let $U^{(n)}, A^{(n)}$ be $\R^d$-valued c\`adl\`ag  processes and $B^{(n)}$ an $\R^{d\times d}$-valued c\`adl\`ag   process  satisfying
     $$U^{(n)}(t)=A^{(n)}(t)+\int_0^t B^{(n)}(s)U^{(n)}(s) \ ds.$$
Assume that the  $\{\sup_{t\leq T}\|A^{(n)}(t)\|\}$ and $\{\sup_{t\leq T}\|B^{(n)}(t)\|\}$ are tight in $[0,\infty)$ and $\{A^{(n)}\}$ is tight in $D([0,T],\R)$. Then $\{U^{(n)}\}$ is tight in $D([0,T],\R)$. If $\{A^{(n)}\}$ is $C$-tight in $D([0,T],\R)$, then so is $\{U^{(n)}\}$. 
\end{lemma}
\begin{proof}
By Gronwall's inequality, $\{\sup_{t\leq T}\|U^{(n)}(t)\|\}$ is tight, and the assertion follows from \Cref{lem:int-eq-tight} 
\end{proof}

\begin{lemma}
    \label{lem:sup-mart-tight}
 Let $\{\mart^{(n)}\}$ be a sequence of real-valued square integrable martingales such that $\{\<\mart^{(n)}\>_T\}$ is tight in $[0,\infty).$ Then, the sequence of random variables  $\sup_{t\leq T}|\mart^{(n)}(t)|$ is tight in $[0,\infty).$   Moreover, if $\<\mart^{(n)}\>_T \prt 0$  then $\sup_{t\leq T}|\mart^{(n)}(t)| \prt 0$ as $n\rt \infty$.
\end{lemma}

\begin{proof}
    Let $\vep>0$. By the tightness of $\{\<\mart^{(n)}\>_T\}$, choose $K_1(\vep)$ such that 
    $$\sup_n \PP\lf(\<\mart^{(n)}\>_T>K_1(\vep)\ri) \leq \vep/2.$$ Let $K_2(\vep) \equiv \vep^{-1/2} (2K_1(\vep))^{1/2}$. By Lenglart--Rebolledo inequality  \cite[Lemma 3.7]{Whitt2007}, for all $n>0$,
    \begin{align}\label{eq:L-R-ineq}
       \PP\lf(\sup_{t\leq T}|\mart^{(n)}(t)|>K_2(\vep)\ri)\leq&\  K_1(\vep)/K_2^2(\vep)+ \PP\lf(\<\mart^{(n)}\>_T>K_1(\vep)\ri)     \leq \ \vep/2+\vep/2=\vep.
    \end{align}
The second part follows by similar techniques.
\end{proof}

\bibliographystyle{plain}
\bibliography{glycolytic}

\begin{thebibliography}{10}

\bibitem{Anderson:2011:CTM}
David~F. Anderson and Thomas~G. Kurtz.
\newblock Continuous time markov chain models for chemical reaction networks.
\newblock In {\em Design and analysis of biomolecular circuits: engineering
  approaches to systems and synthetic biology}, pages 3--42. Springer, 2011.

\bibitem{Ball:2006:AAM}
Karen Ball, Thomas~G. Kurtz, Lea Popovic, and Greg Rempala.
\newblock Asymptotic analysis of multiscale approximations to reaction
  networks.
\newblock {\em Annals of Applied Probability}, 16(4):1925--1961, 2006.

\bibitem{Billingsley1999Convergence}
Patrick Billingsley.
\newblock {\em Convergence of probability measures}.
\newblock John Wiley \& Sons, 1999.

\bibitem{Crudu2012AAP}
A.~Crudu, A.~Debussche, A.~Muller, and O.~Radulescu.
\newblock Convergence of stochastic gene networks to hybrid piecewise
  deterministic processes.
\newblock {\em The Annals of Applied Probability}, 22(5), 10 2012.

\bibitem{Flach:2006:UAQ}
Edward~H Flach and Santiago Schnell.
\newblock Use and abuse of the quasi-steady-state approximation.
\newblock {\em IEE Proceedings-Systems Biology}, 153(4):187--191, 2006.

\bibitem{foll99}
Gerald~B Folland.
\newblock {\em Real analysis: modern techniques and their applications}.
\newblock John Wiley \& Sons, 1999.

\bibitem{FrWe08}
Mark~I. Freidlin and Alexander~D. Wentzell.
\newblock Some recent results on averaging principle.
\newblock In {\em Topics in stochastic analysis and nonparametric estimation},
  volume 145 of {\em IMA Vol. Math. Appl.}, pages 1--19. Springer, New York,
  2008.

\bibitem{FrWe12}
Mark~I. Freidlin and Alexander~D. Wentzell.
\newblock {\em Random perturbations of dynamical systems}, volume 260 of {\em
  Grundlehren der mathematischen Wissenschaften [Fundamental Principles of
  Mathematical Sciences]}.
\newblock Springer, Heidelberg, third edition, 2012.
\newblock Translated from the 1979 Russian original by Joseph Sz\"ucs.

\bibitem{GaKh26}
Arnab Ganguly and Wasiur~R. KhudaBukhsh.
\newblock Asymptotic analysis of the total quasi-steady state approximation for
  the michaelis–menten enzyme kinetic reactions.
\newblock {\em Journal of Mathematical Analysis and Applications},
  561(1):130551, 2026.

\bibitem{Gillespie:1977:ESS}
Daniel~T. Gillespie.
\newblock Exact stochastic simulation of coupled chemical reactions.
\newblock {\em The journal of physical chemistry}, 81(25):2340--2361, 1977.

\bibitem{Gillespie:2007:SSC}
Daniel~T. Gillespie.
\newblock Stochastic simulation of chemical kinetics.
\newblock {\em Annu. Rev. Phys. Chem.}, 58(1):35--55, 2007.

\bibitem{Goldbeter:2018:DSB}
Albert Goldbeter.
\newblock Dissipative structures in biological systems: bistability,
  oscillations, spatial patterns and waves.
\newblock {\em Philosophical Transactions of the Royal Society A: Mathematical,
  Physical and Engineering Sciences}, 376(2124):20170376, 2018.

\bibitem{Khas68}
R.~Z. Hasminski{\u i}.
\newblock On the principle of averaging the {I}t\^o's stochastic differential
  equations.
\newblock {\em Kybernetika (Prague)}, 4:260--279, 1968.

\bibitem{Higgins:1967:TOR}
Joseph Higgins.
\newblock The theory of oscillating reactions-kinetics symposium.
\newblock {\em Industrial \& Engineering Chemistry}, 59(5):18--62, 1967.

\bibitem{Kang:2013:STM}
Hye-Won Kang and Thomas~G. Kurtz.
\newblock Separation of time-scales and model reduction for stochastic reaction
  networks.
\newblock {\em Annals of Applied Probability}, 23(2):529--583, 2013.

\bibitem{Kim:2020:MMM}
Jae~Kyoung Kim and John~J. Tyson.
\newblock Misuse of the michaelis--menten rate law for protein interaction
  networks and its remedy.
\newblock {\em PLOS Computational Biology}, 16(10):e1008258, 2020.

\bibitem{Kurtz72}
Thomas~G. Kurtz.
\newblock The relationship between stochastic and deterministic models for
  chemical reactions.
\newblock {\em The Journal of Chemical Physics}, 57(7):2976--2978, 1972.

\bibitem{Kurtz:1992:AMP}
Thomas~G. Kurtz.
\newblock Averaging for martingale problems and stochastic approximation.
\newblock In Ioannis Karatzas and Daniel Ocone, editors, {\em Applied
  Stochastic Analysis}, pages 186--209, Berlin, Heidelberg, 1992. Springer
  Berlin Heidelberg.

\bibitem{Novak:2008:DPB}
B{\'e}la Nov{\'a}k and John~J. Tyson.
\newblock Design principles of biochemical oscillators.
\newblock {\em Nature reviews Molecular cell biology}, 9(12):981--991, 2008.

\bibitem{Othmer:1978:ECD}
Hans~G Othmer and John~A Aldridge.
\newblock The effects of cell density and metabolite flux on cellular dynamics.
\newblock {\em Journal of Mathematical Biology}, 5(2):169--200, 1978.

\bibitem{PaVe01}
E.~Pardoux and A.~Yu. Veretennikov.
\newblock On the {P}oisson equation and diffusion approximation. {I}.
\newblock {\em Ann. Probab.}, 29(3):1061--1085, 2001.

\bibitem{PaVe02}
\`E. Pardoux and A.~Yu. Veretennikov.
\newblock On {P}oisson equation and diffusion approximation. {II}.
\newblock {\em Ann. Probab.}, 31(3):1166--1192, 2003.

\bibitem{Segel:1989:QSS}
Lee~A Segel and Marshall Slemrod.
\newblock The quasi-steady-state assumption: a case study in perturbation.
\newblock {\em SIAM review}, 31(3):446--477, 1989.

\bibitem{Selkov:1968:SOG}
Evgeny~Evgenievich SEL'KOV.
\newblock Self-oscillations in glycolysis 1. a simple kinetic model.
\newblock {\em European Journal of Biochemistry}, 4(1):79--86, 1968.

\bibitem{Whitt2007}
Ward Whitt.
\newblock Proofs of the martingale {FCLT}.
\newblock {\em Probability Surveys}, 4:268--302, 2007.

\end{thebibliography}
\end{document}